\documentclass[11pt]{amsart}   	
\usepackage{etex}
\usepackage{geometry}             
\usepackage[dvipsnames]{xcolor}   		             		
\usepackage{graphicx}				
\usepackage{amssymb,mathrsfs}
\usepackage{amsthm,amsmath,stmaryrd}
\usepackage{tikz}
\usepackage{tikz-cd}
\usepackage{accents,upgreek,enumerate}
\usepackage[headings]{fullpage}
\usepackage{bm}
\usepackage{mathtools}
\usepackage[all]{xy}
\usepackage{caption}
\usepackage[euler-digits]{eulervm}
\usepackage[normalem]{ulem}

\tikzset{
  commutative diagrams/.cd, 
  arrow style=tikz, 
  diagrams={>=stealth}
}

\newtheorem{innercustomthm}{Theorem}
\newenvironment{customthm}[1]
  {\renewcommand\theinnercustomthm{#1}\innercustomthm}
  {\endinnercustomthm}

  \makeatletter
\def\@tocline#1#2#3#4#5#6#7{\relax
  \ifnum #1>\c@tocdepth % then omit
  \else
    \par \addpenalty\@secpenalty\addvspace{#2}%
    \begingroup \hyphenpenalty\@M
    \@ifempty{#4}{%
      \@tempdima\csname r@tocindent\number#1\endcsname\relax
    }{%
      \@tempdima#4\relax
    }%
    \parindent\z@ \leftskip#3\relax \advance\leftskip\@tempdima\relax
    \rightskip\@pnumwidth plus4em \parfillskip-\@pnumwidth
    #5\leavevmode\hskip-\@tempdima
      \ifcase #1
       \or\or \hskip 1em \or \hskip 2em \else \hskip 3em \fi%
      #6\nobreak\relax
    \dotfill\hbox to\@pnumwidth{\@tocpagenum{#7}}\par
    \nobreak
    \endgroup
  \fi}
\makeatother

\usetikzlibrary{calc}
\usetikzlibrary{fadings}
\usetikzlibrary{decorations.pathmorphing}
\usetikzlibrary{decorations.pathreplacing}

\newcounter{marginnote}
\DeclareMathAlphabet{\mathpzc}{OT1}{pzc}{m}{it}

\usepackage[backref=page]{hyperref}
\hypersetup{
  colorlinks   = true,          %Colors links instead of ugly boxes
  urlcolor     = blue,          %Color for external hyperlinks
  linkcolor    = violet,          %Color of internal links
  citecolor   = blue             %Color of citations
}

\newtheorem{theorem}{Theorem}[subsection]

\newtheorem{lemma}[theorem]{Lemma}
\newtheorem{proposition}[theorem]{Proposition}

\newtheorem{quasi-theorem}[theorem]{Quasi-Theorem}

\theoremstyle{definition}
\newtheorem{definition}[theorem]{Definition}

\newtheorem{remark}[theorem]{Remark}

\newtheorem{example}[theorem]{Example}
\newtheorem{blank remark}[theorem]{}

\newtheorem{not1}[theorem]{Notation}

\newcommand{\CC} {{\mathbb C}}          
            
\newcommand{\NN} {{\mathbb N}}		
\newcommand{\PP}{\mathbb{P}}         
\newcommand{\QQ} {{\mathbb Q}}		
\newcommand{\RR} {{\mathbb R}}		
\newcommand{\ZZ} {{\mathbb Z}}		
\newcommand{\sfi}{\circ}
\newcommand{\sfbu}{{{\color{black}\circ}}}	

\def\setminus{\smallsetminus}

\DeclareMathOperator{\Aut}{Aut}

\DeclareMathOperator{\val}{val}
\DeclareMathOperator{\mult}{mult}

\DeclareMathOperator{\sign}{sign} 

\DeclareMathOperator{\ord}{ord}

\DeclareMathOperator{\spec}{Spec}

\newcommand{\cal}{\mathcal}

\def\cM{{\cal M}}

\def\cO{{\cal O}}

\newcommand{\Mbar}{\overline{\cM}}

\newcommand{\TSM}{\mathsf{DR}_{g}^{\trop}(\mathbf x,0)}
\newcommand{\TMC}{\mathsf{DR}^{\sfi, \trop}_{g}( \mathbf{x},0)}
\newcommand{\TMB}{\mathcal{M}^{\trop}_{g} (\mathbf{x}, 0)}
\newcommand{\MB}{\mathcal{{M}}_{g} (\mathbf{x},0)}  
\newcommand{\TC}{\mathsf{DR}_{g}^{\sfi,\lightening}(\mathbf{x},0)}
\newcommand{\LDR}{\mathsf{DR}_{g}^\lightening(\mathbf{x},0)}
\newcommand{\RSM}{{\mathcal{M}^{\trop,\sim}_{g}(\PP^1, \mathbf{x})}}
\def\trop{\mathrm{trop}}

\newcommand{\Spec}{\operatorname{Spec}}

\newcommand{\lightening}{{\mathsf{log}}}
\newcommand{\leak}{\mathsf{DR}_g^{\trop}(\mathbf x,k)}
\newcommand{\leakyn}{L_{g}(\mathbf{x},k)}

\newcommand{\renzo}[1]{}
\newcommand{\rcolor}[1]{{ \color{black}{#1}}}

\makeatletter
\def\blfootnote{\xdef\@thefnmark{}\@footnotetext}
\makeatother

\title{Pluricanonical cycles and tropical covers}

\date{}

\author{Renzo Cavalieri, Hannah Markwig, {\it \&} Dhruv Ranganathan}

\begin{document}

\begin{abstract}
We extract a system of numerical invariants from logarithmic intersection theory on pluricanonical double ramification cycles, and show that these invariants exhibit a number of properties that are enjoyed by double Hurwitz numbers. Among their properties are (i) the numbers can be efficiently calculated by counts of tropical curves with a modified balancing condition, (ii) they are piecewise polynomial in the entries of the ramification vector, and (iii) they are matrix elements of operators on the Fock space. The numbers are extracted from the logarithmic double ramification cycle, which is a lift of the standard double ramification cycle to a blowup of the moduli space of curves. The blowup is determined by tropical geometry. We show that the traditional double Hurwitz numbers are intersections of the refined cycle with the cohomology class of a piecewise polynomial function on the tropical moduli space of curves. This perspective then admits a natural, combinatorially motivated, generalization to the pluricanonical setting. Tropical correspondence results for the new invariants lead immediately to the structural results for these numbers. 
\end{abstract}
\maketitle

%\eject
\setcounter{tocdepth}{1}
\tableofcontents

\section{Introduction}

Double Hurwitz covers are maps 
\[
(C,p_1,\ldots, p_n)\to (\PP^1,0,\infty)
\]
from a smooth projective genus $g$ curve to $\mathbb P^1$. The ramification data is fixed by a length $n$ vector $\mathbf x$ with vanishing sum, whose entries record the orders of zeroes and poles at the corresponding marked point. We fix the ramification vector $\bf x$. There are two basic directions of interest in double Hurwitz theory. 

\begin{enumerate}[(i)]
\item {\bf Numerical.}  After fixing the locations of the branch points, there are finitely many covers of the form above. The count of such covers is the \textit{double Hurwitz number} $H_g(\mathbf x)$. 
\item {\bf Cohomological.} A natural compactification of the locus in $\cM_{g,n}$ of points that admit a double Hurwitz cover with ramification $\bf x$ produces tautological cohomology class on $\Mbar_{g,n}$ known as the \textit{double ramification cycle}. 
\end{enumerate}

Double Hurwitz numbers have a remarkable structure. They form basic building blocks in the Gromov--Witten theory of curves, can be computed via tropical correspondence theorems, exhibit a piecewise polynomiality property when varying $\bf x$, and can be connected to the representation theory of the Heisenberg algebra~\cite{CJM1,GJV,OP06}. The numbers and their structure remain of significant contemporary interest, for example in the context of topological recursion~\cite{BDKLM,DL}.

The double ramification cycle is a compactification of the locus in $\cM_{g,n}$ parameterizing double Hurwitz covers, and arises from the virtual geometry of the space of relative stable maps to $\mathbb P^1$. Developments in logarithmic and orbifold geometry have led to a beautiful explicit formula for the cycle in terms of standard classes~\cite{BHPSS,JPPZ,JPPZ2}. The double ramification cycle has natural generalizations coming from the geometry of the Picard variety, and of special interest are the \textit{pluricanonical double ramification cycles}, where one studies curves admitting pluricanonical divisors of a form dictated by $\bf x$, rather than principal divisors~\cite{BHPSS,FP18}. 

\subsection{Main results} We propose a pluricanonical generalization of the numerical side of double Hurwitz theory, and show that they share many of the properties of double Hurwitz numbers. The approach is based on a new connection, presented here, that extracts the numerical theory from the cohomological theory, which then naturally generalizes to the pluricanonical situation. Precisely, we observe that the double Hurwitz numbers can be extracted from the \textit{logarithmic} double ramification cycle by intersection with the cohomology class associated to a piecewise polynomial function on the tropical moduli space. These logarithmic enhancements were first considered in~\cite{HPS19} and studied further in~\cite{HMPPS,HS21,MR21}.The logarithmic context is critical; there appears to be no simple way to extract the Hurwitz numbers from the traditional double ramification cycle calculated in~\cite{JPPZ}. 

Once the double Hurwitz numbers have been expressed as above, the double ramification cycle can be replaced with its pluricanonical variants, with the intersection number coming from the same piecewise polynomial function as before. We show that these intersection numbers can be calculated by  tropical geometry, and import statements from double Hurwitz theory into the pluricanonical context via their tropical interpretations. 

Given a stack $\mathsf X$ with a normal crossings boundary divisor, a \textit{simple toroidal blowup} is a blowup $\mathsf X'\to \mathsf X$ along a smooth stratum; the preimage of the boundary is again normal crossings. A \textit{toroidal blowup} is a morphism $\mathsf Y\to \mathsf X$ obtained by a sequence of simple toroidal blowups. The \textit{logarithmic Chow ring of $\mathsf X$} is the colimit of the Chow rings $\mathsf{CH}^\star(\mathsf X')$ under pullback. A fruitful source of logarithmic cohomology classes in this Chow ring are piecewise polynomial functions on the cone complex of $\mathsf X$; this is explained in Section~\ref{sec:tcpp}, following~\cite{HS21,MPS21,MR21}.

Given an integer $k$ and a vector $\mathbf x$ in $\mathbb Z^n$, the \textit{logarithmic $k$-pluricanonical double ramification cycle} $\mathsf{DR}_g^{\mathsf{log}}(\mathbf x,k)$ is defined in the logarithmic Chow ring of $\Mbar_{g,n}$. When restricted to the interior $\cM_{g,n}$ the class is represented by the locus of pointed curves $(C,p_1,\ldots, p_n)$ such that the divisor $\sum x_ip_i$ is isomorphic to the $k^{th}$ power of the logarithmic canonical bundle of $C$. 

Fix a genus $g$ and ramification vector $\bf x$. In Section~\ref{sec:dhn} we construct an explicit piecewise polynomial function on the tropical moduli space of curves, called the \textit{branch polynomial}. It is a function
\[
br(\mathbf x,0): |\cM^\trop_{g,n}|\to \RR,
\]
and upon restriction to the locus of tropical curves admitting a balanced map to $\RR$ with ramification $\mathbf x$, it is constructed from the data of the images in $\mathbb R$ of the\renzo{I would actually propose to just remove trivalent and genus $0$ here because this is only the case on maximal dimenional cones. For lower dimensional cones we are using the position of vertices (as well as valence and genus) to describe br } %trivalent genus $0$ 
vertices of the tropical curve. We take $br(\mathbf x,0)$ to be any extension of this polynomial to the whole moduli space. 

\begin{customthm}{A}
The intersection number of $\mathsf{DR}_g^{\mathsf{log}}(\mathbf x,0)$ with the branch polynomial is equal to the double Hurwitz number:
\[
H_g(\mathbf x) = \mathrm{deg}\left([br(\mathbf x,0)]\cap \mathsf{DR}_g^{\mathsf{log}}(\mathbf x,0)\right).
\]
The branch polynomial may be explicitly represented by a cycle consisting of a collection of strata in a toroidal blowup of $\Mbar_{g,n}$.
\end{customthm}

The result serves as motivation for us to \textit{define} the $k$-pluricanonical double Hurwitz number with ramification vector $\bf x$ as the intersection number of $\mathsf{DR}_g^{\mathsf{log}}(\mathbf x,k)$ with the branch polynomial $br_g(\mathbf x, k)$\renzo{Here we have a subscript $g$ that is not present in the br-0. Should we make the notation uniform?} which has essentially the same definition, see Section~\ref{sec:lpn}. The pluricanonical numbers exhibit remarkable structural properties. They can be effectively computed by counts of \textit{leaky tropical covers}, analogous to tropical correspondence theorems~\cite{BBM,CJM1}. A sum-over-graphs formula is presented in Theorem~\ref{thm: leaky-formula}. 

\begin{customthm}{B}
The $k$-pluricanonical double Hurwitz number is equal to a weighted count of $k$-leaky tropical covers. Moreover, the counts are piecewise polynomial in the ramification data $\mathbf x$. 
\end{customthm}

The numbers also fit well into the operator formalism governing double Hurwitz numbers. We define an operator on the Bosonic Fock space in Section~\ref{sec:fs} whose matrix elements recover the $k$-pluricanonical double Hurwitz numbers.

The generalization of the tropical double Hurwitz numbers to the pluricanonical context is combinatorially natural, and this was our motivation in examining it. Tropical double Hurwitz covers are continuous piecewise linear maps from tropical curves to $\RR$ that satisfy the balancing condition -- the sum of the directional derivatives of the function at each vertex is $0$. The pluricanonical variant allows a measured failure of the balancing condition. Equivalently, the numbers are a weighted count of rational functions in the tropical pluricanonical series.  Despite structural similarities, we are not aware of a simple enumerative definition of the pluricanonical Hurwitz numbers. 

It is in general an interesting problem to give compact expressions for the intersections of logarithmic double ramification cycles with piecewise polynomial classes, and tropical geometry provides one tool for finding such expressions. This point is illustrated by the following result. In Section~\ref{sec:lg} we construct the \textit{branch polynomial at level $h$} for any $h$ between $0$ and $g$ as a piecewise polynomial class. We show that the \textit{lower genus} double Hurwitz numbers can be extracted from higher genus cycles. 

\begin{customthm}{C}
For each integer $h$ between $0$ and $g$, the intersection number of $\mathsf{DR}_g^{\mathsf{log}}(\mathbf x,0)$ with the branch polynomial at level $h$ is equal to the genus $h$ double Hurwitz number with ramification $\bf x$.\renzo{Are we hiding some factors of $1/24$ here or are they folded into the definition of the branch poly at level $h$ - check.} 
\end{customthm}

The research community is aware that the virtual class of stable map spaces in genus $g$ contains information about the lower genus curves. However, it is difficult to extract this information numerically. We believe the results, taken together, give a concrete instance of the type of information carried by the logarithmic double ramification cycle that is not carried by its less refined counterpart. 

\subsection{Further discussion and future directions} A geometric motivation to understand the class $\mathsf{DR}_g^{\mathsf{log}}(\mathbf x,0)$ arises from the product formula in Gromov--Witten theory~\cite{Herr,HPS19,R19b}. For example, it can be used to completely solve the logarithmic Gromov--Witten invariants of toric varieties~\cite{RUK22}. The cycle class of curves admitting a map to a toric variety with prescribed contact orders with its toric boundary can be recovered from products of double ramification cycles in the logarithmic tautological ring~\cite{MR21,R19b}. In turn, these \textit{toric contact cycles} in the moduli space of curves, and natural variants for torus bundles over manifolds, form the basic building block in the Gromov--Witten theory of simple normal crossings pairs. Precisely, in logarithmic Gromov--Witten theory with expansions, boundary strata naturally parameterize maps to toric varieties and toric bundles; these loci contribute to Gromov--Witten theory via the toric contact cycle~\cite{MR21,R19}.

The logarithmic double ramification cycle remains mysterious, though there has been recent progress. It has been shown that the class $\mathsf{DR}^\lightening(\mathbf x,0)$ lies in an appropriately defined tautological ring on the blowup~\cite{HS21,MR21}. The main result of the latter of these papers describes the class $\mathsf{DR}^\lightening(\mathbf x,0)$ as a \textit{virtual strict transform} of the well-studied double ramification cycle on $\Mbar_{g,n}$. The operation is analogous to the standard strict transform, however, birational modification under which this operation must be calculated has a high degree of combinatorial complexity. A Pixton-type formula for the cycle has been established in recent work~\cite{HMPPS} by applying universal Abel--Jacobi theory to compactified Picard stacks~\cite{BHPSS}, though the formula is complicated. 

The results here motivate a different approach, via generalized tropical correspondence theorems. For arbitrary values of $k$, we take the view that $\mathsf{DR}^\lightening(\mathbf x,k)$ determines an operator on the logarithmic Chow ring of the moduli space of curves
\[
\mathsf{dr}^\lightning(\mathbf x,k): \mathsf{CH}_{\mathsf{log}}^\star(\Mbar_{g,n})\to \QQ
\] 
by intersecting with $\mathsf{DR}^\lightening(\mathbf x,k)$ and taking degree. While the operator does not determine the cycle $\mathsf{DR}^\lightening(\mathbf x,k)$, it captures an aspect of the class that is highly relevant to enumerative geometry. The correspondence results suggest that efficient formulas may exist for intersections of the logarithmic double ramification cycle with special classes. Two of particular interest are: 

\begin{enumerate}[(i)]
\item The descendant classes $\psi_i$ associated to the marked points, pulled back from $\Mbar_{g,n}$.
\item The piecewise polynomials on subdivisions of the tropical moduli space of curves $\cM_{g,n}^{\trop}$.
\end{enumerate}

The first is completely standard. The second are the natural generalizations of decorated strata classes in $\Mbar_{g,n}$ to these blowups~\cite{HS21,MPS21,MR21}. The restriction of the operator above to descendant classes is known in terms of standard classes by a combination of Pixton's formula and the multiplication rules for descendants. There is remarkable hidden structure that becomes visible in special cases of $\mathsf{dr}^\lightning(\mathbf x,k)$:
\begin{enumerate}[(i)]
\item If $\mathbf x$ is zero, the restriction of $\mathsf{dr}^\lightning(\mathbf x,k)$ to classes of the form (i) is determined by the $\lambda_g$-conjecture, see work of Faber--Pandharipande~\cite{FP03}.
\item If $k$ is zero and $\mathbf x$ has no zero entries, there is a remarkable formula for $\mathsf{dr}^\lightning(\mathbf x,0)$ on the descendants, see work of Buryak--Shadrin--Spitz--Zvonkine~\cite{BSSZ}. A closely related direction for other values of $k$ is found in~\cite{CSS21}. 
\end{enumerate}

We expect that it will be natural to marry our calculation via tropical curve counts with the descendant integrals in the second listed result above, to graphically control the operator $\mathsf{dr}^\lightning(\mathbf x,0)$ on the full space of piecewise polynomial classes. 

A further direction of study would be to examine the higher dimensional variant of the problem here, i.e. to study tropical curves in $\mathbb R^n$ with a pluricanonically modified balancing condition. It is natural to expect a tropical correspondence theorem relating these to intersection numbers on toric contact cycles of rank $n$, following~\cite{MR21} and generalizing Mikhalkin's correspondence theorem~\cite{Mi03}. We leave this for future work. 

\subsection*{Corrections from the published version} The present version of this article corrects an error in the previous version, and in the published version. Specifically, a set of combinatorial types of leaky tropical covers -- those whose domains have genus one leaves -- were incorrectly ignored. These graphs should be counted, with a vertex multiplicity explained here in Section~\ref{sec: leaky-numbers-tropically}. Accounting for these graphs, the proof is essentially unchanged. 

\subsection*{Acknowledgements} We are grateful to friends and colleagues for numerous conversations over the years related to correspondence theorems,  double Hurwitz numbers, and double ramification cycles. In particular, we wish to thank L. Battistella, N. Chidambaram, D. Holmes, V. K\"orber, S. Molcho, S. Shadrin, J. Schmitt, N. Nabijou, P. Johnson, R. Pandharipande, and J. Wise. We extend a special thanks to J. Schmitt and S. Molcho for helping us catch a missing combinatorial factor in an earlier version. We also thank the anonymous referees for numerous useful suggestions on an earlier version. As noted above gap was found after work of D. Accadia, M. Karev and D. Lewanski with particular support by J. Schmitt. We thank these four for the careful work with our paper and for helping us to fill this gap. Special thanks to J. Schmitt and admcycles for confirming an example computation.

\subsection*{Funding} R.C. has received support from Simons Collaboration Grant 420720 and  NSF grant DMS - 2100962.  H.M. has received support from the SFB-TRR 195 ``Symbolic Tools in Mathematics and their Application'' of the German Research Foundation (DFG). D.R. has received support from EPSRC New Investigator Award EP/V051830/1.

\section{Transverse cycles and piecewise polynomials} \label{sec:tcpp}

The study of double ramification cycles and their pluricanonical generalizations has received substantial interest in recent years, beginning in~\cite{FP18,JPPZ} and culminating in~\cite{BHPSS}. We refer the reader to the references in the latter paper for a full bibliography on the topic. These cycles have typically been constructed on the moduli space of stable curves, stable maps, and prestable curves. We work here with a more refined object, the logarithmic double ramification cycles and its pluricanonical versions~\cite{Hol17,HPS19,HS21,MW17,MR21}. We recall the construction here, and place it in the necessary context. 

\subsection{Nonsingular curves} \label{sec:nonsingc}The moduli space $\cM_{g,n}$ of smooth $n$-marked genus $g$ curves is equipped with a universal curve. The universal Picard group is a family
\[
\mathbf{Pic}_{g,n}\to \cM_{g,n}
\]
of relative dimension $g$, whose fiber over a marked curve is its Picard variety. The family admits tautological sections of two flavours. The first sends a curve $C$ to a power of the logarithmic canonical bundle $\omega_C^{\mathsf{log},\otimes k}$ for some integer $k$. The second depends on a tuple $\mathbf x = (x_1,\ldots, x_n)$ of integers, and sends a marked curve $(C,\underline p)$ to the line bundle $\mathcal O_C(\mathbf x \cdot\underline p)$. We denote this section by
\[
\sigma_{\mathbf x}:\cM_{g,n}\to \mathbf{Pic}_{g,n}.
\]
The map is a section of a smooth fibration, and in particular, it is a regular embedding. 

\begin{definition}
Given a tuple of integers $\mathbf x$ whose sum is $k(2g-2+n)$, the \textit{$k$-pluricanonical cycle with ramification $\mathbf x$} is defined to be the Gysin pullback 
\[
\mathsf{DR}_g^\sfbu(\mathbf{x},k):=\sigma_{\mathbf x}^![\omega^k]
\]
where $[\omega^k]$ is the image of the logarithmic $k$-pluricanonical section of $\mathbf{Pic}_{g,n}\to \cM_{g,n}$.
\end{definition}

The cycle may be extended to the compact moduli space of curves $\Mbar_{g,n}$ either by Gromov--Witten theory or by extending the Abel--Jacobi map using birational modifications~\cite{Hol17,MW17}. In fact, the Holmes--Molcho--Wise construction produces something more refined, which is a class on a blowup of $\Mbar_{g,n}$, as we shortly explain. 

\subsection{Tropical covers and their moduli} We recall that there is moduli stack $\cM_{g,n}^{\trop}$ of tropical curves of genus $g$ with $n$-marked points, constructed in~\cite{CCUW}. Formally, it is stack over the category of rational polyhedral cone complexes. Its fiber over a cone $\sigma$ is the groupoid of all tropical curves metrized by elements in the dual monoid $S_\sigma$. We refer to~\cite[Section~3]{CCUW} for further details.\footnote{In fact, for our purposes, the older perspective employed by Abramovich--Caporaso--Payne suffices~\cite{ACP}.}

As we momentarily explain, the tropical analogue of the discussion in Section \ref{sec:nonsingc} gives rise to ``tropical substacks''
\[
\mathsf{DR}_g^{\trop}(\mathbf x,k)\hookrightarrow \cM_{g,n}^{\trop}
\]
of tropical analogues of pluricanonical sections. We introduce the relevant notions. 

To each $n$-pointed curve $(C,p_1,\ldots,p_n)$ one may associate its  {\em weighted dual graph}: 
$$\mathsf{G} =\ (V, E, L, h),$$

\begin{enumerate}[(i)]
\item the set $V$ of vertices is taken to be the set of irreducible components of $C$;
\item the edge set  $E$ is the set of nodes of $C$, with the natural incidence relation between components;
\item the { ordered} set of { ends} of  $L$ labelled $\{1,\ldots, n\}$ is the set of marked points on $C$,  with a marked end incident to a vertex if the corresponding point lies on the component corresponding to the vertex;
\item the genus function $g: V \to \ZZ_{\geq 0}$ gives the geometric genus $g(v)$ of the component dual to $v$.
\end{enumerate}

The \textit{genus} of a weighted dual graph is the sum of the values of the genus function at all vertices, plus the first Betti number of the geometric realization of $\mathsf G$. 

\begin{definition}
An \textit{abstract tropical curve} is a pair $\Gamma = (\mathsf G,\ell)$ where $\mathsf G$ is a weighted dual graph as above and 
\[
\ell: E\to \RR_{>0}
\]
is a length function.  
\end{definition}

The function $\ell$ enhances the topological realization of the graph to a metric space. We enhance it further by attaching, for each end, a copy of $\RR_{\geq 0}$ to this metric where $0$ is identified with the vertex which supports the end. It will often be useful to consider piecewise linear functions on the metric realization of an abstract tropical curve.

We want to examine covers of $\RR$ by graphs up to additive translation,  and  equip  $\RR$ with a polyhedral subdivision to ensure the result is a map of metric graphs (see e.g.\ Section 5.4 and \ Figure 3 in \cite{MW17}). A \textit{metric line graph} is any metric graph obtained from a polyhedral subdivision of $\RR$. The metric line graph determines the polyhedral subdivision up to translation. We fix an orientation of a metric line graph going from left to right (i.e.\ from negative values in $\RR$ to positive values).

% {\color{blue} Dhruv: the way this presently set up, we only treat canonical divisor as opposed to pluricanonical divisors. I don't think there's any reason for the restriction, but I'm happy saying: ``we treat the canonical case to avoid $k$'s floating around''.}

\begin{definition}[Leaky cover]
Let $\pi:\Gamma\rightarrow T$ be a surjective map of metric graphs where $T$ is a metric line graph. (A map of metric graphs is a continuous function such that the preimage of a vertex contains only vertices.) We require that $\pi$ is piecewise integer affine linear, i.e.\ on each edge (which we identify with a real interval of suitable length) it is of the form $t\mapsto a+t\cdot \omega(e)$ with $a \in T, \omega(e)\in \NN_{> 0}$. The $\omega(e)$ of $\pi$ on a flag or edge $e$ is called the \emph{expansion factor}.
(We always measure the slope in positive orientation. We do not need to consider contracted edges for our purposes.)

For a vertex $v\in \Gamma$, the \emph{left (resp.\ right) degree of $\pi$ at $v$} is defined as follows. Let $f_l$ be the flag of $\pi(v)$ in $T$ pointing to the left and $f_r$ the flag pointing to the right. Add the expansion factors of all flags $f$ adjacent to $v$ that map to $f_l$ (resp.\ $f_r$):
\begin{equation}
d_v^l=\sum_{f\mapsto f_l} \omega(f), \;\;\;\; d_v^r=\sum_{f\mapsto f_r} \omega(f).
\end{equation} 

\noindent
The map $\pi:\Gamma\rightarrow T$ is called a \emph{$k$-leaky cover} if for every $v\in \Gamma$
$$d_v^l-d_v^r= k(2g(v)-2+\val(v)).$$
\end{definition}

%To avoid a mess of notation, we describe carefully the case $k=1$, and refer to $1$-leaky covers simply as leaky covers. The modifications required to treat the case $k\not=1$ are straightforward.

\begin{remark}[Vertex set]
From now on, we impose a stability condition: A leaky cover $\Gamma\to T$ is called stable if the preimage of every vertex of $T$ contains a vertex of $\Gamma$ in its preimage which is of genus greater than $0$ or valence greater than $2$. Figure \ref{fig-leakycoverex} shows an example of a stable leaky cover.

\begin{figure}

\tikzset{every picture/.style={line width=0.75pt}} %set default line width to 0.75pt        

\begin{tikzpicture}[x=0.75pt,y=0.75pt,yscale=-1,xscale=1]

%Straight Lines [id:da5958354051302961] 
\draw    (209,219.98) -- (285.42,219.98) ;
%Curve Lines [id:da25275447191433165] 
\draw    (285.42,219.98) .. controls (297.68,210.09) and (309.95,206) .. (331.65,205.32) ;
%Curve Lines [id:da9914654939629761] 
\draw    (285.42,219.98) .. controls (308.06,225.77) and (348.63,229.86) .. (385.42,218.27) ;
%Straight Lines [id:da5480064285624986] 
\draw    (385.42,218.27) -- (417.5,218.95) ;
%Straight Lines [id:da18652718729718287] 
\draw    (331.65,205.32) -- (415.61,204.64) ;
%Straight Lines [id:da27102601636794255] 
\draw    (331.65,205.32) -- (385.42,218.27) ;
%Shape: Circle [id:dp5346476536546783] 
\draw  [fill={rgb, 255:red, 0; green, 0; blue, 0 }  ,fill opacity=1 ] (383.67,204.92) .. controls (383.67,203.95) and (384.46,203.17) .. (385.42,203.17) .. controls (386.39,203.17) and (387.17,203.95) .. (387.17,204.92) .. controls (387.17,205.88) and (386.39,206.67) .. (385.42,206.67) .. controls (384.46,206.67) and (383.67,205.88) .. (383.67,204.92) -- cycle ;
%Shape: Circle [id:dp3309498386673603] 
\draw  [fill={rgb, 255:red, 0; green, 0; blue, 0 }  ,fill opacity=1 ] (283.67,219.98) .. controls (283.67,219.01) and (284.45,218.23) .. (285.42,218.23) .. controls (286.39,218.23) and (287.17,219.01) .. (287.17,219.98) .. controls (287.17,220.94) and (286.39,221.73) .. (285.42,221.73) .. controls (284.45,221.73) and (283.67,220.94) .. (283.67,219.98) -- cycle ;
%Shape: Circle [id:dp2811213581720141] 
\draw  [fill={rgb, 255:red, 0; green, 0; blue, 0 }  ,fill opacity=1 ] (329.9,205.32) .. controls (329.9,204.35) and (330.68,203.57) .. (331.65,203.57) .. controls (332.61,203.57) and (333.4,204.35) .. (333.4,205.32) .. controls (333.4,206.28) and (332.61,207.07) .. (331.65,207.07) .. controls (330.68,207.07) and (329.9,206.28) .. (329.9,205.32) -- cycle ;
%Shape: Circle [id:dp733182153578922] 
\draw  [fill={rgb, 255:red, 0; green, 0; blue, 0 }  ,fill opacity=1 ] (329.9,225.92) .. controls (329.9,224.95) and (330.68,224.17) .. (331.65,224.17) .. controls (332.61,224.17) and (333.4,224.95) .. (333.4,225.92) .. controls (333.4,226.88) and (332.61,227.67) .. (331.65,227.67) .. controls (330.68,227.67) and (329.9,226.88) .. (329.9,225.92) -- cycle ;
%Shape: Circle [id:dp09515173918739561] 
\draw  [fill={rgb, 255:red, 0; green, 0; blue, 0 }  ,fill opacity=1 ] (383.67,217.58) .. controls (383.67,216.62) and (384.46,215.83) .. (385.42,215.83) .. controls (386.39,215.83) and (387.17,216.62) .. (387.17,217.58) .. controls (387.17,218.55) and (386.39,219.33) .. (385.42,219.33) .. controls (384.46,219.33) and (383.67,218.55) .. (383.67,217.58) -- cycle ;
%Straight Lines [id:da9599609425905865] 
\draw    (319,250.75) -- (319,277.75) ;
\draw [shift={(319,279.75)}, rotate = 270] [color={rgb, 255:red, 0; green, 0; blue, 0 }  ][line width=0.75]    (10.93,-3.29) .. controls (6.95,-1.4) and (3.31,-0.3) .. (0,0) .. controls (3.31,0.3) and (6.95,1.4) .. (10.93,3.29)   ;
%Straight Lines [id:da21828343760028213] 
\draw    (210,288.25) -- (414.5,288.75) ;
%Shape: Circle [id:dp058794972583231964] 
\draw  [fill={rgb, 255:red, 0; green, 0; blue, 0 }  ,fill opacity=1 ] (283.67,288) .. controls (283.67,287.03) and (284.45,286.25) .. (285.42,286.25) .. controls (286.39,286.25) and (287.17,287.03) .. (287.17,288) .. controls (287.17,288.97) and (286.39,289.75) .. (285.42,289.75) .. controls (284.45,289.75) and (283.67,288.97) .. (283.67,288) -- cycle ;
%Shape: Circle [id:dp3435708183391648] 
\draw  [fill={rgb, 255:red, 0; green, 0; blue, 0 }  ,fill opacity=1 ] (329.9,288.5) .. controls (329.9,287.53) and (330.68,286.75) .. (331.65,286.75) .. controls (332.61,286.75) and (333.4,287.53) .. (333.4,288.5) .. controls (333.4,289.47) and (332.61,290.25) .. (331.65,290.25) .. controls (330.68,290.25) and (329.9,289.47) .. (329.9,288.5) -- cycle ;
%Shape: Circle [id:dp2017160615516288] 
\draw  [fill={rgb, 255:red, 0; green, 0; blue, 0 }  ,fill opacity=1 ] (383.67,288.25) .. controls (383.67,287.28) and (384.46,286.5) .. (385.42,286.5) .. controls (386.39,286.5) and (387.17,287.28) .. (387.17,288.25) .. controls (387.17,289.22) and (386.39,290) .. (385.42,290) .. controls (384.46,290) and (383.67,289.22) .. (383.67,288.25) -- cycle ;

% Text Node
\draw (233.25,205.11) node [anchor=north west][inner sep=0.75pt]  [font=\footnotesize] [align=left] {5};
% Text Node
\draw (290.8,199.66) node [anchor=north west][inner sep=0.75pt]  [font=\footnotesize] [align=left] {3};
% Text Node
\draw (322.87,230.34) node [anchor=north west][inner sep=0.75pt]  [font=\footnotesize] [align=left] {1};
% Text Node
\draw (364.77,191.33) node [anchor=north west][inner sep=0.75pt]  [font=\footnotesize] [align=left] {1};
\draw (400.77,191.33) node [anchor=north west][inner sep=0.75pt]  [font=\footnotesize] [align=left] {1};
% Text Node
\draw (400.24,224.2) node [anchor=north west][inner sep=0.75pt]  [font=\footnotesize] [align=left] {1};
% Text Node
\draw (339.86,211.59) node [anchor=north west][inner sep=0.75pt]  [font=\footnotesize] [align=left] {1};

\end{tikzpicture}

\caption{A $1$-leaky cover of genus $1$ and degree $(5,-1,-1)$, with its minimal vertex set. We do not specify length data in this picture, as the lengths in $\Gamma$ are imposed by the distances of the points in $T$. All vertices are supposed to be of genus $0$. For simplicity, we also suppress the labels for the ends in this picture.}\label{fig-leakycoverex}
\end{figure}
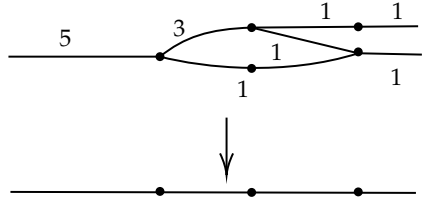

\end{remark}

\begin{definition}[Left and right degree]
The \emph{left (resp.\ right) degree}  of a leaky cover is the tuple of expansion factors of its ends mapping asymptotically to $-\infty$ (resp.\ $+\infty$). 
The tuple is indexed by the lables of the ends mapping to  $-\infty$ (resp.\ $+\infty$).  When the order imposed by the labels of the ends plays no role, we drop the information and treat the left and right degree only as a multiset. 
\end{definition}
By convention, we denote the left degree by $\mathbf{x}^{+}$ and the right degree by $\mathbf{x}^{-}$. In the right degree, we use negative signs for the expansion factors, in the left degree positive signs. We also merge the two to one vector which we denote $\mathbf{x}=(x_1,\ldots,x_n)$  called the \emph{degree}\footnote{The term \textit{degree} here is consistent with the tropical geometry literature, where it is sometimes called the ``toric degree''. In logarithmic geometry $\mathbf x$ would be called the contact data.}. The labeling of the ends plays a role: the expansion factor of the end with the label $i$ is $x_i$. In $\mathbf{x}$, we  distinguish the expansion factors of the left ends from those of the right ends by their sign.
An Euler characteristic calculation, combined with the $k$-leaky cover condition, shows that 
$$ \sum_{i=1}^n x_i=k\cdot (2g-2+n), $$
where $g$ denotes the genus of $\Gamma$.

An automorphism of a leaky cover is an automorphism of $\Gamma$ compatible with $\pi$, i.e.\ it is an automorphism lying over the identity on $T$.

If we view the expansion factors of a $1$-leaky tropical cover as slopes of a rational function on $\Gamma$, then the divisor of this rational function is the canonical divisor of $\Gamma$. For readers not familiar with divisors on abstract tropical curves, we repeat the most important notions in the following. More details on divisors on abstract tropical curves can be found e.g.\ in \cite{BJ, HMY10}.

\begin{definition}[Divisors and the tropical canonical]
Let $\Gamma$ be an abstract tropical curve. A \emph{divisor} $D$ on $\Gamma$ is a formal finite $\mathbb{Z}$-linear combination $\sum_{v\in \Gamma} D(v) \cdot v$ of points on $\Gamma$. Let $\Gamma$ be an abstract tropical curve. The \emph{canonical divisor} on $\Gamma$ is given by $\sum_{v\in \Gamma} (\val(v)-2+2g(v)) \cdot v$.
\end{definition}

\begin{definition}[Rational functions on abstract tropical curves and their divisors]
Let $\Gamma$ be an abstract tropical curve. A \emph{rational function} $f$ on $\Gamma$ is a continuous function $f:\Gamma\rightarrow \RR$ which is piecewise linear with finitely many distinct linear pieces and integer slopes. The \emph{order} $\ord_v(f)$ of $f$ at a point $v$ is the sum of the outgoing slopes. The \emph{divisor of a rational function} $f$ is defined as $$(f):=\sum_{v\in \Gamma} \ord_v(f) \cdot v.$$
\end{definition}

Given a $1$-leaky tropical cover $\pi:\Gamma\rightarrow T$, we view the expansion factors on the edges as slopes of a rational function $f$ (up to global shift). Then, by definition, the divisor $(f)$ equals the canonical divisor of $\Gamma$.

\begin{definition}
The \emph{combinatorial type} of a leaky cover $\Gamma\to T$ is the data obtained on dropping the edge length function. That is, we keep the information of the abstract graph underlying $\Gamma$ including the unbounded ends, the genus function, the underlying graph of the line graph $T$, the graph theoretic map between them, and the expansion factors along bounded and unbounded edges.
\end{definition}

\subsection{Leaky moduli}
\label{sec:leakymod}
 A moduli space of %leaky (or 
$k$-leaky tropical covers may be constructed as a stack over the category of cone complexes, or as a generalized cone complex. The construction for the double ramification cycle problem, when $k = 0$, is explained in~\cite{MR21,UZ19} and is nearly identical to~\cite{ACP,CMR14a}. The higher $k$ generalization is straightforward, and we only provide a sketch. 

\noindent
{\sc Rigidified moduli for a fixed graph.} Fix a combinatorial type $\Theta$ of a leaky cover. Let $E$ be the edge set of $\Gamma$ and $F$ the edge set of the target $T$. Let $E_0\subset E$ be the subset of edges that have expansion factor equal to $0$; we refer to them as \textit{vertical}. In order to enhance this to a leaky cover, we provide each edge in $F$ with an edge length. Each element of $E$ maps piecewise linearly with determined integer slope onto a corresponding element of $F$. Assuming this slope is nonzero, given an element $e$ in $E$, its length is fully determined by the length of its corresponding image edge. If additionally, we fix an isomorphism of the source with a fixed copy of $\Gamma$, all leaky covers of this combinatorial type are naturally parameterized by the open cell in:
\[
\sigma_\Theta = \RR_{\geq 0}^F\times\RR_{\geq 0}^{E_0}.
\]
{We declare a point to be \textit{integral} if  all edges, both of source and of target, have integer length. Note that there is another possible integral structure which can also be useful, see Remark~\ref{rem: two-lattices}.}

\noindent
{\sc Boundary of the cell and edge contractions.}  The boundary faces of the cone $\sigma_\Theta$ parameterize leaky covers from ``degenerations'' of $\Theta$. Precisely, the set of faces is characterized by the notion of a \textit{contraction}. \rcolor{ A \textit{simple contraction} of $\Theta$ contracts  either one edge of $\Gamma$ that maps to a vertex of $T$, or an edge of $T$ and all the edges of $\Gamma$ mapping to it. A \textit{contraction} is a sequence of simple contractions. Contractions give rise to new combinatorial types, which  index the faces of $\sigma_\Theta$. }%\renzo{Got rid of the notion of commensurate edges, that was never used after this paragraph.}

%Say that elements in $E$ are \textit{commensurate} if they map with positive slope to the same element of $F$. A \textit{simple contraction} of $\Theta$ contracts a subset of commensurate edges, and if the contracted edge is not vertical, also contracts the corresponding target edge. A \textit{contraction} is a sequence of simple contractions. Contractions give rise to new combinatorial types, which precisely index the faces of $\sigma_\Theta$. 

\noindent
{\sc Automorphisms.} Given a combinatorial type $\Theta$, an automorphism of $\Theta$ is an automorphism of the source -- as a weighted dual graph -- that commutes with the map to the target graph. An automorphism of $\Theta$ give rise to an automorphism of the cone $\sigma_\Theta$. 

\noindent
{\sc Colimit construction.} Consider the category $I_{g,n}(\mathbf x)$ whose objects are combinatorial types of leaky covers of fixed genus and markings, with unbounded directions having slope $\mathbf x$ as discussed above. The morphisms are given by graph contractions and automorphisms. Let $\mathbf{RPC}$ be the category of rational polyhedral cones, equipped with face morphisms. The association $\Theta \mapsto \sigma_\Theta$ defines a functor
\[
I_{g,n}(\mathbf x)\to \mathbf{RPC}.
\]
By~\cite[Section~3]{CCUW} there is a cone stack obtained as the colimit of this diagram as a stack over the category $\mathbf{RPC}$. 

%\noindent
%{\sc Integral structure.} Finally, we declare a point of the colimit space  above to be \textit{integral} if  all edges, both of source and of target, have integer length. Note that this is a finite index sublattice of the natural lattice in the cones $\sigma_\Theta$ described above, since the source must also have integer edge lengths. 

\begin{definition}[Moduli space of leaky covers] \label{def:msleaky}
We denote the \emph{moduli space of $k$-leaky covers of genus $g$ and degree $\mathbf{x}$} as constructed above by $\leak$. 
\end{definition}

By forgetting the data of the cover, each leaky cover gives rise to a tropical curve, so there is a map of cone stacks
\[
\leak\hookrightarrow \cM_{g,n}^{\trop}. 
\]
Note that this forgetful arrow is injective on automorphism groups: the automorphisms of a cover are defined as a subset of those of the curve. The map is therefore representable by cone spaces in the sense of~\cite[Section~2.1]{CCUW}. In fact, one easily sees explicitly that if we fix a cone $\sigma\to \cM_{g,n}^{\trop}$, then forming the fiber product, we a map
\[
\Delta\to \sigma,
\]
which is a union of cones in $\sigma$ parameterizing leaky covers. Note that the underlying map on topological spaces is injective, which justifies the notation above. 

\begin{remark}[The two integral structures]\label{rem: two-lattices}
Note that when making calculations, and specifically those having to do with indices of maps between cone complexes, it is convenient to use the alternate integral structure on $\leak$, where an integer point is one such that all contracted source edge and target edge lengths are integers, and to endow the cone with a weight, as is standard in tropical geometry. Given a type $\Theta$, we have described two integral structures -- the first demanding integrality of source and target edges, and a second demanding integrality of \textit{vertical} source and and target edges. The first is a finite index sublattice of the latter. We call the resulting index the \textit{weight} of the cone.  The next example illustrates the typical situation.\renzo{we are slightly cheating here, because the index of the lattice is only one factor in the weight of the cone. There is the other factor coming from the normalization of the HM space...} 
\end{remark}

%{\color{red}
%\begin{remark}\label{rem-modulispace}
%The set of all leaky covers of a given combinatorial type forms an open polyhedron in a vector space parametrizing the lengths of all edges. The equations are given by the condition that the cycles have to close up, the inequalities by the fact that edge lengths are positive. We can identify a point on the boundary of such a polyhedron with the cover for which we remove the edges whose lengths have been shrunk to zero. In this way, we can form an \emph{abstract polyhedral complex} parametrizing all leaky covers  of genus $g$ and degree $\mathbf{x}$. As common in tropical geometry, the top-dimensional polyhedra in a complex are equipped with a \emph{weight}, which is defined to be the index of the lattice given by the equations that the cycles close up and that vertices have the same image as required times the size of the automorphism group.
%\end{remark}
%}

\begin{example}\label{ex-polyhedron}
Consider the leaky cover from Figure \ref{fig-leakycoverex} and its combinatorial type. It has five bounded edges, of which four form a cycle. The equation that the cycle closes up is $3l_1+l_2=l_3+l_4$, where $l_1$ and $l_2$ denote the lengths of the upper edges of the cycle and $l_3$ and $l_4$ the lengths of the lower. The equation that the two middle vertices have the same image is $3l_1=l_3$, or, equivalently, $l_2=l_4$.
All leaky covers of this combinatorial type are parametrized by the points in the $2$-dimensional open polyhedron
$$\RR_{>0}^4\cap \{3l_1+l_2-l_3-l_4=0, l_2-l_4=0\}.$$
The index of the lattice defined by $3l_1+l_2-l_3-l_4=0, l_2-l_4-0$ equals the greatest common divisor of the absolute values of the $2\times2$-minors of the matrix
$$\left(\begin{matrix}
3&1&-1&-1&\\0&1&0&-1
\end{matrix}\right)
$$
which equals $1$. Thus, the weight of the corresponding top-dimensional stratum in the moduli space of leaky covers of genus $1$ and degree $(5,-1,-1)$, in the sense of Remark~\ref{rem: two-lattices}, is $1$.\renzo{We have not mentioned any notion of weights for cones at this point. Also, these weights  imply that the integral structure is not the natural one on $\RR_{\geq 0}^F\times\RR_{\geq 0}^{E_0}$ but rather the one inherited from $\RR^E$, which is indeed the choice we make but it is stated in the next subsection. Hannah: I agree with Renzo, we should say something about the weight before. I don't feel like going ahead and changing anything before talking about it...(maybe via email will be sufficient)}
\end{example}

%The set of $\PP^1_{\trop}$'s  subdivided with $r:=n-2+2g$ $2$-valent genus $0$ vertices is an open orthant parametrizing the lengths of the distances between the vertices. We can identify points on the bounday with $\PP^1_{\trop}$'s subdivided with fewer vertices. We denote this closed parameter space by $P_r$.

\subsection{Subdivisions and blowups} Generalizing the standard dictionary of toric geometry, modifications of the polyhedral structure on $\cM_{g,n}^{\trop}$ give rise to birational modifications of $\Mbar_{g,n}$, see~\cite{AW}. There are three basic operations of interest to us. Let $\Sigma$ be a cone complex. 
\begin{description}
\item[Complete subdivisions] A complete subdivision is a morphism $\Sigma'\to \Sigma$ of cone complexes that is bijective on the supports and preserves integral points. 
\item[Rooting] Given $\Sigma$, a root construction is a morphism of cone complexes $\Sigma'\to \Sigma$ that is bijective on supports and injective on lattice points. 
\item[Passage to a subcomplex] A \textit{subcomplex} is a morphism $\Sigma'\to \Sigma$ that is an inclusion of a subset of cones. 
\end{description}

The operations for cone stacks are obtained in the standard fashion: a morphism $\cM\to \cM_{g,n}^\trop$ is a complete subdivision if it becomes a complete subdivision after pulling back along a strict morphism $\Sigma\to \cM_{g,n}^{\trop}$ from a cone complex. One analogously extends root constructions and subcomplexes. Complete subdivisions give rise to proper, representable, birational modifications of $\Mbar_{g,n}$; the passage to a sublattice gives rise to a generalized root construction along the boundary of $\Mbar_{g,n}$; passage to a subcomplex corresponds to an open immersion obtained by removing a collection of closed boundary strata.

\subsection{The transverse cycle} Consider the morphism of tropical moduli stacks 
\[
\mathsf{DR}_g^{\trop}(\mathbf x, k)\hookrightarrow \cM_{g,n}^{\trop}. 
\]
By the equivalence of categories in~\cite[Section~6.11]{CCUW} this determines a not necessarily proper, birational transformation of the Artin fan of $\Mbar_{g,n}$ of Deligne--Mumford type. By pulling this back to $\Mbar_{g,n}$ we obtain a stack equipped with a birational morphism of Deligne--Mumford type to $\Mbar_{g,n}$: 
\[
\cM_{g}^\sfbu(\mathbf x, k)\to \Mbar_{g,n}.
\]
%\rcolor{Hannah: to be consistent with the notation used later, should we remove the $n$ here? I.e. always remove $n$ when $x$ is around, as it is implicitely given then?}

At each point on $\leak$, the tropical curve $\Gamma$ is equipped with a piecewise linear function, by viewing $\Gamma\to T$ as a function well-defined up to translation. By the main results of the paper of Marcus and Wise~\cite{MW17} there is an associated Abel--Jacobi map:
\[
\mathsf{aj}_{\mathbf x}:\cM_{g,n}^\sfbu(\mathbf x, k)\to \mathbf{Pic},
\]
where the codomain is taken to be the universal Picard scheme of the universal curve over $\cM_{g,n}^\sfbu(\mathbf x, k)$. 

We would like to pull back in Chow homology along this map. We take a moment to explain how to think about why this pullback exists. First, Marcus--Wise prove that the map $\mathsf{aj}_{\mathbf x}$ is finite and unramified~\cite[Section~4.3]{MW17}; its cotangent complex therefore vanishes in degrees smaller than $-1$, so it is a lci morphism. The map therefore carries a Gysin pullback. We will use this in a moment.

Though we will not need it, the way we set up this problem, the logarithmic structure on the domain is saturated, and in this case the map above is actually a section of a smooth fibration, so it is a regular embedding rather than merely a local complete intersection. See~\cite[Section~6.2]{BHPSS}. 

In fact, there is yet another way to think about the pullback, which can be useful in practice. The cone complex $\mathsf{DR}_g^{\trop}(\mathbf x, k)$ is simplicial: a complete set of coordinates for the cone associated to a combinatorial type is given by the target edge lengths and the lengths of the contracted edges in the source graph. A simplicial cone complex determines both an Artin fan with toric orbifold singularities, and a smooth cover of such an Artin fan by a morphism of Deligne--Mumford type. Indeed, a simplicial affine toric variety has a canonical such resolution, see for instance~\cite[Section~2.3]{Mol16}. If we take the latter perspective, the space $\cM_{g,n}^\sfbu(\mathbf x, k)$ is smooth as a Deligne--Mumford stack. The target of the Abel--Jacobi map is a smooth fibration over this Deligne--Mumford stack. In this case, the pullback in Chow homology defined simply by applying Poincar\'e  duality, pulling back in cohomology, and applying Poincar\'e duality again.

\begin{remark}
In this paper, we typically take the view that simplicial fans give rise to {\it smooth} stacks rather than singular {\it varieties}. However, since we work with rational coefficients, this difference is mild; this is explained in the introduction to and in Section~6 of Vistoli's paper~\cite{Vi89}.
\end{remark}

Equipped with this pullback, we have for each $k$, maps
\[
\cM_{g,n}^\sfbu(\mathbf x, k)\to \mathbf{Pic},
\]
obtained by sending a curve $C$ to the pluricanononical line bundle $\omega_C^{\mathrm{log},\otimes k}$ of the marked logarithmic curve. We refer to it as the $k$-pluricanonical section $\sigma_k$.

\begin{definition}\label{def: DR-spce}
The \textit{double ramification space} for the data $\mathbf x$ is the pullback along $\mathsf{aj}_{\mathbf x}$ of the image of the $k$-pluricanonical section. Diagramatically, it is the fiber product
\[
\begin{tikzcd}
\mathsf{DR}_g(\mathbf x,k)\arrow{r}\arrow{d} & \cM_{g,n}^\sfbu(\mathbf x, k)\arrow{d}{\sigma_k}\\
\cM_{g,n}^\sfbu(\mathbf x, k)\arrow{r}[swap]{\mathsf{aj}_{\mathbf x}} &  \mathbf{Pic}.
\end{tikzcd}
\] 
It is equipped with a virtual fundamental class by refined Gysin pullback:
\[
[\mathsf{DR}_g(\mathbf x,k)]^{\mathsf{vir}} = {\mathsf{aj}_{\mathbf x}}^![\cM_{g,n}^\sfbu(\mathbf x, k)].
\]
The class will be called the \textit{logarithmic double ramification cycle}.
\end{definition}

The space $\mathsf{DR}_g(\mathbf x,k)$ is proper over $\Spec(\mathbb C)$ by~\cite[Corollary~E]{MW17}. The pushforward of the logarithmic double ramification cycle to $\Mbar_{g,n}$ is the usual \textit{double ramification cycle}, though we will essentially never work with this pushforward, and instead work with intersections with classes on open subsets of \textit{blowups} of the moduli space, i.e. directly on $\cM_{g,n}^\sfbu(\mathbf x, k)$.

\subsection{Piecewise polynomials} The main players in this paper are intersection numbers of the logarithmic double ramification cycle with cohomology classes coming from the tropical moduli space. The classes are described using piecewise polynomials. 

Let $X$ be a fine and saturated logarithmic scheme or stack. There are two basic combinatorial objects associated to $X$. At each point $x\in X$, there is an associated characteristic monoid $P_x$. The \textit{cone} at $x$ is defined to be the dual cone $\sigma_x$ to $P_x$. The \textit{Artin cone} at $x$ is the Artin stack $\mathsf{A}_x =  \Spec k[P_x]/k[P_x^{\mathrm{gp}}]$. We then define
$$
\Sigma_X = \varinjlim \sigma_x \\
$$
and 
$$
\mathsf{A}_X = \varinjlim \mathsf{A}_x
$$
to be the \textit{cone complex} and \textit{Artin fan} of $X$. Note that the colimits are taken over the diagram of all points in $X$ under generization maps, and are taken respectively in the category of cone stacks and in logarithmic algebraic stacks. 

A \textit{piecewise polynomial function} on $\Sigma_X$ is the assignment of a polynomial function
\[
\sigma_x\to\RR
\]
with rational coefficients on each cone $\sigma_x$ that is compatible with the description of $\Sigma_X$ as a colimit. The notion is independent of presentation of $\Sigma_X$. Piecewise polynomials on $\Sigma_X$ form a ring denoted $\mathsf{PP}^\star(\Sigma_X)$. Piecewise polynomial functions are related to Chow rings of Artin stacks. A basic result proved in~\cite{MPS21,MR21} is the following.

\begin{theorem}
The Chow cohomology of the Artin stack $\mathsf A_X$ is isomorphic to the ring $\mathsf{PP}^\star(\Sigma_X)$.
\end{theorem}

If $\Sigma'\to \Sigma$ is a subdivision, there is an injective pullback map
\[
\mathsf{PP}^\star(\Sigma)\to \mathsf{PP}^\star(\Sigma'),
\]
obtained by identifying functions on $\Sigma$ as functions on $\Sigma'$; if $X'\to X$ is the proper birational map induced by this subdivision, the pullback map on Chow cohomology is precisely this pullback. Essentially by definition, there is a morphism of logarithmic stacks
\[
X\to \mathsf A_X,
\]
and there is a cohomological pullback map
\begin{equation}\label{eq:cohpb}
\mathsf{PP}^\star(\Sigma_X)\to \mathsf{CH}^\star_{\mathsf{op}}(X),
\end{equation}
preserving the degree. Piecewise polynomial functions are cohomology classes on Artin fans, and therefore exhibit the expected contravariant functoriality properties.  See~\cite[Appendix~C]{BS22} for details on operational Chow rings for algebraic stacks. 

Let us consider a proper stack $V$ of pure dimension $m$, equipped with a map $g: V\to X$. We obtain an operator:
\begin{eqnarray*}
\varphi_{V}: \mathsf{PP}^m(\Sigma_X)&\to& \QQ\\
f&\mapsto& \int_{V} g^\star f
\end{eqnarray*}
The logarithmic double ramification cycle similarly gives rise to an operator, where we use the virtual fundamental class to replace the pure dimensionality condition above.%: the stack $X$ is the open $\cM^\circ_{g,n}(\mathbf x,k)$ inside the blowup of $\Mbar_{g,n}$ constructed above, the class $\alpha$ is the logarithmic double ramification cycle, and the support $X$ is the space $\mathsf{DR}_g(\mathbf x,k)$ is the proper substack.

\begin{definition}
We denote by $$\mathsf{dr}_g^\lightning(\mathbf x, k): \mathsf{PP}^g(\mathsf{DR}_g^{\trop}(\mathbf x, k))\to \QQ$$ the operator defined by pulling back the piecewise polynomial to the proper space $\mathsf{DR}_g(\mathbf x,k)$, capping with the virtual fundamental class, and then taking degree.
\end{definition}

A basic question, which the present paper begins to answer, is to determine an efficient calculus for evaluating the operator on a given piecewise polynomial.

\subsection{Tropical evaluation map}\label{sec: tropical-evaluations} Given a leaky cover $\Gamma\to T$, we can remember only the target line graph $T$. There is a simple parameter space for such objects. Informally, it is obtained from $\RR_{\geq 0}^\infty$ by identifying all faces of equal dimension. 

More formally, consider the diagram $D$ of cones with cones given by a single $\RR_{\geq 0}^e$ for each $e\geq 0$, with exactly $e+1$ arrows, named $w_1^{(e)},\ldots, w_{e+1}^{(e)}$, from $\mathbb R_{\geq 0}^e$ to $\mathbb R_{\geq 0}^{e+1}$. The arrow $w_j^{(e)}$ is the face inclusion obtained by inserting $0$ into the $j^{th}$ coordinate. The diagram defines a cone space according to~\cite[Definition~2.12]{CCUW}, which we denote by $\mathsf{tEx}$. There is a natural morphism 
\[
\leak\to\mathsf{tEx}
\]
obtained by interpreting the codomain as a moduli space for metric line graphs, and then sending a cover to its target line graph. 

To see the moduli interpretation, observe that the set of line graphs with exactly $e+1$ vertices has $e$ edges and is parameterized by the interior cell in $\RR_{\geq 0}^e$. However, each interior point of the $e$ codimension $1$ cells can naturally be seen to parameterize line graphs with $e-1$ edges. Similarly, the codimension $2$ cells parameterize line graphs with $e-2$ edges and so on. The cone stack above is exactly obtained by gluing all these faces together, and we make the claimed identification as a consequence. A subdivision of $\RR$ is given by prescribing vertices at positions $t_0,\ldots, t_e$. Two such prescriptions give the same subdivision when \textit{all} the differences $t_i-t_{i-1}$ coincide. We can thus normalize to set $t_0 = 0$. The space $\mathsf{tEx}$ makes the resulting identifications. A detailed treatment will appear in~\cite{STR22}. See also~\cite{ACFW,Oes19}.

\subsection{Geometric evaluations} We have exhibited a morphism of cone complexes
\[
\mathsf{DR}_g^{\trop}(\mathbf x,k)\to \mathsf{tEx}.
\]
The cone space $\mathsf{tEx}$ determines an Artin fan which we denote   $\mathsf{Ex}$. The stack $\mathsf{Ex}$ has appeared elsewhere: it is isomorphic to the stack of expansions of an unparameterized $\mathbb P^1$ along $0$ and $\infty$, and can be identified with a natural open substack of $\mathfrak M_{0,2}^{\mathsf{ss}}$ of semistable $2$-pointed curves, see~\cite{ACFW,GV05}. A treatment via the correspondence between tropical geometry and Artin fans may be obtained by a rubber variant of the stack of expansions in~\cite{MR20}. 

By construction, the tropical evaluation map of Section~\ref{sec: tropical-evaluations} determines a morphism of stacks
\[
\mathsf{ev}: \cM_{g,n}^\sfbu(\mathbf x,k)\to \mathsf{Ex}. 
\]
The virtual class $\mathsf{DR}_g(\mathbf x,k)$ is a refined class with proper support in the domain of this morphism $\mathsf{ev}$. In the sequel, we intersect this virtual class with a cohomology class of codimension $2g-3+n$ on $\mathsf{Ex}$, specified by a piecewise polynomial function on the tropical space $\mathsf{tEx}$.

\section{Double Hurwitz numbers revisited}\label{sec:dhn}
This section uses the perspectives developed in Section \ref{sec:tcpp}  to recast the classical enumerative geometric problem of double Hurwitz numbers. In the first subsection, which may be skipped by the experts, we recall the notion of double Hurwitz numbers and provide a brief overview of some connections between double Hurwitz numbers and intersection theory on moduli spaces. In the second subsection we view double Hurwitz numbers as the degree of an intersection cycle on the logarithmic double ramification cycle. The third subsection shows how to extract double Hurwitz numbers of genus $h\leq g$ as natural evaluations against the logarithmic cycle.

\subsection{A few perspectives on Hurwitz numbers} 
Let $\mathbf{x}\in \ZZ^n \smallsetminus \{\underline 0\}$ be a vector of integers adding to zero. The {\it double Hurwitz number} $H_g(\mathbf{x})$ counts the number of covers of $\CC \PP^1$ by a curve of genus $g$, with ramification profiles 
\[
\mathbf{x}^{-} = \{x_i|x_i<0\} \ \ \textrm{over $\infty$}, \ \ \textrm{and} \ \ \ \mathbf{x}^{+} = \{x_i|x_i>0\} \ \  \textrm{over $0$},
\] 
and simple ramification over $r = 2g-2+n$ fixed points of $\PP^1$. 

It is standard to weight every cover by the size of its automorphism group;  the ramification data $\mathbf{x}$ is a vector, rather than a multi-set. Note in particular that the inverse images of $0$ and $\infty$ in the covering curve are labelled, and cannot be interchanged by automorphisms.
Double Hurwitz numbers have incarnations in many different mathematical areas. For example, they may be viewed as counting factorizations of the identity in the symmetric group or as the result  of a multiplication problem in the class algebra of the symmetric group. The piecewise polynomial structure of double Hurwitz numbers was first obtained by interpreting them as counts of decorated ribbon graphs, which naturally led to a translation to the count of lattice points inside appropriate polytopes~\cite{GJV}. 

In this section, we focus our attention on some perspectives that connect double Hurwitz numbers to the geometry of the moduli space of curves.

%The double Hurwitz number $H_g(\mathbf{x})$ has a natural interpretation in terms of the geometry of some moduli spaces. 

\subsubsection{Open moduli of covers} Let $\mathsf{Hurw}_g(\mathbf{x})$ denote the  Hurwitz space of isomorphism classes of smooth covers of $\CC\PP^1$ with specified discrete invariants as in the previous paragraph, and with exactly $r$ distinct simple branch points. There is a natural {\it branch morphism} 
\begin{equation}
br: \mathsf{Hurw}_g(\mathbf{x}) \to [\cM_{0,r+2}/S_r]
\end{equation}
 recording the position of the branch points. By construction, the number $H_g(\mathbf{x})$ is the degree of the branch morphism. Although elementary, this perspective does not yield a method of calculation.  

\subsubsection{Compact moduli of covers} The source and target of the branch morphism admit natural compactifications. In \cite{cm:dhn}, the authors consider the branch morphism from a moduli space of rubber relative stable maps to the symmetrized Losev--Manin compactification of the space of $r+2$ points on $\mathbb P^1$, where $r$ is the expected number of simple branchings:
\[
br: \overline{\cM}_{g}^\sim(\PP^1; \mathbf{x}^-[0],  \mathbf{x}^+[\infty])\to [LM(r)/S_r]
\]
The maps are summarized in Figure~\ref{fig:dia}. 

The double Hurwitz number $H_g(\mathbf{x})$ is the degree of the \rcolor{zero-dimensional} cycle $br^\star([pt])$. The pushforward of this cycle to $\Mbar_{g,n}$ is again a cycle whose degree is the Hurwitz number. Once the cycle has been placed on $\Mbar_{g,n}$, it admits a natural formula. Specifically, by using cotangent line comparison methods, one obtains an expression for this cycle in terms of $\psi$-classes. The result is a sum of boundary strata decorated with $\psi$ classes, each with a coefficient which is a polynomial in the entries of $\mathbf{x}$. The formula exhibits the piecewise polynomiality of $H_g(\mathbf{x})$, first proved in~\cite[Theorem 2.1]{GJV}.

\begin{figure}[tbp]
$$\xymatrix{
\overline{\cM}_{g}^\sim(\PP^1; \mathbf{x}^-[0],  \mathbf{x}^+[\infty])   \ar[r]^{\ \ \ \ \ \ \ \ \ \ \ sr.} \ar[d]^{br} & \Mbar_{g,n}\\
[LM(r)/S_r]
} 
$$
\caption{The tautological diagram used to express $H_g(\mathbf{x})$ as an intersection number on $\Mbar_{g,n}$.}
\label{fig:dia}
\end{figure}

\subsubsection{Via target degenerations} An alternative approach to the degree of $br^\star([pt])$ is to first choose a zero-dimensional boundary stratum  $\Delta \in [LM(r)/S_r]$ as a representative for the class of a point. The preimage consists of a chain of $r$ projective lines, attached at nodes. The two special branch points with ramification profiles $\mathbf{x}^\pm$ map to opposite external components of the chain, and there is exactly one simple branch point on each component of the chain. For any inverse image $[f: C\to T]\in br^{-1}(\Delta)$, the irreducible components of $C$ are rational and contain either two or three special points, where a special point is either a node or a marking. The degree of
$br^\star([pt])$ is then obtained by counting each inverse image $[f: C\to T]$ with the multiplicity prescribed by the {\it degeneration formula} (see \cite[Theorem 4.5]{GV05}  or \cite[Theorem 3.15]{Li02}). 

\subsubsection{Via tropical geometry} The degeneration formula provides a natural bridge to tropical geometry. The dual graphs of the source curves of maps   $[f: C\to T]\in br^{-1}(\Delta)$ are naturally identified with combinatorial types of tropical covers  $F: \Gamma \to \RR$ of the tropical line. The expansion factors of the edges correspond to the ramification orders of the corresponding nodes, which are well-defined by the pre-deformability condition in relative Gromov--Witten theory. The identification gives a bijection between the points $[f: C\to T]\in br^{-1}(\Delta)$ and the {\it monodromy graphs} of \cite{CJM1} (see also Definition \ref{def-mongraph}), which may be interpreted as the inverse images of a  \rcolor{single} maximal cone via a tropical branch morphism \rcolor{$br^{\trop}$}, a map of equidimensional cone complexes, \rcolor{where each cone is endowed with an integral lattice}. The local degree \rcolor{of the map $br^{\trop}$} at a maximal dimensional cone $\sigma_F$ \rcolor{is defined to be} the lattice index of the image of the integral lattice of $\sigma_F$ inside the integral lattice of $br^{\trop}(\sigma_F)$.\rcolor{The degree of $br^{\trop}$, a.k.a. the {\it tropical Hurwitz number}, is obtained by adding up the local degrees over all inverse images of a maximal cone in the target}. The correspondence theorem between algebraic and tropical Hurwitz numbers follows from the fact that the local degree of the tropical branch morphism at $\sigma_F$ equals the degeneration formula multiplicity for the corresponding algebraic cover $f:C\to T$. The perspective here generalizes to all Hurwitz numbers via admissible covers, and to the full stationary descendant theory of $\mathbb P^1$ with certain atomic intersection theory inputs, see~\cite{CMR14a,CJMR}. 

\rcolor{
\subsection{Hurwitz numbers as intersection numbers in $\mathsf{CH}_{\mathsf{log}}^\star(\Mbar_{g,n})$}
The present section is an original contribution of this paper to the story of Hurwitz numbers. We  present a new perspective on  double Hurwitz numbers as intersection numbers in the logarithmic Chow ring of the moduli spaces of curves. We begin by recalling some notation and facts from the tropical computation of double Hurwitz numbers in \cite{CJM1} that play a role in the upcoming constructions. 

\begin{definition}[Monodromy graphs]
\label{def-mongraph}
For fixed $g$ and $\mathbf{x}=(x_1,\ldots,x_n)$, a graph $\Gamma$ is a \emph{monodromy graph of type $(g, \mathbf{x})$}  if:
\begin{enumerate}
\item $\Gamma$ is a connected, directed graph,\rcolor{with first Betti number equal to $g$}. 
\item $\Gamma$ has $n$ ends %\footnote{We assume the ends to be unbounded, and therefore to not be adjacent to a one-valent vertex. If one firmly believes that graphs should be compact, then in the following add "not one-valent" to any time the word vertices is used.} 
which are directed inward, and labeled by the expansion factors $x_1,\ldots,x_n$. If $x_i >0$, we say it is an \emph{in-end}, otherwise it is an \emph{out-end}. 
\item \rcolor{Vertices of $\Gamma$ of valence greater than $2$ are exactly} $3$-valent.
\item After reversing the orientation of the out-ends, $\Gamma$ does not have sinks or sources. The  vertices are ordered compatibly with the partial ordering now induced by the directions of the edges.
\item Every bounded edge $e$ of the graph is equipped with an expansion factor $w(e)\in \NN$.  \rcolor{Each integer  $x_i$ should be thought as an expansion factor for the corresponding  end of $\Gamma$, and its  sign should be switched when reversing the orientation of the end.} These satisfy the \emph{balancing condition} at each $3$-valent vertex: the sum of all expansion factors of  incoming  edges equals the sum of the expansion factors of all outgoing edges.
\end{enumerate}
\end{definition}

Monodromy graphs remember the combinatorial information that is needed to compute the double Hurwitz numbers.% In \cite{CJM1} it is shown that the degree of the tropical branch morphism is obtained as a weighted sum over monodromy graphs.

\begin{theorem}[{\cite[Corollary~4.4]{CJM1}}] \label{thm:thn}
The double Hurwitz number $H_{g}(\mathbf{x})$ equals the sum over all monodromy graphs $\Gamma$ of type $(g, \mathbf{x})$, where each is given multiplicity $m_\Gamma$  equal the product of the expansion factors of its bounded edges:
$$H_{g}(\mathbf{x}) = \sum_{\Gamma} \frac{1}{|\Aut(\Gamma)|} \prod_e \omega(e).$$\end{theorem}

In \cite{CJM1} the above weighted sum over monodromy graphs is further shown to be the degree of a tropical branch morphism, that we now recall.

Consider the moduli space of tropical relative rubber stable maps  $\TSM$, i.e. the $k=0$ case of the spaces constructed  in Section \ref{sec:leakymod}; there is a natural bijection between monodromy graphs and maximal cones of $\TSM$ parameterizing covers where all expansion factors are strictly positive.
Denoting $r = 2g-2+n$, there is a branch morphism \begin{equation} \label{eq:tropbr}br_{\trop}: \TSM \to \mathsf{tEx}, %\left[\LM/S_{r}\right]=: P_r,
\end{equation}
to the tropical evaluation space of Section~\ref{sec: tropical-evaluations}, \rcolor{and a stabilization morphism
\begin{equation}
st_{\trop}: \TSM\to \cM_{g,n}^{\trop}.
\end{equation}}
Concretely, the target of the branch morphism parameterizes metric line graphs, or equivalently, polyhedral decompositions of $\mathbb R$ by finitely many vertices, considered up to a global shift. In this latter perspective, by normalizing so that the \rcolor{smallest} point is $0\in \RR$, the interior of the cone parameterizing targets with $r$ vertices may be coordinatized as $$\sigma = \{0< t_1< t_2 <\ldots < t_{r-1}\}\subset \RR^{r-1}.$$

Another useful characterization of the cones in $\TSM$ indexed by monodromy graphs is that $br^{\trop}$ maps them onto the $(r-1)$-dimensional cone of $ \mathsf{tEx}$.

\begin{definition}
We denote by $\TMB$ any simplicial subdivision of the moduli space of tropical curves $\mathcal{M}_{g,n}^\trop$ containing $st_{\trop}(\TSM)$ (or a refinement thereof) as a subcomplex.
\end{definition}
Such a subdivision certainly exists if one does not worry about possibly subdividing the cones of $st_{\trop}(\TSM)$. In fact one may adapt the construction in \cite[Section 7.3]{HMPPS} to construct a subdivision where there is no need to refine the cones of $st_{\trop}(\TSM)$ corresponding to monodromy graphs, which are already simplicial, see~\cite{Mol22}.

For any ray $\rho$ of $\TSM$,  denote by  $\varphi_\rho$  the piecewise linear function with slope $1$ along $\rho$, and slope $0$ along any other ray of $\TMB$.

\begin{definition}
We define the {\it branch  polynomial} $br_g(\mathbf x, 0)$ to be the following piecewise polynomial function on $\TMB$:
$$
br_g(\mathbf x, 0):=\sum_{\sigma_\Gamma } m_\Gamma\cdot  \prod_{\rho\subseteq \sigma_\Gamma}\varphi_{\rho},
$$
where the sum runs over all cones of $\TMB$ corresponding to monodromy graphs  and the multiplicities $m_\Gamma$ are as defined in Theorem \ref{thm:thn}.
\end{definition}

A key property of the branch polyomial (and the reason for the name) is that when restricted to  $\TSM$ one has:
\begin{equation}\label{eq:brprop}
br_g(\mathbf x, 0)_{|\TSM} = br_{\trop}^\star\left(t_1\prod_{i=2}^{r-1} (t_i-t_{i-1})\right).
\end{equation}
Observe that both functions vanish outside cones corresponding to monodromy graphs. The  expression on the left hand side by definition, the one on the right hand side because such cones map to lower dimensional faces of $\mathsf{tEx}$, where at least one of the terms $t_i-t_{i-1}$ vanishes.

We have all the ingredients to state the main theorem in this section.

\begin{customthm}{A} \label{thm:A}
With notation as  throughout this section, we have
\begin{equation} \label{eq:prodoc}
H_g (\mathbf{x})= \deg\left(
\LDR \cdot br_g(\mathbf x, 0)
\right).
\end{equation}
\end{customthm}
\begin{proof}
Recall from Definition \ref{def: DR-spce} that the subdivision $\TMB$ gives rise to a double ramification space with a morphism $st: \LDR \to \cM_{g,n}^\sfbu(\mathbf x,0)$. By projection formula, 
\begin{equation}
\deg\left(\LDR \cdot br_g(\mathbf x, 0)\right)=  \deg(st^\star(br_g(\mathbf x, 0))).
\end{equation}
Since the assignment of a cohomology class to a piecewise polynomial function is functorial, we have
\begin{equation} \label{funct}
 st_{\trop}^\star(br_g(\mathbf x, 0)) =  st^\star(br_g(\mathbf x, 0)) ,
\end{equation}
where in the left hand side of \eqref{funct}, $br_g(\mathbf x, 0)$ is regarded as piecewise polynomial function on $\TMB$, whereas on the right hand side as a cohomology class on $\cM_{g,n}^\sfbu(\mathbf x,0)$(namely the image of the former via the cohomological pullback from \eqref{eq:cohpb}).

As observed in \eqref{eq:brprop}, 
\begin{equation}\label{eq:stbr}
st_{\trop}^\star(br_g(\mathbf x, 0)) = br_{\trop}^\star\left(t_1\prod_{i=2}^{r-1} (t_i-t_{i-1})\right),\end{equation}
where the expression  in the parenthesis in the right hand side is a (piecewise) polynomial function on $\mathsf{tEx}$.

Consider the stack quotient  $\left[LM(r)/S_r\right]$ of the Losev--Manin moduli space by the symmetric group $S_r$. Points of this space correspond to stable chains of projective lines, where the stability condition demands that each element of the chain contains at least $1$ of the $r$ light points, and that the $r$ points may coincide but must be smooth. By forgetting the $r$ points, we obtain a point of $\mathsf{Ex}$, since the latter parameterizes expansions of a rubber $\mathbb P^1$ along $0$ and $\infty$. 

Consider the chain of morphisms:
\begin{equation}\label{LM-EX}
\LDR \stackrel{\pi}{\to}\mathsf{DR}_g(\mathbf x, 0) \stackrel{br}{\to}\left[LM(r)/S_r\right]\stackrel{p}{\to} \mathsf{Ex}
\end{equation}

The polynomial function $t_1\prod_{i=2}^{r-1} (t_i-t_{i-1})$  on $\mathsf{tEx}$ determines a cohomology class on $\mathsf{Ex}$ which pulls back along $p$ from \eqref{LM-EX} to the class of the unique closed stratum in $\left[LM(r)/S_r\right]$, which is equivalent to the class of a point. From \eqref{eq:stbr} one then obtains:

\begin{align} \label{boh3}
\deg( st_{\trop}^\star\left(br_g(\mathbf x, 0)\right)) = \deg(\pi_\star st_{\trop}^\star\left(br_g(\mathbf x, 0)\right))   = \nonumber \\
 \deg\left(br^\star p^\star \left(t_1\prod_{i=2}^{r-1} (t_i-t_{i-1})\right)\right)= \deg(br^\star\left([pt]\right)) = H_{g}(\mathbf{x}),
\end{align}
thus concluding the proof of the Theorem.
\end{proof}

\subsection{Geometric interpretation}

In this section we give a geometric, meaning ``non-virtual'', interpretation of  Theorem \ref{thm:A} by exhibiting \eqref{eq:prodoc} as the intersection of tautological cycles 
 on a birational model of the moduli space of curves. In this birational model, we consider the actual strict transform of the double ramification cycle, namely the closure of the locus of maps from smooth curves. In this space, it possible to find a collection of strata whose intersection with this closure is a zero dimensional cycle of degree equal to $H_g{(\mathbf{x})}$. To make the exposition more natural, we begin by giving a name to the subcomplex of $\TSM$ whose maximal cones are indexed by monodromy graphs, which should be considered the ``main component'' of $\TSM$.

%One has also a stabilization morphism:
 %$$st_{\trop}: \TSM \to \mathcal{M}^{\trop}_{g, n}.$$

\begin{definition}
We denote by $\TMC$  the union of cones in $\TSM$ such that in the associated combinatorial type, all bounded edges have non-zero expansion factor. 

%
%the  image via $st_\trop$ of the closure of the inverse image via the branch morphism of the interior of the cone $\sigma$. As a sentence, that's a mouthful; in symbols:
%$$
%\TMC := st_\trop\left( \overline{br_{\trop}^{-1}(\sigma^\sfbu)} \right).
%$$
%The cone complex structure on $\TMC$ is induced from the cone complex structure on the space $\TSM$; notice that in general it does not agree with the one induced by restriction from $\mathcal{M}^{\trop}_{g, n}$. However the integral lattice on $\TMC$ is  restricted from the integral lattice of $\mathcal{M}^{\trop}_{g, n}$. 
%{\color{blue} I convinced myself of this last statement by looking at the $M_{1,2}$ example - is it correct?}
\end{definition}

The space $\TMC$ has another important description. Let $\sigma^\sfi$ be the interior of the cone  in $\mathsf{tEx}$ consisting of line graphs with $r$ distinct vertices. Then
\[
\TMC := \overline{br_{\trop}^{-1}(\sigma^\sfi)},
\]
where the closure is taken in $\TSM$. The maximal cones of $\TMC$ all have dimension $2g-3+n$ and are naturally indexed by monodromy graphs of type $(g, \mathbf{x})$.
%\begin{remark}
%{\color{blue}Explain here that one may still work with the full relative stable maps space, but this closure of the main component makes things more transparent. Make a nod to tropical admissible covers.}
%\end{remark}

%The cone complex \rcolor{$st_{\trop}(\TMC)$} does not typically give a complete subdivision of $\mathcal{M}^{\trop}_{g, n}$, but it is a subcomplex of a subdivision. 

As discussed in Section \ref{sec:tcpp}, any subdivision $\TMB$  of $\mathcal{M}^{\trop}_{g, n}$  containing $\TMC$ as a subcomplex determines a birational morphism. Precisely, we first pass to the non-proper birational transformation of the Artin fan of $\Mbar_{g,n}$ determined by this subdivision. We then perform the ordinary stack theoretic pullback of the map $\Mbar_{g,n}$ to its Artin fan to obtain:
\begin{equation}
\pi: \MB\to \Mbar_{g,n},
\end{equation}
and piecewise polynomial functions on $\TMB$ give rise to cohomology classes on $\MB$. %\rcolor{We  have a stabilization morphism:
%\[st: \mathsf{DR}_g(\mathbf x,0)\to \cM_g(\mathbf x,0).
%\]}

%Consider the proper transform of the closure of the main component (whose generic point parameterizes maps from smooth source curves) of the space of relative stable maps  $\TC$. It is equipped with a stabilization morphism $st: \TC\to \MB$. From the discussion in Section \ref{}, one sees that $\TC$ is dimensionally transverse to the boundary of $\MB$. For any $m$-dimensional cone $\sigma \in \TMC$, there is a corresponding codimension-$m$, locally closed stratum $\Delta_\sigma\subseteq \MB$. In the next proposition we show that the double Hurwitz number $H_{g}(\mathbf{x})$ is obtained as the intersection of $\TC$ with a weighted union of some of these strata. 

{

\begin{definition}
The \textit{main component} of $\LDR$ is the closure of the locus of consisting of non-constant maps from smooth domains. It is denoted $\TC$.
\end{definition} 

The main component is certainly compact, and from the Riemann--Hurwitz theorem it has dimension equal to its expected dimension $2g-3+n$. Indeed, it coincides with the image of an appropriate space of admissible covers, see for instance~\cite[Section~3 and 4]{Cav16}. However, the locus can be singular. The following lemma asserts that we can ``stay away'' from the singular locus of the closure for the purpose of studying double Hurwitz numbers. 

\begin{lemma}\label{lem: transverse-where-needed}
Let $\sigma$ be a cone of $\TMC$ of dimension $2g-3+n$ and let $V_\sigma$ be the corresponding closed stratum in $\cM_{g}(\mathbf x,0)$. The intersection of $\TC$ with $V_\sigma$ consists of finitely many points.
\end{lemma}

\begin{proof}
First consider the intersection of $\TC$ with the locally closed stratum $V_\sigma^\circ$. Any point of $\TC$ that lies in this stratum is maximally degenerate, and no component of the curve is contracted. Relative stable maps to $\mathbb P^1$ relative to $0$ and $\infty$ that are also \textit{finite} as morphisms, are always points where the moduli space has expected dimension. Indeed, such maps are automatically admissible covers, and all infinitesimal deformations of them are as well. It follows that the intersection above must be finitely many points. Since no points of $\LDR$ have a combinatorial type whose cone contains $\sigma$ as a face, the intersection with locally closed and closed strata coincide. The lemma is a consequence. 
\end{proof}

\begin{remark}
In case it is useful for future applications, we record that a stronger statement is true than the lemma. Precisely, if $\sigma$ is any cone on $\TMC$, the intersection of the \textit{locally closed} stratum associated to $\sigma$ with $\TC$ is smooth and of the expected dimension. This can be deduced by either comparing with the deformation theory of admissible covers, as in~\cite[Section~4]{CMR14a} or using the explicit deformation sequence for fine and saturated logarithmic maps~\cite[Proposition~4.2]{CFPU}. Note for the stronger conclusion, it is essential to take the finer of the two integral structures described in Remark~\ref{rem: two-lattices}.
\end{remark}

}

\begin{proposition}\label{prop:dhntc}
With notation as developed throughout the section, we have
\begin{equation}
H_{g}(\mathbf{x}) =\deg\left( \TC \cdot \sum_{\sigma_\Gamma} m_\Gamma\ \Delta_{\sigma_\Gamma}\right),
\end{equation}
where the sum ranges over all maximal cones of $\TMC$, and $\sigma_\Gamma$ denotes the cone indexed by the monodromy graph $\Gamma$.
\end{proposition}

\begin{proof}
%We start with the observation that the tropical moduli space $\TSM$ is simplicial, in the sense that it is a colimit of simplicial cones. But moreover, any cone corresponding to a tropical cover without contracted edges has any many distinct one-dimensional rays as its dimension.% \renzo{I removed the following sentence (still in the source), which for me was actually a bit imprecise and confusing, even though I think I know what it means. This fact seems to me simple enough that it doesn't need to be explained more. }

% Indeed, ray generators for the cone Logarithmic intersections of double ramification cycles compact target edges, which cannot be identified by automorphisms. If $\Gamma$ is a monodromy graph as above, the associated cone is necessarily of this form. Fix such a graph $\Gamma$. 

\rcolor{Fix a monodromy graph $\Gamma$, and note that the corresponding cone $\sigma_\Gamma$ is simplicial:  consider a tropical cover $f: \Gamma \to T$ of topological type identified by $\Gamma$; denote by $\Gamma_{< i}$ the inverse image $f^{-1}((-\infty, v_i))$ (i.e. the part of $\Gamma$ which maps to the left of the $i$-th vertex) and by $S_{\Gamma_{< i}}$ a smooth topological surface with genus equal to the first Betti number of  $\Gamma_{< i}$, and with as many punctures as the ends of $\Gamma_{< i}$; one can see that since the surfaces associated to two consecutive vertices differ by a pair of pants,  the following relation holds among  Euler characteristics:  $\chi(S_{\Gamma_{< i}})  = \chi(S_{\Gamma_{< i+1}}) +1$.   The rays  of the cone corresponding to the monodromy graph $\Gamma$ are obtained by contracting all compact edges that do not lie above two consecutive vertices in the line graph $T$. No two rays can be the same, as that would require the Euler characteristics of the surfaces associated to parts of $\Gamma$ mapping to the  left of two distinct vertices of $T$ to be equal.}  Using the language from Section \ref{sec:tcpp}, the class of the stratum $\Delta_{\sigma_\Gamma}$ may be described as a piecewise polynomial function on $\TMB$. As in the previous section, we denote by $\varphi_\rho$  the  piecewise linear function whose slope along this ray is $1$, and whose slope along any other ray is $0$. Then we have:
%As notation, we declare that given a ray $\rho$ in a simplicial cone complex, there is a
\begin{equation}\label{eq:bpolym}
[\Delta_{\sigma_\Gamma}] = \prod_{\rho \in {\sigma_\Gamma}} \varphi_\rho =\ : \varphi_{\sigma_\Gamma},
\end{equation}
where the product runs over all the rays of the cone ${\sigma_\Gamma}$. Since the cone $\sigma_\Gamma$ is simplicial, the stratum is indeed obtained as the intersection of the divisors corresponding to the rays of the cone $\sigma_\Gamma$. By Lemma~\ref{lem: transverse-where-needed} above, the intersection of all these divisors and $\TC$ is dimensionally proper, and there is no excess intersection calculation to perform. 

We pause to clarify that, as is common in the literature about the tautological ring, there is potential ambiguity about what one means by ``class of a stratum". It could be interpreted as the class of a locus as opposed to the pushforward of the fundamental class via the appropriate gluing morphism. We adopt the latter perspective, which is consistent both with the most recent literature, and the conventions of the software package admcycles.  We also take this opportunity to note that in the first ar$\chi$iv version of the paper, the multiplicity $m_\Gamma$ was erroneously dropped from the formula. We thank J. Schmitt for pointing out the correction, which arose from cross-checking the formula with the one in~\cite{HMPPS} using the admcycles package.

Let us now return to the proof. By applying the projection formula, we have 
\begin{equation}\label{eq:strint}\deg(\TC \cdot \Delta_{\sigma_\Gamma}) = \deg(st^\star(\varphi_{\sigma_\Gamma})). \end{equation} Since the assignment of a cohomology class to a piecewise polynomial function is functorial, we have
\begin{equation} \label{funct2}
 st^\star(\varphi_{\sigma_\Gamma}) = st_{\trop}^\star(\varphi_{\sigma_\Gamma}),
\end{equation}
where in the right hand side of \eqref{funct}, $\varphi_{\sigma_\Gamma}$ is regarded as a piecewise polynomial function, whereas in the left hand side as a cohomology class on $\MB$. %By definition of the maps $st_{\trop}$ and $br_{\trop}$, one sees that 
The multiplicities $m_\Gamma$ are defined in order for the following equality to hold:
\begin{equation} \label{boh}
st_{\trop}^\star(m_\Gamma \ \varphi_{\sigma_\Gamma})_{|{\sigma_\Gamma}} = br_{\trop}^\star\left(t_1\prod_{i=2}^{r-1} (t_i-t_{i-1})\right)_{|{\sigma_\Gamma}};
\end{equation}
for $\tilde{\Gamma}\not= \Gamma$, we have
\begin{equation} \label{boh1}
st_{\trop}^\star(\varphi_{\sigma_\Gamma})_{|{\sigma_{\tilde{\Gamma}}}} = 0.
\end{equation}

Summing over all monodromy graphs, one obtains
\begin{eqnarray} \label{boh2}
\deg\left( \TC \cdot \sum_{\sigma_\Gamma} m_\Gamma\ \Delta_{\sigma_\Gamma}\right) &=& \deg\left(st_{\trop}^\star \left( \sum_\Gamma m_\Gamma\cdot \varphi_{\sigma_\Gamma}\right)\right)\\ &=& \deg\left( br_{\trop}^\star\left( t_1\prod_{i=2}^{r-1} (t_i-t_{i-1})\right)\right), 
\end{eqnarray}
 which reproduces the computation in the proof of Theorem \ref{thm:A}.
\end{proof}}

\begin{example}
We consider the double Hurwitz number $H_1((3,-3)) = 2$ and compute it as in Proposition \ref{prop:dhntc}.
The moduli space of tropical rubber stable maps $\mathsf{DR}^{\trop}_1((3,-3),0)$ , illustrated in red in Figure \ref{fig:+}, is not equidimensional: it is the union of a two dimensional cone parameterizing tropical maps with a contracting tropical elliptic curve, with a one dimensional cone, whose general element parameterizes the covers with two compact edges of length $l_1$ and $l_2$, as drawn on the leftmost side of Figure \ref{fig:+}. The tropical branch morphism contracts the two dimensional cone of  $\mathsf{DR}^{\trop}_1((3,-3),0)$ and maps the one dimensional cone onto the unique ray of  $\mathsf{tEx}$. %$[\mathcal{M}^{\trop,\sim}_{0,2+2\cdot\epsilon}/S_2]$. 
The image of the one dimensional cone of 
$\mathsf{DR}^{\trop}_1((3,-3),0)$ via the tropical stabilization morphism is the slope $-1/2$ ray $\rho$ in $\mathcal{M}^{\trop}_{1,2}$. This ray alone gives a subdivision $\mathsf{DR}^{\trop}_1((3,-3),0)$ of $\mathcal{M}^{\trop}_{1,2}$, which imposes a  (weighted) toroidal blowup in the algebraic moduli space $\overline{\mathcal{M}}_{1,2}$, as depicted in Figure \ref{fig:tblow}.  
\begin{figure}

\begin{tikzpicture}[x=0.75pt,y=0.75pt,yscale=-1,xscale=1]

%Straight Lines [id:da8917500140504453] 
\draw [color={rgb, 255:red, 208; green, 2; blue, 27 }  ,draw opacity=1 ][line width=1.5]    (187.33,121.33) -- (320.33,120.33) ;
%Straight Lines [id:da8200732785289048] 
\draw [line width=2.25]    (185.67,179.67) -- (317.67,180.33) ;
%Shape: Arc [id:dp7616791884162275] 
\draw  [draw opacity=0][line width=1.5]  (54.95,34.55) .. controls (55.06,34.26) and (55.18,33.97) .. (55.3,33.69) .. controls (61.73,18.42) and (79.32,11.25) .. (94.59,17.68) .. controls (109.2,23.84) and (116.39,40.22) .. (111.35,55) -- (82.95,45.33) -- cycle ; \draw  [line width=1.5]  (54.95,34.55) .. controls (55.06,34.26) and (55.18,33.97) .. (55.3,33.69) .. controls (61.73,18.42) and (79.32,11.25) .. (94.59,17.68) .. controls (109.2,23.84) and (116.39,40.22) .. (111.35,55) ;
%Shape: Arc [id:dp35616020118992187] 
\draw  [draw opacity=0][line width=1.5]  (111.4,55.28) .. controls (102.29,60.93) and (89.43,61.94) .. (77.27,56.91) .. controls (66.38,52.42) and (58.5,44.1) .. (55.2,34.92) -- (88.59,29.49) -- cycle ; \draw  [line width=1.5]  (111.4,55.28) .. controls (102.29,60.93) and (89.43,61.94) .. (77.27,56.91) .. controls (66.38,52.42) and (58.5,44.1) .. (55.2,34.92) ;
%Straight Lines [id:da8146619750442132] 
\draw [line width=1.5]    (111.4,55.28) -- (141.65,68.25) ;
%Straight Lines [id:da023398398352143968] 
\draw [line width=1.5]    (55.2,34.26) -- (29.15,23.82) ;
%Straight Lines [id:da11757853516478245] 
\draw [line width=1.5]    (8.33,82) -- (108.67,126.67) ;
%Straight Lines [id:da24963758995311358] 
\draw [fill={rgb, 255:red, 248; green, 175; blue, 184 }  ,fill opacity=1 ]   (132.42,139.68) -- (177.2,160.27) ;
\draw [shift={(179.01,161.11)}, rotate = 204.7] [color={rgb, 255:red, 0; green, 0; blue, 0 }  ][line width=0.75]    (10.93,-3.29) .. controls (6.95,-1.4) and (3.31,-0.3) .. (0,0) .. controls (3.31,0.3) and (6.95,1.4) .. (10.93,3.29)   ;
%Straight Lines [id:da8041736436073036] 
\draw    (71.67,69) -- (61.19,91.19) ;
\draw [shift={(60.33,93)}, rotate = 295.28] [color={rgb, 255:red, 0; green, 0; blue, 0 }  ][line width=0.75]    (10.93,-3.29) .. controls (6.95,-1.4) and (3.31,-0.3) .. (0,0) .. controls (3.31,0.3) and (6.95,1.4) .. (10.93,3.29)   ;
%Straight Lines [id:da7328072317262506] 
\draw [line width=1.5]    (379.33,19.67) -- (380.33,120.33) ;
%Straight Lines [id:da13683366829476906] 
\draw [line width=1.5]    (517.67,119.67) -- (380.33,120.33) ;
%Straight Lines [id:da3445213173218835] 
\draw [line width=1.5]  [dash pattern={on 5.63pt off 4.5pt}]  (380.33,120.33) -- (517.67,257.67) ;
%Straight Lines [id:da29652019461079016] 
\draw    (249.67,133.67) -- (249.67,167) ;
\draw [shift={(249.67,169)}, rotate = 270] [color={rgb, 255:red, 0; green, 0; blue, 0 }  ][line width=0.75]    (10.93,-3.29) .. controls (6.95,-1.4) and (3.31,-0.3) .. (0,0) .. controls (3.31,0.3) and (6.95,1.4) .. (10.93,3.29)   ;
%Straight Lines [id:da9269738295304971] 
\draw    (333.33,120.67) -- (364.33,120.98) ;
\draw [shift={(366.33,121)}, rotate = 180.58] [color={rgb, 255:red, 0; green, 0; blue, 0 }  ][line width=0.75]    (10.93,-3.29) .. controls (6.95,-1.4) and (3.31,-0.3) .. (0,0) .. controls (3.31,0.3) and (6.95,1.4) .. (10.93,3.29)   ;
%Straight Lines [id:da6291823020693358] 
\draw [color={rgb, 255:red, 208; green, 2; blue, 27 }  ,draw opacity=1 ][line width=1.5]    (380.33,120.33) -- (519.67,187) ;
%Shape: Circle [id:dp898592894634013] 
\draw  [color={rgb, 255:red, 208; green, 2; blue, 27 }  ,draw opacity=1 ][fill={rgb, 255:red, 208; green, 2; blue, 27 }  ,fill opacity=1 ] (181.17,121.33) .. controls (181.17,119.63) and (182.55,118.25) .. (184.25,118.25) .. controls (185.95,118.25) and (187.33,119.63) .. (187.33,121.33) .. controls (187.33,123.04) and (185.95,124.42) .. (184.25,124.42) .. controls (182.55,124.42) and (181.17,123.04) .. (181.17,121.33) -- cycle ;
%Shape: Circle [id:dp6594003524128602] 
\draw  [fill={rgb, 255:red, 0; green, 0; blue, 0 }  ,fill opacity=1 ] (182.58,179.67) .. controls (182.58,177.96) and (183.96,176.58) .. (185.67,176.58) .. controls (187.37,176.58) and (188.75,177.96) .. (188.75,179.67) .. controls (188.75,181.37) and (187.37,182.75) .. (185.67,182.75) .. controls (183.96,182.75) and (182.58,181.37) .. (182.58,179.67) -- cycle ;
%Shape: Circle [id:dp6466681654472077] 
\draw  [fill={rgb, 255:red, 0; green, 0; blue, 0 }  ,fill opacity=1 ] (258.83,119.83) .. controls (258.83,118.13) and (260.21,116.75) .. (261.92,116.75) .. controls (263.62,116.75) and (265,118.13) .. (265,119.83) .. controls (265,121.54) and (263.62,122.92) .. (261.92,122.92) .. controls (260.21,122.92) and (258.83,121.54) .. (258.83,119.83) -- cycle ;
%Shape: Circle [id:dp5571626958660276] 
\draw  [fill={rgb, 255:red, 0; green, 0; blue, 0 }  ,fill opacity=1 ] (257.67,180) .. controls (257.67,178.3) and (259.05,176.92) .. (260.75,176.92) .. controls (262.45,176.92) and (263.83,178.3) .. (263.83,180) .. controls (263.83,181.7) and (262.45,183.08) .. (260.75,183.08) .. controls (259.05,183.08) and (257.67,181.7) .. (257.67,180) -- cycle ;
%Straight Lines [id:da8414523515303862] 
\draw  [dash pattern={on 0.84pt off 2.51pt}]  (55.2,34.92) -- (29,92) ;
%Straight Lines [id:da5614163689503024] 
\draw  [dash pattern={on 0.84pt off 2.51pt}]  (111.35,55) -- (84,114) ;
%Shape: Triangle [id:dp37935996626466917] 
\draw  [color={rgb, 255:red, 255; green, 255; blue, 255 }  ,draw opacity=1 ][fill={rgb, 255:red, 248; green, 175; blue, 184 }  ,fill opacity=1 ] (199.18,19.7) -- (184.25,121.33) -- (140.22,26.1) -- cycle ;
%Straight Lines [id:da33492422901314] 
\draw    (156.01,73.11) -- (240.17,108.34) ;
\draw [shift={(242.01,109.11)}, rotate = 202.71] [color={rgb, 255:red, 0; green, 0; blue, 0 }  ][line width=0.75]    (10.93,-3.29) .. controls (6.95,-1.4) and (3.31,-0.3) .. (0,0) .. controls (3.31,0.3) and (6.95,1.4) .. (10.93,3.29)   ;
%Straight Lines [id:da731727338319873] 
\draw [color={rgb, 255:red, 208; green, 2; blue, 27 }  ,draw opacity=1 ][line width=1.5]    (140.74,22.51) -- (184.25,121.33) ;
%Straight Lines [id:da4547673722856662] 
\draw [color={rgb, 255:red, 208; green, 2; blue, 27 }  ,draw opacity=1 ][line width=1.5]    (199.18,19.7) -- (184.25,118.25) ;
%Shape: Circle [id:dp6415510427401139] 
\draw  [color={rgb, 255:red, 208; green, 2; blue, 27 }  ,draw opacity=1 ][fill={rgb, 255:red, 208; green, 2; blue, 27 }  ,fill opacity=1 ] (181.25,121.42) .. controls (181.25,119.71) and (182.63,118.33) .. (184.33,118.33) .. controls (186.04,118.33) and (187.42,119.71) .. (187.42,121.42) .. controls (187.42,123.12) and (186.04,124.5) .. (184.33,124.5) .. controls (182.63,124.5) and (181.25,123.12) .. (181.25,121.42) -- cycle ;

% Text Node
\draw (258,141.4) node [anchor=north west][inner sep=0.75pt]    {$br_{trop}$};
% Text Node
\draw (332,95.4) node [anchor=north west][inner sep=0.75pt]    {$st_{trop}$};
% Text Node
\draw (80,76.4) node [anchor=north west][inner sep=0.75pt]    {$f$};
% Text Node
%\draw (529,31.4) node [anchor=north west][inner sep=0.75pt]    {$\mathcal{M}^{\trop}_{1}((3,-3),0) $};
% Text Node
\draw (83,-1.6) node [anchor=north west][inner sep=0.75pt]    {$1$};
% Text Node
\draw (80,40.4) node [anchor=north west][inner sep=0.75pt]    {$2$};
% Text Nodet_
\draw (18,95.4) node [anchor=north west][inner sep=0.75pt]    {$0$};
% Text Node
\draw (71,119.4) node [anchor=north west][inner sep=0.75pt]    {$t_{1}$};
% Text Node
\draw (-17,67.4) node [anchor=north west][inner sep=0.75pt]    {$\mathbb{R}$};
% Text Node
\draw (3,6.4) node [anchor=north west][inner sep=0.75pt]    {$\Gamma$};
% Text Node
\draw (200,51.4) node [anchor=north west][inner sep=0.75pt]    {$\textcolor[rgb]{0.82,0.01,0.11}{\mathsf{DR}^\trop_1((3,-3),0)}$};
% Text Node
\draw (200,198.4) node [anchor=north west][inner sep=0.75pt]    {$\mathsf{tEx}$};
% Text Node
\draw (32,7.4) node [anchor=north west][inner sep=0.75pt]    {$3$};
% Text Node
\draw (126,45.4) node [anchor=north west][inner sep=0.75pt]    {$3$};
% Text Node
\draw (319.67,183.73) node [anchor=north west][inner sep=0.75pt]    {$t_{1}$};
% Text Node
\draw (294,128.4) node [anchor=north west][inner sep=0.75pt]    {$\textcolor[rgb]{0.82,0.01,0.11}{l}\textcolor[rgb]{0.82,0.01,0.11}{_{1}}\textcolor[rgb]{0.82,0.01,0.11}{=2l}\textcolor[rgb]{0.82,0.01,0.11}{_{2}}$};
% Text Node
\draw (529,183.4) node [anchor=north west][inner sep=0.75pt]    {$\textcolor[rgb]{0.82,0.01,0.11}{\mathsf{DR}^{\sfi,\trop}_{1}((3,-3),0) =\rho}$};

\end{tikzpicture}

%\mathcal{M}^{\sim}_1(\PP^1,(3,-3))

\caption{The subdivision of the moduli space $\mathcal{M}^{\trop}_{1,2}$ induced by the moduli space of tropical rubber maps in genus $1$ with contact $(3,-3)$.}\label{fig-subd}
\label{fig:+}
\end{figure}
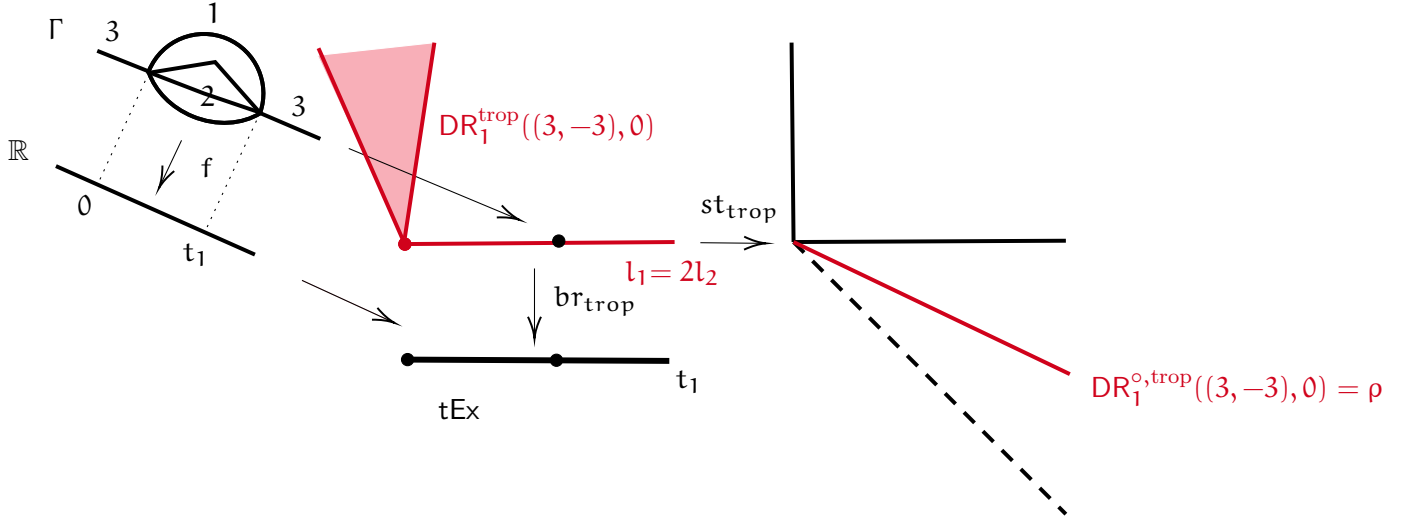
The closure of  $\mathsf{DR}^{\sfi}_1(3,-3)$ is not dimensionally transverse to the boundary of $\overline{\mathcal{M}}_{1,2}$, but its proper transform $\mathsf{DR}^{\sfi, \lightening}_1(3,-3)$ is dimensionally transverse to the exceptional divisor $E = D_\rho$. One may compute the intersection multiplicity $|D_\rho \cdot \mathsf{DR}^{\sfi,\lightening}_1((3,-3),0)|$ via a direct local coordinate computation as in \cite{CMR14a}. We bypass this technical computation by observing that the piecewise linear function  $\varphi_\rho$ associated to the exceptional divisor $D_\rho$ is such that:
\begin{equation}\label{eq:stbr2}
 st_\trop^\star(2\varphi_\rho) =  2l_2 = br_\trop^\star(t_1). 
\end{equation}
Note that the multiplicity $m_\Gamma$ prescribed in the definition of branch polynomial for this monodromy graph is $2$. The multiplicity of $br_\trop^\star(t_1)$ is  computed as in \cite{CJM1}, where it is shown to equal the product of the weights of the compact edges of $\Gamma$.
 
Alternatively, one may observe that the piecewise linear function $t_1$ determines the class of a point on the Losev--Manin space $[LM(2)/S_2]$,  and the assignment of a cohomology class on the algebraic moduli space to a piecewise polynomial on the tropical moduli space is functorial. It follows from \eqref{eq:stbr} that the degree of the class associated to the piecewise polynomial function $st_\trop^\star(\varphi_\rho)$ equals the  degree of $br^\star[pt]$, which is by definition the Hurwitz number $H_1(3,-3)$.

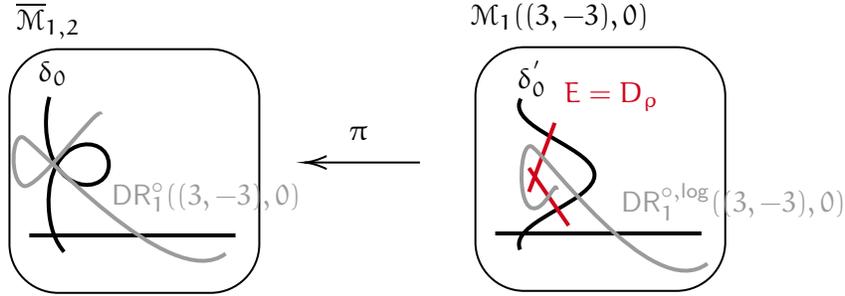
\begin{figure}

\tikzset{every picture/.style={line width=0.75pt}} %set default line width to 0.75pt        

\begin{tikzpicture}[x=0.75pt,y=0.75pt,yscale=-1,xscale=1]

%Shape: Free Drawing [id:dp2848108588399382] 
%\draw  [color={rgb, 255:red, 0; green, 0; blue, 0 }  ][line width=3] [line join = round][line cap = round] (622,447) .. controls (622,447) and (622,447) .. (622,447) ;
%Rounded Rect [id:dp7710357702157611] 
\draw   (100,68.6) .. controls (100,55.01) and (111.01,44) .. (124.6,44) -- (201.4,44) .. controls (214.99,44) and (226,55.01) .. (226,68.6) -- (226,142.4) .. controls (226,155.99) and (214.99,167) .. (201.4,167) -- (124.6,167) .. controls (111.01,167) and (100,155.99) .. (100,142.4) -- cycle ;
%Straight Lines [id:da3812145786252499] 
\draw [line width=1.5]    (110,138) -- (214,138) ;
%Shape: Spring [id:dp6813855233645498] 
\draw  [line width=1.5]  (127.54,146.28) .. controls (123.32,141.57) and (120.13,134) .. (119.76,122.92) .. controls (118.58,86.94) and (149.58,85.92) .. (150.11,101.91) .. controls (150.63,117.9) and (119.63,118.92) .. (118.45,82.94) .. controls (118.26,77.11) and (118.91,72.2) .. (120.14,68.13) ;
%Shape: Wave [id:dp3154830081533053] 
\draw  [line width=1.5]  (357.9,69.01) .. controls (357.29,70.15) and (356.96,71.31) .. (356.96,72.51) .. controls (356.96,78.84) and (366.25,84.3) .. (376.01,90.01) .. controls (385.77,95.71) and (395.06,101.17) .. (395.06,107.51) .. controls (395.06,113.84) and (385.77,119.3) .. (376.01,125.01) .. controls (366.25,130.71) and (356.96,136.17) .. (356.96,142.51) .. controls (356.96,143.42) and (357.16,144.32) .. (357.52,145.2) ;
%Rounded Rect [id:dp5974339669230184] 
\draw   (335,67.6) .. controls (335,54.01) and (346.01,43) .. (359.6,43) -- (436.4,43) .. controls (449.99,43) and (461,54.01) .. (461,67.6) -- (461,141.4) .. controls (461,154.99) and (449.99,166) .. (436.4,166) -- (359.6,166) .. controls (346.01,166) and (335,154.99) .. (335,141.4) -- cycle ;
%Straight Lines [id:da9407891063583087] 
\draw [line width=1.5]    (345,137) -- (449,137) ;
%Straight Lines [id:da8878062715600266] 
\draw [color={rgb, 255:red, 208; green, 2; blue, 27 }  ,draw opacity=1 ][line width=1.5]    (375.01,81.2) -- (362,116) ;
%Straight Lines [id:da4271176512192283] 
\draw    (307.01,101.2) -- (251.01,101.2) ;
\draw [shift={(249.01,101.2)}, rotate = 360] [color={rgb, 255:red, 0; green, 0; blue, 0 }  ][line width=0.75]    (10.93,-3.29) .. controls (6.95,-1.4) and (3.31,-0.3) .. (0,0) .. controls (3.31,0.3) and (6.95,1.4) .. (10.93,3.29)   ;
%Curve Lines [id:da9866685324979689] 
\draw [color={rgb, 255:red, 155; green, 155; blue, 155 }  ,draw opacity=1 ][line width=1.5]    (146,77) .. controls (145.39,70.91) and (114,116) .. (107,113) .. controls (100,110) and (101,87) .. (106.66,88.43) .. controls (112.32,89.86) and (180.69,168.44) .. (209.01,147.2) ;
%Straight Lines [id:da5682877455489312] 
\draw [color={rgb, 255:red, 208; green, 2; blue, 27 }  ,draw opacity=1 ][line width=1.5]    (362,104) -- (382,133) ;
%Curve Lines [id:da17964030755695215] 
\draw [color={rgb, 255:red, 155; green, 155; blue, 155 }  ,draw opacity=1 ][line width=1.5]    (376,115) .. controls (375.39,108.91) and (370,127) .. (363,124) .. controls (356,121) and (357.34,91.57) .. (363,93) .. controls (368.66,94.43) and (423.69,173.44) .. (452.01,152.2) ;
% Text Node
\draw (270,81.4) node [anchor=north west][inner sep=0.75pt]    {$\pi $};
% Text Node
\draw (102,19.4) node [anchor=north west][inner sep=0.75pt]    {$\overline{\mathcal{M}}_{1,2}$};
% Text Node
\draw (331,19.4) node [anchor=north west][inner sep=0.75pt]    {${\mathcal{M}}_{1}((3,-3),0)$};
% Text Node
\draw (122,148) node [anchor=north west][inner sep=0.75pt]   [align=left] {};
% Text Node
\draw (378,59.4) node [anchor=north west][inner sep=0.75pt]    {$\textcolor[rgb]{0.82,0.01,0.11}{E= D_\rho}$};
% Text Node
\draw (150,110) node [anchor=north west][inner sep=0.75pt]    {$\textcolor[rgb]{0.61,0.61,0.61}{\mathsf{DR}_1^{\sfi} ((3,-3),0)}$};
% Text Node
\draw (407,110) node [anchor=north west][inner sep=0.75pt]    {$\textcolor[rgb]{0.61,0.61,0.61}{\mathsf{DR}_1^{\sfi,\lightening} ((3,-3),0)}$};
% Text Node
\draw (113,48.4) node [anchor=north west][inner sep=0.75pt]    {$\delta _{0}$};
% Text Node
\draw (348,48.4) node [anchor=north west][inner sep=0.75pt]    {$ \begin{array}{l}
\delta _{0}^{'}\\
\end{array}$};
\end{tikzpicture}
\caption{The birational transformation of $\overline{\mathcal{M}}_{1,2}$ for the $(3,-3)$ double ramifucation problem in genus $1$ and the dimensionally transverse cycle $\TC$. \rcolor{We denote by $\mathsf{DR}_1^{\sfi} ((3,-3),0)$ the closure of the component of the space of relative stable maps where for the general point the source curve is smooth.}}
\label{fig:tblow}
\end{figure}
\end{example}

%In the course of the proof of Proposition \ref{prop:dhntc} we introduced a piecewise polynomial function that plays an important role in our story, and so we now give it a name. 

\subsection{Lower genus double Hurwitz numbers and $\LDR$.}\label{sec:lg}

The double Hurwitz number $H_g(\mathbf{x})$ is the degree of the class $br^\star([pt])$ in the moduli space of rubber relative stable maps $\RSM$. For a general choice of cycle representing the class of a point, the cycle $br^\star([pt])$ is supported on the main component of the moduli spaces of relative stable maps, the closure of the locus parameterizing maps from smooth source curves. In Proposition \ref{prop:dhntc}  the sub-cone complex of the moduli space of tropical relative stable maps generically parameterizing covers with no contracting subgraph is used to obtain a piecewise polynomial function and subsequently a cohomology class on $\MB$ whose intersection with $\LDR$ has degree $H_g(\mathbf{x})$. In this section, different choices of subcomplexes of the space of tropical relative stable maps will yield cohomology classes on $\MB$ whose intersection with $\LDR$ allows to extract the double Hurwitz numbers $H_h(\mathbf{x})$, for $h<g$.
 
\begin{definition}
For $n\geq 1$, let $T_n$ denote a graph obtained from a rooted, trivalent tree with $n+1$ leaves by  attaching vertices of genus $1$ at every leaf except the root, see the red part of  Figure \ref{fig:lowerg}. While it does not matter what trivalent tree one considers, the graph $T_n$ is to be considered fixed.
\end{definition}

For every monodromy graph of type $(h, \mathbf{x})$, attach a copy of $T_{g-h}$ on the end labeled by $x_1$ \rcolor{and give all edges of $T_{g-h}$ expansion factor equal to $0$}. The graphs so obtained index a cone sub-complex $M^{\trop}_{h<g}(\mathbf{x})$ of the moduli space $\TSM$, which we use to define a cohomology class on $\MB$ extracting the double Hurwitz number $H_{h}(\mathbf{x})$.

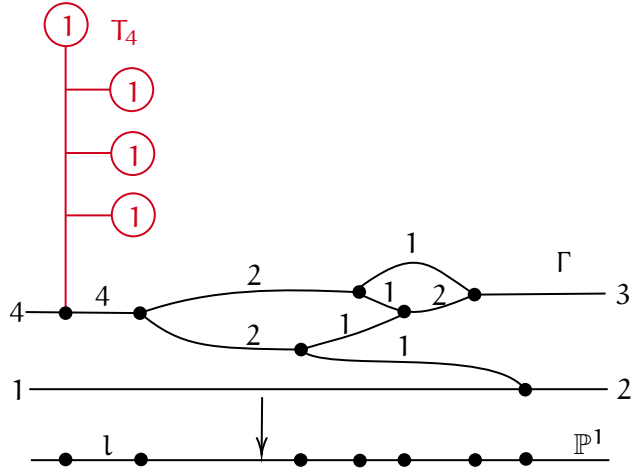
\begin{figure}

\tikzset{every picture/.style={line width=0.75pt}} %set default line width to 0.75pt        

\begin{tikzpicture}[x=0.75pt,y=0.75pt,yscale=-.8,xscale=.8]

%Shape: Circle [id:dp27563210514037073] 
\draw  [color={rgb, 255:red, 208; green, 2; blue, 27 }  ,draw opacity=1 ] (68.89,60.22) .. controls (68.89,52.74) and (74.96,46.67) .. (82.44,46.67) .. controls (89.93,46.67) and (96,52.74) .. (96,60.22) .. controls (96,67.71) and (89.93,73.78) .. (82.44,73.78) .. controls (74.96,73.78) and (68.89,67.71) .. (68.89,60.22) -- cycle ;
%Shape: Circle [id:dp060896288085674044] 
\draw  [color={rgb, 255:red, 208; green, 2; blue, 27 }  ,draw opacity=1 ] (69.56,100.56) .. controls (69.56,93.07) and (75.62,87) .. (83.11,87) .. controls (90.6,87) and (96.67,93.07) .. (96.67,100.56) .. controls (96.67,108.04) and (90.6,114.11) .. (83.11,114.11) .. controls (75.62,114.11) and (69.56,108.04) .. (69.56,100.56) -- cycle ;
%Shape: Circle [id:dp6660592559611882] 
\draw  [color={rgb, 255:red, 208; green, 2; blue, 27 }  ,draw opacity=1 ] (69.89,139.22) .. controls (69.89,131.74) and (75.96,125.67) .. (83.44,125.67) .. controls (90.93,125.67) and (97,131.74) .. (97,139.22) .. controls (97,146.71) and (90.93,152.78) .. (83.44,152.78) .. controls (75.96,152.78) and (69.89,146.71) .. (69.89,139.22) -- cycle ;
%Shape: Circle [id:dp49787870181772953] 
\draw  [color={rgb, 255:red, 208; green, 2; blue, 27 }  ,draw opacity=1 ] (27.56,19.22) .. controls (27.56,11.74) and (33.62,5.67) .. (41.11,5.67) .. controls (48.6,5.67) and (54.67,11.74) .. (54.67,19.22) .. controls (54.67,26.71) and (48.6,32.78) .. (41.11,32.78) .. controls (33.62,32.78) and (27.56,26.71) .. (27.56,19.22) -- cycle ;
%Curve Lines [id:da6475063654714281] 
\draw    (87.67,201.22) .. controls (132.89,185.22) and (176.89,185.89) .. (222.22,187.89) ;
%Curve Lines [id:da05000723221737857] 
\draw    (87.67,201.22) .. controls (113.89,222.89) and (129.89,223.56) .. (189.22,224.22) ;
%Straight Lines [id:da4640485796557752] 
\draw    (15.56,200.56) -- (87.67,201.22) ;
%Curve Lines [id:da4537485318988346] 
\draw    (222.22,187.89) .. controls (262.22,157.89) and (267.89,168.89) .. (298.56,189.56) ;
%Curve Lines [id:da27441515646355885] 
\draw    (222.22,187.89) .. controls (251.89,198.89) and (253.89,208.89) .. (298.56,189.56) ;
%Curve Lines [id:da37338096572050405] 
\draw    (189.22,224.22) .. controls (222.56,214.22) and (218.56,217.56) .. (257.89,200.22) ;
%Curve Lines [id:da7927212269124388] 
\draw    (189.22,224.22) .. controls (201.22,239.56) and (310.56,221.56) .. (330.56,249.56) ;
%Straight Lines [id:da17245161162496703] 
\draw    (381.22,188.89) -- (298.56,189.56) ;
%Straight Lines [id:da8157024380531648] 
\draw    (17.89,249.22) -- (382.56,249.22) ;
%Straight Lines [id:da651704761340897] 
\draw    (18.89,294) -- (383.56,294) ;
%Straight Lines [id:da0386828495718754] 
\draw    (164.67,254.22) -- (164.35,288.56) ;
\draw [shift={(164.33,290.56)}, rotate = 270.53] [color={rgb, 255:red, 0; green, 0; blue, 0 }  ][line width=0.75]    (10.93,-3.29) .. controls (6.95,-1.4) and (3.31,-0.3) .. (0,0) .. controls (3.31,0.3) and (6.95,1.4) .. (10.93,3.29)   ;
%Shape: Circle [id:dp4202605103756456] 
\draw  [fill={rgb, 255:red, 0; green, 0; blue, 0 }  ,fill opacity=1 ] (84,201.22) .. controls (84,199.2) and (85.64,197.56) .. (87.67,197.56) .. controls (89.69,197.56) and (91.33,199.2) .. (91.33,201.22) .. controls (91.33,203.25) and (89.69,204.89) .. (87.67,204.89) .. controls (85.64,204.89) and (84,203.25) .. (84,201.22) -- cycle ;
%Shape: Circle [id:dp7613579091155085] 
\draw  [fill={rgb, 255:red, 0; green, 0; blue, 0 }  ,fill opacity=1 ] (185.56,224.22) .. controls (185.56,222.2) and (187.2,220.56) .. (189.22,220.56) .. controls (191.25,220.56) and (192.89,222.2) .. (192.89,224.22) .. controls (192.89,226.25) and (191.25,227.89) .. (189.22,227.89) .. controls (187.2,227.89) and (185.56,226.25) .. (185.56,224.22) -- cycle ;
%Shape: Circle [id:dp9858285798209727] 
\draw  [fill={rgb, 255:red, 0; green, 0; blue, 0 }  ,fill opacity=1 ] (222.22,187.89) .. controls (222.22,185.86) and (223.86,184.22) .. (225.89,184.22) .. controls (227.91,184.22) and (229.56,185.86) .. (229.56,187.89) .. controls (229.56,189.91) and (227.91,191.56) .. (225.89,191.56) .. controls (223.86,191.56) and (222.22,189.91) .. (222.22,187.89) -- cycle ;
%Shape: Circle [id:dp9670320356961506] 
\draw  [fill={rgb, 255:red, 0; green, 0; blue, 0 }  ,fill opacity=1 ] (250.56,200.22) .. controls (250.56,198.2) and (252.2,196.56) .. (254.22,196.56) .. controls (256.25,196.56) and (257.89,198.2) .. (257.89,200.22) .. controls (257.89,202.25) and (256.25,203.89) .. (254.22,203.89) .. controls (252.2,203.89) and (250.56,202.25) .. (250.56,200.22) -- cycle ;
%Shape: Circle [id:dp30485023652229826] 
\draw  [fill={rgb, 255:red, 0; green, 0; blue, 0 }  ,fill opacity=1 ] (294.89,189.56) .. controls (294.89,187.53) and (296.53,185.89) .. (298.56,185.89) .. controls (300.58,185.89) and (302.22,187.53) .. (302.22,189.56) .. controls (302.22,191.58) and (300.58,193.22) .. (298.56,193.22) .. controls (296.53,193.22) and (294.89,191.58) .. (294.89,189.56) -- cycle ;
%Shape: Circle [id:dp16529333256967815] 
\draw  [fill={rgb, 255:red, 0; green, 0; blue, 0 }  ,fill opacity=1 ] (326.89,249.56) .. controls (326.89,247.53) and (328.53,245.89) .. (330.56,245.89) .. controls (332.58,245.89) and (334.22,247.53) .. (334.22,249.56) .. controls (334.22,251.58) and (332.58,253.22) .. (330.56,253.22) .. controls (328.53,253.22) and (326.89,251.58) .. (326.89,249.56) -- cycle ;
%Shape: Circle [id:dp7072697483646009] 
\draw  [fill={rgb, 255:red, 0; green, 0; blue, 0 }  ,fill opacity=1 ] (84.67,293.56) .. controls (84.67,291.53) and (86.31,289.89) .. (88.33,289.89) .. controls (90.36,289.89) and (92,291.53) .. (92,293.56) .. controls (92,295.58) and (90.36,297.22) .. (88.33,297.22) .. controls (86.31,297.22) and (84.67,295.58) .. (84.67,293.56) -- cycle ;
%Shape: Circle [id:dp5263471006531724] 
\draw  [fill={rgb, 255:red, 0; green, 0; blue, 0 }  ,fill opacity=1 ] (185.33,293.89) .. controls (185.33,291.86) and (186.97,290.22) .. (189,290.22) .. controls (191.03,290.22) and (192.67,291.86) .. (192.67,293.89) .. controls (192.67,295.91) and (191.03,297.56) .. (189,297.56) .. controls (186.97,297.56) and (185.33,295.91) .. (185.33,293.89) -- cycle ;
%Shape: Circle [id:dp4221550343934821] 
\draw  [fill={rgb, 255:red, 0; green, 0; blue, 0 }  ,fill opacity=1 ] (222.67,294.22) .. controls (222.67,292.2) and (224.31,290.56) .. (226.33,290.56) .. controls (228.36,290.56) and (230,292.2) .. (230,294.22) .. controls (230,296.25) and (228.36,297.89) .. (226.33,297.89) .. controls (224.31,297.89) and (222.67,296.25) .. (222.67,294.22) -- cycle ;
%Shape: Circle [id:dp44067753059606085] 
\draw  [fill={rgb, 255:red, 0; green, 0; blue, 0 }  ,fill opacity=1 ] (250.67,293.89) .. controls (250.67,291.86) and (252.31,290.22) .. (254.33,290.22) .. controls (256.36,290.22) and (258,291.86) .. (258,293.89) .. controls (258,295.91) and (256.36,297.56) .. (254.33,297.56) .. controls (252.31,297.56) and (250.67,295.91) .. (250.67,293.89) -- cycle ;
%Shape: Circle [id:dp8108719186879174] 
\draw  [fill={rgb, 255:red, 0; green, 0; blue, 0 }  ,fill opacity=1 ] (295.33,293.89) .. controls (295.33,291.86) and (296.97,290.22) .. (299,290.22) .. controls (301.03,290.22) and (302.67,291.86) .. (302.67,293.89) .. controls (302.67,295.91) and (301.03,297.56) .. (299,297.56) .. controls (296.97,297.56) and (295.33,295.91) .. (295.33,293.89) -- cycle ;
%Shape: Circle [id:dp7248254583035711] 
\draw  [fill={rgb, 255:red, 0; green, 0; blue, 0 }  ,fill opacity=1 ] (327.33,293.22) .. controls (327.33,291.2) and (328.97,289.56) .. (331,289.56) .. controls (333.03,289.56) and (334.67,291.2) .. (334.67,293.22) .. controls (334.67,295.25) and (333.03,296.89) .. (331,296.89) .. controls (328.97,296.89) and (327.33,295.25) .. (327.33,293.22) -- cycle ;
%Straight Lines [id:da18845051294545856] 
\draw [color={rgb, 255:red, 208; green, 2; blue, 27 }  ,draw opacity=1 ]   (40.78,32.56) -- (40.78,201.22) ;
%Straight Lines [id:da492858465601711] 
\draw [color={rgb, 255:red, 208; green, 2; blue, 27 }  ,draw opacity=1 ]   (40.44,60.22) -- (69.11,60.22) ;
%Straight Lines [id:da7892568527722641] 
\draw [color={rgb, 255:red, 208; green, 2; blue, 27 }  ,draw opacity=1 ]   (41.11,100.22) -- (69.78,100.22) ;
%Straight Lines [id:da15700016121967275] 
\draw [color={rgb, 255:red, 208; green, 2; blue, 27 }  ,draw opacity=1 ]   (41.11,139.56) -- (69.78,139.56) ;
%Shape: Circle [id:dp17776195921417015] 
\draw  [fill={rgb, 255:red, 0; green, 0; blue, 0 }  ,fill opacity=1 ] (37.11,201.22) .. controls (37.11,199.2) and (38.75,197.56) .. (40.78,197.56) .. controls (42.8,197.56) and (44.44,199.2) .. (44.44,201.22) .. controls (44.44,203.25) and (42.8,204.89) .. (40.78,204.89) .. controls (38.75,204.89) and (37.11,203.25) .. (37.11,201.22) -- cycle ;
%Shape: Circle [id:dp11358140381840354] 
\draw  [fill={rgb, 255:red, 0; green, 0; blue, 0 }  ,fill opacity=1 ] (37.11,293.56) .. controls (37.11,291.53) and (38.75,289.89) .. (40.78,289.89) .. controls (42.8,289.89) and (44.44,291.53) .. (44.44,293.56) .. controls (44.44,295.58) and (42.8,297.22) .. (40.78,297.22) .. controls (38.75,297.22) and (37.11,295.58) .. (37.11,293.56) -- cycle ;

% Text Node
\draw (77.33,53.4) node [anchor=north west][inner sep=0.75pt]  [color={rgb, 255:red, 208; green, 2; blue, 27 }  ,opacity=1 ]  {$1$};
% Text Node
\draw (78,93.73) node [anchor=north west][inner sep=0.75pt]  [color={rgb, 255:red, 208; green, 2; blue, 27 }  ,opacity=1 ]  {$1$};
% Text Node
\draw (78.33,132.4) node [anchor=north west][inner sep=0.75pt]  [color={rgb, 255:red, 208; green, 2; blue, 27 }  ,opacity=1 ]  {$1$};
% Text Node
\draw (36,11.73) node [anchor=north west][inner sep=0.75pt]  [color={rgb, 255:red, 208; green, 2; blue, 27 }  ,opacity=1 ]  {$1$};
% Text Node
\draw (67.33,14.4) node [anchor=north west][inner sep=0.75pt]  [color={rgb, 255:red, 208; green, 2; blue, 27 }  ,opacity=1 ]  {$T_{4}$};
% Text Node
\draw (62,276.07) node [anchor=north west][inner sep=0.75pt]    {$l$};
% Text Node
\draw (4,192.4) node [anchor=north west][inner sep=0.75pt]    {$4$};
% Text Node
\draw (4.67,241.73) node [anchor=north west][inner sep=0.75pt]    {$1$};
% Text Node
\draw (58,182.07) node [anchor=north west][inner sep=0.75pt]    {$4$};
% Text Node
\draw (153.33,169.07) node [anchor=north west][inner sep=0.75pt]    {$2$};
% Text Node
\draw (152.67,206.73) node [anchor=north west][inner sep=0.75pt]    {$2$};
% Text Node
\draw (252.67,148.73) node [anchor=north west][inner sep=0.75pt]    {$1$};
% Text Node
\draw (210.67,199.73) node [anchor=north west][inner sep=0.75pt]    {$1$};
% Text Node
\draw (248,214.07) node [anchor=north west][inner sep=0.75pt]    {$1$};
% Text Node
\draw (238.67,179.73) node [anchor=north west][inner sep=0.75pt]    {$1$};
% Text Node
\draw (270,181.73) node [anchor=north west][inner sep=0.75pt]    {$2$};
% Text Node
\draw (386.67,240.73) node [anchor=north west][inner sep=0.75pt]    {$2$};
% Text Node
\draw (386,181.07) node [anchor=north west][inner sep=0.75pt]    {$3$};
% Text Node
\draw (348,161.07) node [anchor=north west][inner sep=0.75pt]    {$\Gamma $};
% Text Node
\draw (359.33,272.4) node [anchor=north west][inner sep=0.75pt]    {$\mathbb{P}^{1}$};
\end{tikzpicture}
\caption{A tropical cover in $M^{\trop}_{2<6}(4,1,-2,-3)$ corresponding to a maximal dimensional cone.}
\label{fig:lowerg}
\end{figure}

\begin{definition}
We define the {\it branch polynomial at level $h$} as
$$
br_{h<g}(\mathbf x, 0)  := \rcolor{ \frac{ 24^{g-h}}{x_1} \cdot |\Aut(T_{g-h})}|\cdot \sum_{\sigma_\Gamma \in M^{\trop}_{h<g}(\mathbf{x})} \rcolor{m_\Gamma} \ \varphi_{\sigma_\Gamma},
$$
with the sum running over all maximal cones of $M^{\trop}_{h<g}(\mathbf{x})$, \rcolor{and $m_\Gamma$ denoting the product of the  non-zero expansion factors of the compact edges of $\Gamma$ }.
\end{definition}

\begin{customthm}{C}
The genus $h$ Hurwitz number equals
\begin{equation}
H_{h}(\mathbf{x}) =\deg\left( \LDR \cdot br_{h<g}(\mathbf x, 0) \right).
\end{equation}
\end{customthm}

\begin{proof}
Each maximal cone $\sigma_\Gamma$ in $M^{\trop}_{h<g}(\mathbf{x})$ has dimension $2g-3+n$.
This is true as each such cone is a codimension $g-h$ boundary of a cone with $g-h$ contracting loops. The latter cone is of dimension $3g-3+n-h$, as only the $h$ non-contracted loops impose conditions on the lengths of the involved edges.
Hence the branch polynomial at level $h$ restricts to $\LDR$  to give a class in top degree, as expected.

 For every cone $\sigma_\Gamma$, the pull-back $st^\star(\Delta_{\sigma_\Gamma})$ is supported on the component $C_h\cong \overline{\mathcal{M}}^{\sim}_{h,1}(\PP^1, \mathbf{x}) \times  \overline{\mathcal{M}}_{g-h, 1}$ of the moduli space of relative stable maps ${\overline{\mathcal{M}}^{\sim}_{g}(\PP^1, \mathbf{x})}$ parameterizing a rubber relative stable map of genus $h$ with attached a contracting curve of genus $g-h$.
 
 The piecewise polynomial function $\varphi_{\sigma_\Gamma}$ decomposes as a product:
 \begin{equation}
 \varphi_{\sigma_\Gamma} =l \cdot  \varphi_{\sigma_{ T_{g-h}}} \cdot  \varphi_{\sigma_{\Gamma \smallsetminus T_{g-h}}} ,
 \end{equation}
where $l$ denotes the length of the edge between the point of attachment of the contracted tree and the next vertex to the right (see Figure \ref{fig:lowerg}).

On the component $C_h$, the pull-back $st_{\trop}^\star(l)$ corresponds to the cohomology class $ev_1^\star([pt])$, fixing a point on the target where the component must contract. The class associated to the piecewise polynomial $st_{\trop}^\star(\sum m_\Gamma \ \varphi_{\sigma_{\Gamma \smallsetminus T_{g-h}}})$ agrees with $x_1 \cdot br_h^\star([pt])$, for the genus $h$ branch morphism on the left factor of $C_h$. Note the additional factor $x_1$ comes from the fact that one compact edge with expansion factor $x_1$ is formed by attaching $T_{g-h}$. Finally $st_{\trop}^\star( \varphi_{\sigma_{ T_{g-h}}})$ gives the intersection of the stratum $\Delta_{T_{g-h}}$ with the virtual class of $C_h$. It follows that $[C_h]^{vir} \cdot \Delta_{T_{g-h}} = {\lambda^{R}_{g-h}}_{|\Delta_{T_{g-h}}}$, where the superscript $R$ denotes that the $\lambda$ class is pulled-back from the right factor of $C_h$. 
In conclusion, we have
\begin{equation}
 \LDR \cdot \sum_{\sigma_\Gamma}m_\Gamma\  \Delta_{\sigma_\Gamma} =  st_{trop}^\star \varphi_{\sigma_\Gamma} =\deg\left(x_1 \cdot br_h^\star([pt]) \boxtimes {\lambda^{R}_{g-h}}_{|\Delta_{T_{g-h}}}\right),
\end{equation}
where the class in parenthesis is a class in $ \overline{\mathcal{M}}^{\sim}_{h}(\PP^1, \mathbf{x}) \times  \overline{\mathcal{M}}_{g-h, 1}$. The class pulled back from the left factor has degree $H_h(\mathbf{x})$.
From the isomorphism $$\Delta_{T_{g-h}}\cong \prod_{i=1}^{g-h} \overline{\mathcal{M}}_{1,1}/|\Aut(|T_{g-h}|),$$
the fact that the class $\lambda_{g-h}$ splits as the product of $\lambda_1$'s on each of the factors, and that $\lambda_1$ has degree $1/24$ on $ \overline{\mathcal{M}}_{1,1}$, one may conclude:
\begin{equation}
 \LDR \cdot br_{h<g}(\mathbf x, 0) = H_h(\mathbf{x}).
\end{equation}
\end{proof}

\section{Leaky and pluricanonical numbers}\label{sec:lpn}

In this section we generalize the story of double Hurwitz numbers. Moduli spaces of tropical leaky covers  give rise in a natural way to an enumerative problem analogous to tropical Hurwitz numbers. \rcolor{ Through the perspective of logarithmic geometry introduced in Section \ref{sec:tcpp} we see that the natural branch polynomial arising from tropical geometry yields a cohomology class on the moduli space of curves, whose evaluation against the logarithmic pluricanonical double ramification cycle  $\mathsf{DR}^\lightening(\mathbf x,k)$ coincides with the count of tropical leaky covers, thus providing an algebraic intersection theoretic counterpart to the tropical enumerative problem. 

%{\color{blue} The following blue part can just be deleted, true?

%With our definition of tropical leaky covers this algebraic counterpart always involves  moduli spaces of meromorphic sections of the pluricanonical bundle with specified divisors of zeroes and poles (and such that the divisor of poles is non-empty, i.e. the vector $\mathbf{x}$ must contain some negative entries). To include the case of holomorphic sections one should modify the set-up of leaky covers to allow for  genus one {\it cul de sacs}, i.e. one-valent vertices of genus one with the edge to their left. While this is an interesting story that deserves to be investigated as it might yield interesting connections to integrable systems (\cite{Shad:pc}), in this work we choose our generalization to stick as close as possible to the case of Hurwitz numbers.}
%It follows that the tropical counts of leaky covers has an algebraic counterpart on the moduli space of curves.
As before we adopt the simplifying convention of referring to $1$-leaky covers simply as leaky covers. }

\subsection{Tropical leaky numbers}\label{sec: leaky-numbers-tropically}

Consider the moduli spaces of leaky covers $\leak$  introduced in Definition \ref{def:msleaky}. As in \eqref{eq:tropbr},  there is a natural \emph{tropical branch morphism}.

\begin{definition}[Number of leaky covers $\leakyn$]
Fix a genus $g$ and a degree $\mathbf{x}$. We define the \emph{number of leaky covers of genus $g$ and degree $\mathbf{x}$}, denoted  $\leakyn$, to be the degree of the tropical branch morphism $br_{\trop}: \leak \rightarrow \mathsf{tEx}.$

Here, the degree is the weighted number of preimages of a point in the interior of the $r-1$-dimensional cone of $\mathsf{tEx}$. 

 As we see below, it is possible for our leaky covers to have $1$-valent vertices of genus $1$. For those, we define the weight of the corresponding cone to obtain a factor of $-\frac{1}{24}$. This choice is motivated by the correspondence theorem.
 
A preimage is weighted by the product of the weight of the polyhedron in which it lives with the index of the image lattice of this polyhedron under $br_{\trop}$ in the natural lattice of $\mathsf{tEx}$.
\end{definition}
This definition of degree is analogous to Definition 5.14 in \cite{CJM1}, see also Definition 4.1 in \cite{GKM07}.

\begin{example}
The leaky cover in Figure \ref{fig-leakycoverex} contributes to the number $L_{1}((5,-1,-1),1)$ with the weight $1$ times the index of the image lattice of this polyhedron under $br_{\trop}$. (The size of its automorphism group is $1$.)
 By Remark 5.19 of \cite{CJM1}, this product equals the index of the linear map given by the square matrix in which we combine the equations for the weight with the evaluations. By Example \ref{ex-polyhedron}, the polyhedron of the combinatorial type of Figure \ref{fig-leakycoverex} is surrounded by $\RR^4$, thus our matrix is of size $4$. Two of its lines are given by the equation of the cycle, $(3,1,-1,-1)$, and the equation that the middle vertices of $\Gamma$ have the same image, $(0,1,0,-1)$ (see Example \ref{ex-polyhedron}). The second vertex of $\PP^1_{\trop}$  is at distance $l_3$ from the first, the third at distance $l_4$ from the second.
Thus the square matrix to consider is
$$\left(\begin{matrix}
3&1&-1&-1&\\0&1&0&-1\\0&0&1&0\\0&0&0&1
\end{matrix}\right).
$$
The index of the lattice of the image of the linear map defined by this matrix equals the absolute value of its determinant, which is $3$. Thus the leaky cover of Figure \ref{fig-leakycoverex} contributes with weight $3$ to the count of  $L_{1}((5,-1,-1),1)$.

\end{example}

\begin{lemma}
The degree of the tropical branch morphism (and thus, the number of leaky covers of genus $g$ and degree $\mathbf{x}$) is well-defined. All leaky covers in the preimage of a point in the interior of the $r-1$-dimensional cone of $\mathsf{tEx}$ are \emph{one-and-trivalent covers}, i.e.\ each preimage of one of the $r$ vertices of %$\PP^1_{\trop}$ 
\rcolor{the metric line graph $T$}\renzo{will be replacing these as well} contains precisely one vertex which is not of genus $0$ and $2$-valent, and that is of genus $0$ and $3$-valent or of genus $1$ and $1$-valent.
\end{lemma}
\begin{proof}
Consider a metric line graph $T$ with $r$ vertices. Given a leaky cover of a certain combinatorial type, we can just vary the edge lengths of $\Gamma$ to obtain a leaky cover of another metric line graph $T'$ with $r$ vertices. We therefore obtain a bijection between the preimages under the tropical branch morphism of $T$ and of $T'$, which preserves the combinatorial type. As the weight only depends on the combinatorial type, the weighted count of preimages of $T$ coincides with the weighted count of preimages of $T'$ and accordingly, the degree of the tropical branch morphism is well-defined. 

The statement about the one-and-trivalent covers follows from counting (virtual) dimensions: each vertex must be leaking exactly $k$, that is, $2g(v)-2+\val(v)=1$ which is fulfilled if either $\val(v)=3$ and $g(v)=0$ or $\val(v)=1$ and $g(v)=1$.
\end{proof}

\begin{lemma}
Let $\pi:\Gamma\rightarrow T$ be a preimage under the tropical branch morphism of a point in the interior of the image of $\leak $ in $\mathsf{tEx}$. Then $\pi$ contributes with the product of the expansion factors of the bounded edges  of $\Gamma$ times one over  the size of the automorphism group of $\pi$ times $-\frac{1}{24}$ for every $1$-valent vertex of genus $1$ to the count of leaky covers.
\end{lemma}

\begin{proof}
The proof is parallel to that of~\cite[Lemma~5.26]{CJM1}, where we show that a tropical cover counts towards the corresponding tropical double Hurwitz number  with the product of the expansion factors of its bounded edges times the size of its automorphism group.
As in Example \ref{ex-polyhedron}, we express the weight of the polyhedron times the index of the image lattice of the evaluation map as the determinant of a square matrix whose lines are given by cycle equations, equations that vertices of $\Gamma$ have the same image as required by the cover, and the evaluations of vertices. This is as in Remark 5.17 and 5.19 in  \cite{CJM1}. Then, we pick an order of the coordinates given by the lengths of the bounded edges which allows us to obtain a triangular form for the square matrix. The entries on the diagonal are precisely the expansion factors of the bounded edges. This is as in~\cite[Section~5.25]{CJM1}.
The factors of $-\frac{1}{24}$ for each $1$-valent vertex of genus $1$ have to be added by definition of the weight of the corresponding cones.
\end{proof}

\begin{example}
Figure \ref{fig-leakycovercountex} shows the count of leaky covers of degree $(5,-1,-1)$ and genus $1$. The third, fourth and seventh pictures have to be counted twice, as they allow two versions of labeling their ends. The second and fourth pictures have an automorphism group of size $2$, as the two edges of the cycle can be permuted. The last three pictures have a factor of $-\frac{1}{24}$ because of the genus $1$ vertex.
Thus the total count equals $9+6+6+3-\frac{3}{24}-\frac{3}{24}-\frac{6}{24}=\frac{47}{2}$. 
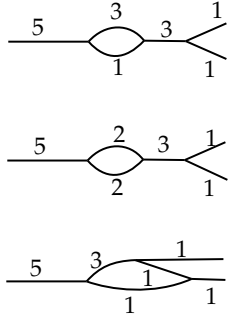
\begin{figure}

\tikzset{every picture/.style={line width=0.75pt}} %set default line width to 0.75pt        

\tikzset{every picture/.style={line width=0.75pt}} %set default line width to 0.75pt        

\begin{tikzpicture}[x=0.75pt,y=0.75pt,yscale=-1,xscale=1]
%uncomment if require: \path (0,784); %set diagram left start at 0, and has height of 784

%Curve Lines [id:da7168840405088402] 
\draw    (250.67,360.67) .. controls (259.92,352.31) and (281.77,351.53) .. (290.67,360.67) ;
%Straight Lines [id:da9489090508529316] 
\draw    (210.67,360.33) -- (250.67,360.67) ;
%Shape: Ellipse [id:dp9510269496728492] 
\draw  [fill={rgb, 255:red, 0; green, 0; blue, 0 }  ,fill opacity=1 ] (250.44,439.47) .. controls (250.44,438.71) and (249.87,438.1) .. (249.17,438.1) .. controls (248.47,438.1) and (247.9,438.71) .. (247.9,439.47) .. controls (247.9,440.22) and (248.47,440.84) .. (249.17,440.84) .. controls (249.87,440.84) and (250.44,440.22) .. (250.44,439.47) -- cycle ;
%Straight Lines [id:da49153682305692215] 
\draw    (160.67,360.33) -- (210.67,360.33) ;
%Curve Lines [id:da9793035230130769] 
\draw    (250.67,360.67) .. controls (260.47,372.73) and (284.87,369.93) .. (290.67,360.67) ;
%Straight Lines [id:da21710379645642974] 
\draw    (290.67,360.67) -- (320.67,360.67) ;
%Curve Lines [id:da6043636337369442] 
\draw    (210.67,360.33) .. controls (230.83,327.3) and (259.17,328.97) .. (310.67,330.33) ;
%Straight Lines [id:da7920352224003738] 
\draw    (160,150) -- (210,150) ;
%Curve Lines [id:da1045457424421311] 
\draw    (209.33,150.22) .. controls (218.59,141.86) and (240.44,141.08) .. (249.33,150.22) ;
%Curve Lines [id:da7237727873242853] 
\draw    (209.33,150.22) .. controls (219.13,162.29) and (243.53,159.49) .. (249.33,150.22) ;
%Straight Lines [id:da8759446566214449] 
\draw    (249.33,150.22) -- (289.5,149.97) ;
%Straight Lines [id:da4284343966770391] 
\draw    (289.5,149.97) -- (318.83,140.97) ;
%Straight Lines [id:da9221374367013978] 
\draw    (289.5,149.97) -- (318.17,160.3) ;
%Straight Lines [id:da9821238260720511] 
\draw    (160.33,216.67) -- (210.33,216.67) ;
%Curve Lines [id:da2874550729091807] 
\draw    (209.67,216.89) .. controls (218.92,208.53) and (240.77,207.75) .. (249.67,216.89) ;
%Curve Lines [id:da0845850385954694] 
\draw    (209.67,216.89) .. controls (219.47,228.96) and (243.87,226.16) .. (249.67,216.89) ;
%Straight Lines [id:da742076591269771] 
\draw    (249.67,216.89) -- (289.83,216.63) ;
%Straight Lines [id:da4117312022729528] 
\draw    (289.83,216.63) -- (319.17,207.63) ;
%Straight Lines [id:da040122132867710714] 
\draw    (289.83,216.63) -- (318.5,226.97) ;
%Straight Lines [id:da8214681934825493] 
\draw    (160,291) -- (210,291) ;
%Curve Lines [id:da6299973464218965] 
\draw    (210,291) .. controls (219.26,282.64) and (232.5,278.97) .. (248.5,279.3) ;
%Curve Lines [id:da9186629452650147] 
\draw    (210,291) .. controls (217.83,298.63) and (276.83,294.63) .. (289.17,289.97) ;
%Straight Lines [id:da5701750482794028] 
\draw    (248.5,279.3) -- (289.17,289.97) ;
%Straight Lines [id:da7820262796255996] 
\draw    (248.5,279.3) -- (319.5,278.3) ;
%Straight Lines [id:da3146286564575055] 
\draw    (289.17,289.97) -- (318.5,290.63) ;
%Straight Lines [id:da5610245512673635] 
\draw    (160.67,429.67) -- (210.67,429.67) ;
%Straight Lines [id:da04601992879666417] 
\draw    (210.67,429.67) -- (249.17,439.47) ;
%Curve Lines [id:da6814340268969252] 
\draw    (210.67,429.67) .. controls (230.5,418.47) and (237.17,413.8) .. (289.17,409.8) ;
%Straight Lines [id:da4617158880473129] 
\draw    (289.17,409.8) -- (319.83,401.47) ;
%Straight Lines [id:da37931536761355755] 
\draw    (289.17,409.8) -- (318.83,418.8) ;
%Shape: Ellipse [id:dp4939436414948164] 
\draw  [fill={rgb, 255:red, 0; green, 0; blue, 0 }  ,fill opacity=1 ] (290.5,510.47) .. controls (290.5,509.71) and (289.93,509.1) .. (289.23,509.1) .. controls (288.53,509.1) and (287.96,509.71) .. (287.96,510.47) .. controls (287.96,511.22) and (288.53,511.84) .. (289.23,511.84) .. controls (289.93,511.84) and (290.5,511.22) .. (290.5,510.47) -- cycle ;
%Straight Lines [id:da1064057289513638] 
\draw    (161,501.33) -- (211,501.33) ;
%Straight Lines [id:da07217097429009511] 
\draw    (211,501.33) -- (290.5,510.47) ;
%Curve Lines [id:da4305312774711011] 
\draw    (211,501.33) .. controls (230.83,490.13) and (232.5,492.47) .. (249.17,486.8) ;
%Straight Lines [id:da3641985633625858] 
\draw    (249.17,486.8) -- (320.17,473.13) ;
%Straight Lines [id:da3013217417997742] 
\draw    (249.17,486.8) -- (319.17,490.47) ;
%Shape: Ellipse [id:dp33108758489813206] 
\draw  [fill={rgb, 255:red, 0; green, 0; blue, 0 }  ,fill opacity=1 ] (291.17,589.1) .. controls (291.17,588.34) and (290.6,587.72) .. (289.9,587.72) .. controls (289.19,587.72) and (288.63,588.34) .. (288.63,589.1) .. controls (288.63,589.85) and (289.19,590.47) .. (289.9,590.47) .. controls (290.6,590.47) and (291.17,589.85) .. (291.17,589.1) -- cycle ;
%Straight Lines [id:da9005634070066763] 
\draw    (160,569.33) -- (210,569.33) ;
%Straight Lines [id:da5425135862819612] 
\draw    (210,569.33) -- (249.5,575.97) ;
%Curve Lines [id:da604345001758313] 
\draw    (210,569.33) .. controls (229.83,558.13) and (303.83,547.3) .. (320.5,541.63) ;
%Straight Lines [id:da2372871821796637] 
\draw    (249.5,575.97) -- (320.5,562.3) ;
%Straight Lines [id:da8179684381779154] 
\draw    (249.5,575.97) -- (289.9,589.1) ;

% Text Node
\draw (225.17,337) node  [font=\footnotesize] [align=left] {\begin{minipage}[lt]{12.34pt}\setlength\topsep{0pt}
$\displaystyle 3$
\end{minipage}};
% Text Node
\draw (309.17,354.33) node  [font=\footnotesize] [align=left] {\begin{minipage}[lt]{12.34pt}\setlength\topsep{0pt}
$\displaystyle 1$
\end{minipage}};
% Text Node
\draw (274.39,347.11) node  [font=\footnotesize] [align=left] {\begin{minipage}[lt]{12.34pt}\setlength\topsep{0pt}
$\displaystyle 1$
\end{minipage}};
% Text Node
\draw (278.72,381.44) node  [font=\footnotesize] [align=left] {\begin{minipage}[lt]{12.34pt}\setlength\topsep{0pt}
$\displaystyle 1$
\end{minipage}};
% Text Node
\draw (302.5,324) node  [font=\footnotesize] [align=left] {\begin{minipage}[lt]{12.34pt}\setlength\topsep{0pt}
$\displaystyle 1$
\end{minipage}};
% Text Node
\draw (194.83,354) node  [font=\footnotesize] [align=left] {\begin{minipage}[lt]{12.34pt}\setlength\topsep{0pt}
$\displaystyle 5$
\end{minipage}};
% Text Node
\draw (187.83,143) node  [font=\footnotesize] [align=left] {\begin{minipage}[lt]{12.34pt}\setlength\topsep{0pt}
$\displaystyle 5$
\end{minipage}};
% Text Node
\draw (267.83,144.22) node  [font=\footnotesize] [align=left] {\begin{minipage}[lt]{12.34pt}\setlength\topsep{0pt}
$\displaystyle 3$
\end{minipage}};
% Text Node
\draw (233.72,135.67) node  [font=\footnotesize] [align=left] {\begin{minipage}[lt]{12.34pt}\setlength\topsep{0pt}
$\displaystyle 3$
\end{minipage}};
% Text Node
\draw (237.39,171) node  [font=\footnotesize] [align=left] {\begin{minipage}[lt]{12.34pt}\setlength\topsep{0pt}
$\displaystyle 1$
\end{minipage}};
% Text Node
\draw (311.5,139) node  [font=\footnotesize] [align=left] {\begin{minipage}[lt]{12.34pt}\setlength\topsep{0pt}
$\displaystyle 1$
\end{minipage}};
% Text Node
\draw (309.83,169.33) node  [font=\footnotesize] [align=left] {\begin{minipage}[lt]{12.34pt}\setlength\topsep{0pt}
$\displaystyle 1$
\end{minipage}};
% Text Node
\draw (187.83,209.33) node  [font=\footnotesize] [align=left] {\begin{minipage}[lt]{12.34pt}\setlength\topsep{0pt}
$\displaystyle 5$
\end{minipage}};
% Text Node
\draw (268.17,210.89) node  [font=\footnotesize] [align=left] {\begin{minipage}[lt]{12.34pt}\setlength\topsep{0pt}
$\displaystyle 3$
\end{minipage}};
% Text Node
\draw (233.72,201.67) node  [font=\footnotesize] [align=left] {\begin{minipage}[lt]{12.34pt}\setlength\topsep{0pt}
$\displaystyle 2$
\end{minipage}};
% Text Node
\draw (233.06,240.33) node  [font=\footnotesize] [align=left] {\begin{minipage}[lt]{12.34pt}\setlength\topsep{0pt}
$\displaystyle 2$
\end{minipage}};
% Text Node
\draw (311.83,205.67) node  [font=\footnotesize] [align=left] {\begin{minipage}[lt]{12.34pt}\setlength\topsep{0pt}
$\displaystyle 1$
\end{minipage}};
% Text Node
\draw (309.83,238.33) node  [font=\footnotesize] [align=left] {\begin{minipage}[lt]{12.34pt}\setlength\topsep{0pt}
$\displaystyle 1$
\end{minipage}};
% Text Node
\draw (332.17,281.33) node  [font=\footnotesize] [align=left] {\begin{minipage}[lt]{12.34pt}\setlength\topsep{0pt}
$\displaystyle 1$
\end{minipage}};
% Text Node
\draw (330.83,298.67) node  [font=\footnotesize] [align=left] {\begin{minipage}[lt]{12.34pt}\setlength\topsep{0pt}
$\displaystyle 1$
\end{minipage}};
% Text Node
\draw (261.17,289.33) node  [font=\footnotesize] [align=left] {\begin{minipage}[lt]{12.34pt}\setlength\topsep{0pt}
$\displaystyle 1$
\end{minipage}};
% Text Node
\draw (233.83,306.33) node  [font=\footnotesize] [align=left] {\begin{minipage}[lt]{12.34pt}\setlength\topsep{0pt}
$\displaystyle 1$
\end{minipage}};
% Text Node
\draw (230.83,275.22) node  [font=\footnotesize] [align=left] {\begin{minipage}[lt]{12.34pt}\setlength\topsep{0pt}
$\displaystyle 3$
\end{minipage}};
% Text Node
\draw (182.5,284) node  [font=\footnotesize] [align=left] {\begin{minipage}[lt]{12.34pt}\setlength\topsep{0pt}
$\displaystyle 5$
\end{minipage}};
% Text Node
\draw (331.83,404.67) node  [font=\footnotesize] [align=left] {\begin{minipage}[lt]{12.34pt}\setlength\topsep{0pt}
$\displaystyle 1$
\end{minipage}};
% Text Node
\draw (329.5,423.67) node  [font=\footnotesize] [align=left] {\begin{minipage}[lt]{12.34pt}\setlength\topsep{0pt}
$\displaystyle 1$
\end{minipage}};
% Text Node
\draw (261.83,446) node  [font=\footnotesize] [align=left] {\begin{minipage}[lt]{12.34pt}\setlength\topsep{0pt}
$\displaystyle 1$
\end{minipage}};
% Text Node
\draw (252.17,409.33) node  [font=\footnotesize] [align=left] {\begin{minipage}[lt]{12.34pt}\setlength\topsep{0pt}
$\displaystyle 3$
\end{minipage}};
% Text Node
\draw (189.5,422.67) node  [font=\footnotesize] [align=left] {\begin{minipage}[lt]{12.34pt}\setlength\topsep{0pt}
$\displaystyle 5$
\end{minipage}};
% Text Node
\draw (332.17,476.33) node  [font=\footnotesize] [align=left] {\begin{minipage}[lt]{12.34pt}\setlength\topsep{0pt}
$\displaystyle 1$
\end{minipage}};
% Text Node
\draw (329.83,495.33) node  [font=\footnotesize] [align=left] {\begin{minipage}[lt]{12.34pt}\setlength\topsep{0pt}
$\displaystyle 1$
\end{minipage}};
% Text Node
\draw (262.17,517.67) node  [font=\footnotesize] [align=left] {\begin{minipage}[lt]{12.34pt}\setlength\topsep{0pt}
$\displaystyle 1$
\end{minipage}};
% Text Node
\draw (226.5,485.33) node  [font=\footnotesize] [align=left] {\begin{minipage}[lt]{12.34pt}\setlength\topsep{0pt}
$\displaystyle 3$
\end{minipage}};
% Text Node
\draw (189.83,494.33) node  [font=\footnotesize] [align=left] {\begin{minipage}[lt]{12.34pt}\setlength\topsep{0pt}
$\displaystyle 5$
\end{minipage}};
% Text Node
\draw (331.17,544.33) node  [font=\footnotesize] [align=left] {\begin{minipage}[lt]{12.34pt}\setlength\topsep{0pt}
$\displaystyle 1$
\end{minipage}};
% Text Node
\draw (275.17,594.33) node  [font=\footnotesize] [align=left] {\begin{minipage}[lt]{12.34pt}\setlength\topsep{0pt}
$\displaystyle 1$
\end{minipage}};
% Text Node
\draw (329.83,566.67) node  [font=\footnotesize] [align=left] {\begin{minipage}[lt]{12.34pt}\setlength\topsep{0pt}
$\displaystyle 1$
\end{minipage}};
% Text Node
\draw (245.83,584.33) node  [font=\footnotesize] [align=left] {\begin{minipage}[lt]{12.34pt}\setlength\topsep{0pt}
$\displaystyle 3$
\end{minipage}};
% Text Node
\draw (188.83,562.33) node  [font=\footnotesize] [align=left] {\begin{minipage}[lt]{12.34pt}\setlength\topsep{0pt}
$\displaystyle 5$
\end{minipage}};

\end{tikzpicture}

\caption{The count of leaky covers of degree $(5,-1,-1)$ and genus $1$.} \label{fig-leakycovercountex}
\end{figure}
\end{example}

We now introduce the analogous concept to the monodromy graphs from Definition \ref{def-mongraph}.

\rcolor{\begin{definition}
\label{def-leakygraph}
For fixed $g$ and $\mathbf{x}=(x_1,\ldots,x_n)$, a graph $\Gamma$ is a \emph{k-leaky graph} if it satisfies conditions $(1)$, $(3)$ and $(4)$ in Definition \ref{def-mongraph}, with $(2)$ changed as follows, and as well as the following \emph{leaky condition} replacing $(5)$:
\begin{description}
\item[(2')] Vertices which are not $2$-valent can be $1$-valent of genus $1$, or $3$-valent.
\item[(5')] Every bounded edge $e$ of $\Gamma$ is equipped with an expansion factor $w(e)\in \NN$. At each $3$-or $1$-valent vertex the sum of all expansion factors of  incoming  edges equals the sum of the expansion factors of all outgoing edges plus $k$.
\end{description}
%\begin{enumerate}
%\item $\Gamma$ is a connected, genus $g$, directed, explicit graph. (Here, we do not allow positive genus at vertices.)
%\item $\Gamma$ has $n$ ends which are directed inward, and labeled by the expansion factors $x_1,\ldots,x_n$. If $x_i >0$, we say it is an \emph{in-end}, otherwise it is an \emph{out-end}.
%\item All inner vertices of $\Gamma$ are $3$-valent.
%\item After reversing the orientation of the out-ends, $\Gamma$ does not have sinks or sources\footnote{We do not consider leaves to be sinks or sources.}.
%\item The inner vertices are ordered compatibly with the partial ordering induced by the directions of the edges.
%\item Every bounded edge $e$ of the graph is equipped with an expansion factor $w(e)\in \NN$. These satisfy the \emph{leaky condition} at each inner vertex: the sum of all expansion factors of  incoming  edges equals the sum of the expansion factors of all outgoing edges plus $k$.
%\end{enumerate}
\end{definition}}\renzo{This seems more efficient than repeating all (1)-(5) again. Also, it highlights what the difference is.}

As in the case for monodromy graphs, $k$-leaky graphs are naturally in bijection with maximal dimensional cones of the moduli space of $k$-leaky covers, whose general element contain no contracting subgraph. The following theorem is an immediate consequence:
%Given a trivalent leaky cover, forget the $2$-valent vertices and straighten the edges. Order the remaining vertices according to their image. Direct the edges such that they point from the vertex with the smaller image to the vertex with the larger image (except for negative ends, which we orient inwards). Forget the lengths of the edges. What we obtain is a leaky graph by definition \ref{def-leakygraph}.

%Vice versa, given a leaky graph, we can reverse the procedure above to obtain a leaky cover.

\begin{theorem}\label{thm: leaky-formula}
The number $\leakyn$ equals the sum over all leaky graphs $\Gamma$, where each is counted with the product of the expansion factors of its bounded edges times vertex multiplicities$\mult_v$  which we declare to be equal to $1$ if $v$ is $3$-valent and equal to $-\frac{1}{24}$ if $v$ is $1$-valent of genus $1$:
$$\leakyn = \sum_{\Gamma} \frac{1}{|\Aut(\Gamma)|} \prod_e \omega(e)\prod_v \mult_v.$$
\end{theorem}

We introduce the leaky branch polynomial.

\begin{definition}
The {\it $k$-leaky branch polynomial} $br_g(\mathbf x, k)$ is
$$
br_g(\mathbf x, k):=\sum \varphi_{\sigma_\Gamma},
$$
where $\varphi_{\sigma_\Gamma}$ is defined as in \eqref{eq:bpolym}, and the sum runs over all cones corresponding to $k$-leaky graphs.
\end{definition}

We conclude this section with the correspondence theorem between leaky  numbers and the degree of a zero dimensional cycle obtained by evaluating the leaky branch polynomial on the pluricanonical double ramification cycle.
\begin{customthm}{B}\label{prop:lhntc}
With notation as developed throughout, we have
\begin{equation}
\leakyn =\deg\left( [\mathsf{DR}^\lightening(\mathbf x,k)] \cdot  br_g(\mathbf x, k)\right).
\end{equation}
\end{customthm}

The proof will pass through the geometry of expanded degenerations in the normal crossings setting, and we recall the required notions. Further details may be found in~\cite{R19}. 

\subsubsection{The tropical map} Let $\Gamma$ be a tropical curve with positive integer edge lengths corresponding to an integer point of $\mathsf{DR}_g^{\trop}(\mathbf x, k)$. There is an associated piecewise linear function, unique up to translation, given by 
\[
t: \Gamma \to \RR,
\]
which is a putative tropicalization of a section of the $k$-pluricanonical line bundle of curve with the given profile of zeros and poles $\mathbf x$. We assume that the tropical curve $\Gamma$ is trivalent, with all vertices having genus $0$, and furthermore that $t$ has positive slope on all bounded edges of $\Gamma$. 

%\begin{lemma}
%Let $C$ be an algebraic curve together with a fixed isomorphism between the dual graph fo $C$ and the underlying graph of $\Gamma$. The set of logarithmic curves $C$ over $\spec(\NN\to\CC)$ whose tropicalization is $\Gamma$ is a torsor under $(\mathbb C^\star)^g$. 
%\end{lemma}

\subsubsection{Projective bundle} Let $C$ be a logarithmically smooth curve over the standard logarithmic point $\spec(\NN\to\CC)$, whose tropicalization is isomorphic to $\Gamma$; note that we view it as a metric graph, rather than as a family over $\RR_{\geq 0}$, by slicing at height $1$. Consider the projective bundle $\mathbb P(\mathcal O_C\oplus \omega_C^{\mathsf{log}})$ completing the total space of the logarithmic cotangent line bundle. We denote it by $\mathbb P$ for brevity. The surface is equipped with a logarithmic structure of rank $2$, by adding the $0$ and $\infty$ divisors of the bundle to the pullback logarithmic structure on $C$. 

\noindent
\subsubsection{Tropical subdivisions} The tropicalization of $\mathbb P$ is a polyhedral decomposition of $\Gamma\times \RR$ with a distinguished integral lattice structure. We will use two refinements of the polyhedral structure, including passing to a finite index sublattice. First, consider the function
\[
t: \Gamma \to \RR,
\]
and refine the polyhedral structure by declaring the image of every vertex to be a vertex of $\RR$ and then further that the preimage of each vertex of $\RR$ is a vertex of $\Gamma$. The result is a map of metric graphs
\[
t: \rcolor{\Gamma^+}\to T.
\]\renzo{I added the $+$ here.}
The first polyhedral structure of interest is $\Gamma^+\times T$. The cells in the decomposition are either vertices, edges, or squares. We may factorize the section through the graph:
\[
 \Gamma^+\to \Gamma^+\times T\to T.
\]
We note that set theoretically the map $t$ is unchanged, we have merely refined the polyhedral structure. We denote $ \Gamma^+\times T$ by $\Delta$ for compactness of notation in the following discussion. Note that the vertices of $\Delta$ may have rational coordinates that are not integral. However, the polyhedral complex can be uniformly scaled. Precisely, to scale by a positive integer, we uniformly multiply all the edges in $\Gamma$ by this length; it admits a unique map to $\RR$ with the given slopes and leaky balancing condition. We can then repeat the construction. There is a well-defined smallest positive scaling factor such that the construction outputs $\Delta$ with all vertices having integer coordinates. If we choose coordinates on the cones, this is the least common multiple of all the denominators that appear in lowest terms of the coordinates. We will use this below, and so refer to this as the {\it base order}. 

In order to understand the need for the second polyhedral structure, note that if $e$ is an edge of $\Gamma$, then $t(e)$  necessarily connects two vertices of $\Delta$; however, this means that the image $t(e)$  is \textit{not} a cell of $\Delta$. Geometrically, this  corresponds to the situation where nodes and marked points on a curve map to strata of the target of smaller dimension than expected. The situation is remedied by a further refinement. Let $\Pi_t$ be polyhedral refinement of $\Delta$ obtained by subdividing along the image under $t$ of all edges $e$ in $\Gamma^+$. We still have a factorization
\[
t:\Gamma^+ \to \Pi_t\to T.
\]
The situation may be visualized as in Figure~\ref{fig: tropical-subdivisions} below. 

\begin{figure}[h!]
\tikzset{every picture/.style={line width=0.75pt}} %set default line width to 0.75pt

\begin{tikzpicture}[x=0.75pt,y=0.75pt,yscale=-1,xscale=1]
%uncomment if require: \path (0,784); %set diagram left start at 0, and has height of 784

%Straight Lines [id:da7618935537007738]
\draw    (130,50) -- (170,50) ;
%Straight Lines [id:da5041310505717755]
\draw    (130,190) -- (170,190) ;
%Straight Lines [id:da06785309858133348]
\draw    (170,190) -- (170,50) ;
%Straight Lines [id:da772622841359318]
\draw    (130,190) -- (130,50) ;
%Curve Lines [id:da016463691750676857]
\draw    (170,50) .. controls (187.2,68.67) and (203.2,68.67) .. (220,50) ;
%Curve Lines [id:da744035336024519]
\draw    (170,190) .. controls (187.2,208.67) and (203.2,208.67) .. (220,190) ;
%Curve Lines [id:da4840978777078948]
\draw    (170,50) .. controls (190.4,40.67) and (198.8,39.87) .. (220,50) ;
%Curve Lines [id:da7967520894565864]
\draw  [dash pattern={on 0.84pt off 2.51pt}]  (170,190) .. controls (190.4,180.67) and (198.8,179.87) .. (220,190) ;
%Straight Lines [id:da2657064849094356]
\draw    (220,190) -- (220,50) ;
%Straight Lines [id:da24479427509606022]
\draw    (260,190) -- (260,50) ;
%Straight Lines [id:da38187861867033734]
\draw    (220,50) -- (260,50) ;
%Straight Lines [id:da9807521067408417]
\draw    (220,190) -- (260,190) ;
%Straight Lines [id:da7529414103441759]
\draw    (260,50) -- (290,40) ;
%Straight Lines [id:da10772532457820094]
\draw  [dash pattern={on 0.84pt off 2.51pt}]  (260,190) -- (290,180) ;
%Straight Lines [id:da3613141092710985]
\draw    (260,50) -- (300,60) ;
%Straight Lines [id:da1914331964219329]
\draw    (260,190) -- (300,200) ;
%Straight Lines [id:da7768544760772145]
\draw  [dash pattern={on 0.84pt off 2.51pt}]  (290,180) -- (290,40) ;
%Straight Lines [id:da3622796316325235]
\draw    (300,200) -- (300,60) ;
%Straight Lines [id:da19868690650659981]
\draw    (290,40) -- (290,60) ;
%Straight Lines [id:da8171693834951224]
\draw    (150,190) -- (170,140) ;
%Curve Lines [id:da4890121073799151]
\draw    (170,140) .. controls (192.8,150.47) and (205.6,148.87) .. (220,130) ;
%Curve Lines [id:da6387074085202329]
\draw  [dash pattern={on 0.84pt off 2.51pt}]  (170,140) .. controls (180.8,127.67) and (198.8,119.87) .. (220,130) ;
%Straight Lines [id:da2223608727745724]
\draw    (220,130) -- (260,110) ;
%Straight Lines [id:da06033757885766544]
\draw    (260,110) -- (300,100) ;
%Straight Lines [id:da6941689351774052]
\draw  [dash pattern={on 0.84pt off 2.51pt}]  (260,110) -- (290,90) ;

% Text Node
\draw (149,154) node [anchor=north west][inner sep=0.75pt] [font=\footnotesize] [align=left] {$\displaystyle 5$};
% Text Node
\draw (201,144) node [anchor=north west][inner sep=0.75pt] [font=\footnotesize] [align=left] {$\displaystyle 1$};
% Text Node
\draw (191,112) node [anchor=north west][inner sep=0.75pt] [font=\footnotesize] [align=left] {$\displaystyle 3$};
% Text Node
\draw (239,104) node [anchor=north west][inner sep=0.75pt] [font=\footnotesize] [align=left] {$\displaystyle 3$};
% Text Node
\draw (279,104) node [anchor=north west][inner sep=0.75pt] [font=\footnotesize] [align=left] {$\displaystyle 1$};
% Text Node
\draw (271,84) node [anchor=north west][inner sep=0.75pt] [font=\footnotesize] [align=left] {$\displaystyle 1$};
\end{tikzpicture}
\caption{The graph of a leaky tropical curve.}\label{fig: tropical-subdivisions}
\end{figure}
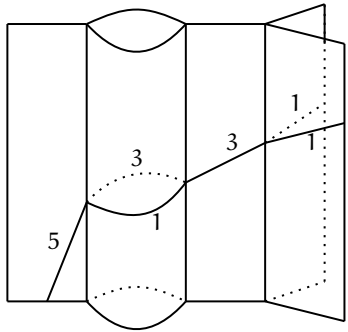

\subsubsection{Geometric expansions and logarithmic lifts} Given a polyhedral decomposition --- for us it will be either $\Pi_t$ or $\Delta$ --- of $\Gamma\times \RR$, there is an associated \textit{expansion} of the total space of $\mathbb P$ over curve $C$, \rcolor{which in our relevant cases will be denoted $\PP(\Pi_t)$ or $\PP(\Delta)$}. The expansion can be viewed as the special fiber of a $1$-parameter degeneration, or as we do here, a logarithmic scheme over $\Spec(\NN\to \CC)$. The {\it base order} mentioned above is relevant here. If the base order is $b$, then we perform a generalized root construction to this scheme $\Spec(\NN\to \CC)$ given by passing to the index $b$ sublattice generated by $b\in \NN$. The expansion of $\mathbb P$ associated to the polyhedral decomposition is defined over this new logarithmic point. See~\cite[Section~6.2]{BHPSS} for a similar discussion, or~\cite{AK00,Mol16} for a discussion about ramified base changes in semistable reduction in a broader context. 

We give a brief explanation of how to use the formalism described in~\cite[Section~3]{R19}.  First, the subdivision $\Gamma^+$ of $\Gamma$ produces a bubbling of the curve $C$ which we denote $C^+$, obtained by adding $2$-pointed rational components wherever new $2$-valent vertices have appeared in $\Gamma^+$. The expansions of $\mathbb P$ obtained from $\Pi_t$ and $\Delta$ are $2$-dimensional schemes with a flat map to $C$; they are typically reducible, and the irreducible components are in bijection with the vertices of $\Pi_t$ or $\Delta$ respectively. Moreover, if $v$ is a vertex, then the irreducible component corresponding to $v$ is birational to the projectively completed logarithmic canonical bundle of $C_v$, where $C_v$ is the component of the curve  $C^+$ determined by the image of $v$ in $\Gamma^+$. See also~\cite{ACGS15,NS06}. 

We record a basic lemma. The notation $C^+$ will be used to denote a semistable bubbling induced by subdivisions $\Gamma^+$ of $\Gamma$, as in the paragraph above. 

\begin{lemma}
Let $s:C\to \mathbb P$ be a section. Then $s$ can be enhanced to a logarithmic map if and only if it it admits a factorization
\[
s:C^+\to \mathbb P(\Pi_t)\to \mathbb P,
\]
such the map $C^+\to \mathbb P(\Pi_t)$ sends nodes and markings to strata of codimension at most $1$, and it is predeformable: at each node $q$ of $C$ mapping to a divisor $D$ in $\mathbb P(\Pi_t)$, the tangency order along $q$ as measured on the two components meeting at $q$ coincide. 
\end{lemma}

\begin{proof}
First, suppose the section can be enhanced to a logarithmic one. There is an induced map on the tropicalizations, i.e. the cone stacks associated to $C$ and $\mathbb P$. Explicitly, these are cone over the tropical curve determined by $C$ and $\mathbb R$ times this tropical curve. By using the semistable reduction procedure~\cite{AK00,Mol16}, as applied for instance in~\cite[Section~2]{R19}, we can subdivide the source and target of this tropicalization so that each cone of the subdivided source maps surjectively onto a cone of the subdivided target. By using the dictionary between subdivisions and logarithmic modification, we obtain the claimed factorization. Conversely, transverse maps always admit logarithmic enhancements, see for instance~\cite[Lemma~6.4.3]{R19}. 
\end{proof}

\subsubsection{Curves with a fixed tropicalization}\label{sec: curves-with-tropicalization} We require one final technical ingredient before proceeding to the proof. Let $C$ be a cycle of rational curves consisting of $k$ components and denote by $G$ its dual graph. Let $u$ and $v$ be vertices of $G$. There are exactly two non-self-intersecting paths from $u$ to $v$. We denote them by $P$ and $P'$. 

Fix positive integers $a$ and $b$. Consider a tropical curve $\Gamma$ over {$G$} with positive integer edge lengths such that $a\ell(P) = b\ell(P')$, where $\ell$ is the length function.  For the following lemma, we fix such a tropical curve. 

\begin{lemma}\label{sec: log-moduli-fixed-tropicalization}
The moduli space of logarithmic curves over $\spec (\NN\to \CC)$ whose underlying curve is $C$, and whose tropicalization is $\Gamma$, is isomorphic to $\mathbb G_m$. 
\end{lemma}

\begin{proof}
Let $E$ be the edge set of $G$. The minimal base logarithmic structure on $C$ gives a curve $\mathcal C\to \spec(\mathbb N^E\to \CC)$. The logarithmic curves in question are all obtained by pullbacks $\spec(\mathbb N\to \CC)\to \spec(\mathbb N^E\to \CC)$, with the map on monoids specified by the given relation among the path lengths of $P$ and $P'$, which is isomorphic to $\mathbb G_m$. Further details and a hands on construction of the space of logarithmic realizations of a tropical curve may be found in~\cite[Section~1.5]{BNR22}. 
\end{proof}

It is useful to understand this $\mathbb G_m$ as determining gluing data for a line bundle. The tropical curve $\Gamma$ above admits a piecewise linear function that is linear along the path $P$ with slope $b$ and is linear along the path $P'$ with slope $a$. The function is unique up to the addition of constants. It determines a line bundle on the curve $C$ -- the piecewise linear function gives rise to a section of the characteristic abelian sheaf $\overline M_C^{\sf gp}$, and the short exact sequence
\[
0\to \cO_C^\star\to M_C^{\sf gp}\to \overline M_C^{\sf gp}\to 0
\]
gives rise a connecting map that produces a $\mathbb G_m$-torsor from this section. Alternatively, viewing the logarithmic structure as giving local maps to toric varieties, the standard construction associating line bundles to piecewise linear functions in toric geometry can be used. 

In any event, the line bundle on each component of $C$ is determined up to isomorphism as the weighted sum of the nodes of the component, where each node is given the weight equal to to the slope of the piecewise linear function at the corresponding edge. If we choose any edge $e$ of the cycle $G$, the $\mathbb G_m$ above can be see as the choice of an isomorphism, i.e. gluing data, of the fibers of the pullback of the line bundle to the partial normalization at the node corresponding to $e$. Additional details, including a formula for these line bundles, may be found in~\cite[Section~2]{RSW17A}.  

\subsection{Proof of Theorem~\ref{prop:lhntc}} Following Section~\ref{sec: tropical-evaluations}, we have a morphism
\[
\mathsf{DR}^{\mathsf{log}}(\mathbf x,k)\to \mathsf{Ex},
\]
and the class of the branch polynomial determines a codimension $n-3+2g$ stratum $S$ in $\mathsf{Ex}$. We examine the usual stack theoretic fiber diagram below:
\[
\begin{tikzcd}
\mathsf{DR}_S\arrow{d}\arrow{r}\ar[rd,phantom,"\square"] & \mathsf{DR}^{\mathsf{log}}(\mathbf x,k)\arrow{d}\\
S\arrow{r}&\mathsf{Ex}.
\end{tikzcd}
\]
By definition, the intersection number we are interested in is obtained by pulling back the virtual class along the regular embedding of $S$ into $\mathsf{Ex}$, and then taking degree. 

We first show the intersection number splits as a sum of strata. The intersection class above may be realized as the pullback
\[
\mathsf{DR}^{\mathsf{log}}(\mathbf x,k)\to \cM_{g,n}^\sfbu(\mathbf x,k)\to \mathsf{Ex}. 
\]
The pullback of $S$ along the second map is a sum of codimension $g$ strata: precisely those corresponding to the cones of $\mathsf{DR}^{\trop}(\mathbf x,k)$ where the target graph has the maximal possible number of bounded edges, as in the preceding tropical discussion. The intersection therefore splits as
\[
\deg\left( [\mathsf{DR}^\lightening(\mathbf x,k)] \cdot  br_g(\mathbf x, k)\right) = \sum_{[\Gamma\to T]} m_{[\Gamma\to T]},
\]
where $[\Gamma\to T]$ ranges over the combinatorial types meeting the tropical branch condition. 

Each summand in the expression above is virtually $0$-dimensional. When $k$ is $0$, the intersection is scheme theoretically $0$-dimensional, and more specifically, is supported on $0$-dimensional strata determined by $[\Gamma\to T]$. When $k$ is nonzero, the components of the intersection still split into components, but some of them may be positive dimensional. These correspond precisely to covers $[\Gamma\to T]$ where there are genus $1$ leaves attached to $\Gamma$. The corresponding summand is obtained by taking the degree of the virtual class. 

We now use the degeneration/splitting formula to argue that this multiplicity is exactly the one calculated by the tropical weights. 

Fix a type $\Gamma\to T$ and let $\mathsf{DR}(\Gamma\to T)$ be the associated stratum. The graph $\Gamma$ has trivalent and bivalent vertices, and when $k$ is nonzero, possibly univalent vertices (i.e. leaves) of genus $1$. We temporarily ignore the bivalent vertices, and later explain why they do not affect the calculation. At each trivalent vertex $v$ of $\Gamma$, the genus is necessarily $0$, and the difference between the incoming and outgoing slopes is $k$, which is precisely the degree of the logarithmic canonical bundle of $C_v$; since the curve has genus $0$ there is a unique section of $\omega_{C_v}^{\mathsf{log}}$ with the zeroes and poles as specified by the slopes at $v$. By recording these, there is a cutting morphism
\[
\kappa: \mathsf{DR}(\Gamma\to T) \to \prod_{v\in \Gamma} \mathsf{DR}(v).
\]
By the same argument as above, there is a unique pluricanonical section with the given vanishing up to scaling, the right hand side is a point. The multiplicity $m_{[\Gamma\to T]}$ may now be computed by the multiplicities in the degeneration formula in logarithmic Gromov--Witten theory~\cite[Section~6]{R19}. It remains only to explain why and how the formula applies, and for this we use the geometry of expansions. 

First, we show that the map $\kappa$ is finite and surjective, or equivalently that $\mathsf{DR}(\Gamma\to T)$ is nonempty of dimension $0$. Since $\Gamma$ is trivalent, there is a unique curve $C$ with this dual graph. Let $\Theta$ denote the cone associated to the type of $\Gamma\to T$. The moduli space of curves with tropicalization given by $\Theta$ is exactly $\mathbb G_m^g$. To see this, choose a spanning tree, and then using the discussion  of Section~\ref{sec: curves-with-tropicalization} to furnish the moduli at the nodes associated to the remaining edges. We denote this torus by $T_\Theta$.

Choose any point of $T_\Theta$. We now examine the expanded projective bundle
\[
\mathbb P(\Pi_t)\to C.
\] 
As we have already explained, at each vertex of $\Gamma$ there is a unique choice of pluricanonical section with vanishing specified by $\Gamma\to T$. These provide maps from each component $C_v$ of $C$ to the surface $\mathbb P(\Pi_t)$. By definition of $\mathbb P(\Pi_t)$, each vertex $v$ of $\Gamma$ determines an irreducible component and the germ of an edge at $v$ determines a divisor of this component. The germ also determines a marked point $p_e$ on $C_v$, which maps to this divisor. The logarithmic interior of this divisor is in turn a $\mathbb G_m$, and the image of $p_e$ under this section is an element\footnote{In fact, it can be described very concretely as the leading coefficient of the leading term of the Laurent expansion of the section at $p_e$; it is essentially a generalized residue.} of this $\mathbb G_m$.

In this manner, if we choose pluricanonical sections at each vertex, each edge determines two points on this $\mathbb G_m$. In order for these local maps at $C_v$ to glue to a section of the expanded bundle above, these evaluations at edges must give the same element of $\mathbb G_m$. There is no guarantee that these glue to give a global map, and indeed typically, they will not glue. 

If vertices $v$ and $v'$ are connected by an edge, then by choosing the section on $C_v$ arbitrarily, we may scale the section at $C_{v'}$ to ensure the gluing condition. However, if the genus of $\Gamma$ is positive, there may be two edges between $v$ and $v'$ and gluing becomes overdetermined. More generally, the phenomenon arises when there are multiple paths between two vertices. It is here that we will use the $\mathbb G_m^g$ worth of moduli. By the discussion following Lemma~\ref{sec: log-moduli-fixed-tropicalization}, by adjusting the point in the moduli $T_\Theta$, we scale the gluing data at the line bundle at the node. Choose a spanning tree and identify $T_\Theta$ with the scalars associated to each edge in the complement. There is a unique choice of scalar for each edge that guarantees that the gluing condition at each of these edges. We are left with a section:
\[
C\to \mathbb P(\Pi_t).
\]
The section is transverse to the strata, and therefore the resulting composite $C\to \mathbb P$ is a logarithmic map, and a point of $\mathsf{DR}(\Gamma\to T)$. 

The argument at univalent vertices is similar. However, since these are attached to the rest of the graph by a single edge, the moduli space of logarithmic enhancements of the underlying curve is unaffected by the presence of such a vertex. Note that two such genus $1$ vertices cannot lie over the same vertex in $T$. The moduli space of pluricanonical sections here is $1$-dimensional and isomorphic to $\Mbar_{1,1}$. The degree of the $k^{th}$ logarithmic pluricanonical bundle is $k$, and since the canonical on an elliptic curve is trivial, every point in $\Mbar_{1,1}$ satisfies the condition that the $k$-fold twist of the trivial bundle by the marking is the $k$-fold twist of the pluricanonical by the marking. The moduli space is smooth; the obstruction bundle is given by the normal bundle to the embedding of $\Mbar_{1,1}$ into the universal degree $1$ Picard scheme. The normal bundle is easily seen to be the line bundle with fiber over $[C]$ given by $H^1(C,\mathcal O_C)$. This is the dual of the Hodge bundle. See also~\cite[Section~5]{MW17}.The contribution for each genus $1$ univalent vertex is $\int_{\Mbar_{1,1}}(-\lambda_1) = \frac{-1}{24}$. 

Two steps remain to conclude: (i) we must count the number of logarithmic lifts and (ii) explain how to deal with $2$-valent vertices. The number of lifts is calculated by applying~\cite[Theorem~5.3.3]{ACGS15} and after accounting for automorphisms, gives precisely the claimed formula. Finally, concerning $2$-valent vertices, we note that if $\Gamma^+$ contains a $2$-valent vertex with slope $d$, the moduli space has a $\mathbb Z/d$-stabilizer at this point. However, the creation of a $2$-valent vertex also creates a new bounded edge of weight equal to $d$. The additional weight $d$ provides a factor of $d$ in the formula of~\cite[Theorem~5.3.3]{ACGS15}, which is cancelled by the automorphism factor. The result is precisely the claimed multiplicity, and the theorem follows. 

\qed

\section{Properties of leaky numbers}
\subsection{Piecewise polynomiality of the counts}
The aim of this section is to prove the following theorem:
\begin{theorem}\label{thm-Ngxpiecewisepoly}
The map ${\mathbf{x}} \mapsto \leakyn$ is piecewise polynomial.

More precisely, the set of definition of the map, i.e.\ $$\big\{\mathbf{x} \in \ZZ^n \;|\; \sum_{i=1}^n x_i= k\cdot (2g-2+n) , \;\mathbf{x}^+>0, \;\mathbf{x}^-<0\big\}$$ is subdivided by walls with equation 

$$\big\{\sum_{i\in I}x_i =  k\cdot(2g_I -2+|I|)\big\},$$
where $I$ is a subset of $\{1,\ldots,n\}$ and $g_I\in \ZZ$ satisfies $0 \leq g_I\leq g$. In each region, $\leakyn$ equals a polynomial of degree $4g-3+n$ in $\mathbf{x}$.

\end{theorem}

The walls correspond to cut $k$-leaky graphs with two components, where the ends $x_i$ for $i\in I$ appear in one component of genus $g_I$.

The proof follows the methods established in \cite{CJM2} for the case of double Hurwitz numbers and in \cite{AB14} for the case of double Hirzebruch numbers.

The arguments differ for the cases $g=0$ and $g>0$. We start by assuming that $g>0$.

 Put roughly, we are counting tropical leaky covers with variables for the weights of ends. As we will see below, each such cover contributes a polynomial count, where the degree of the polynomial equals the number of bounded edges plus the first Betti number of the graph. Consequently, graphs that have genus $1$ vertices contribute lower degree summands, and we can therefore ignore them in the following study. We restrict ourselves to considering only trivalent leaky tropical covers.

\begin{definition}\label{def-xgraph}
Given $g$ and $\mathbf{x}$, an \emph{ $\mathbf{x}$-graph} $\Gamma$  is a connected, genus $g$, trivalent graph with $n$ ends labeled $x_1, \ldots, x_n$.
\end{definition}

\begin{figure}
\begin{center}

\tikzset{every picture/.style={line width=0.75pt}} %set default line width to 0.75pt        

\begin{tikzpicture}[x=0.75pt,y=0.75pt,yscale=-1,xscale=1]

%Straight Lines [id:da34377374973703767] 
\draw    (60.5,71.25) -- (99.5,89.25) ;
%Straight Lines [id:da21059052015767976] 
\draw    (99.5,89.25) -- (59.5,110.25) ;
%Straight Lines [id:da16081529621099] 
\draw    (99.5,89.25) -- (139,89.75) ;
%Straight Lines [id:da039566895467258556] 
\draw    (139,89.75) -- (170.5,70.25) ;
%Straight Lines [id:da5194255195490357] 
\draw    (139,89.75) -- (169.5,110.25) ;
%Straight Lines [id:da3354695366027892] 
\draw    (170.5,70.25) -- (169.5,110.25) ;
%Straight Lines [id:da2521243381034911] 
\draw    (170.5,70.25) -- (209.5,69.75) ;
%Straight Lines [id:da4848016126445681] 
\draw    (169.5,110.25) -- (209.5,110.25) ;
%Straight Lines [id:da814603322540856] 
\draw    (209.5,69.75) -- (209.5,110.25) ;
%Straight Lines [id:da5574961449618587] 
\draw    (209.5,69.75) -- (249.5,61.25) ;
%Straight Lines [id:da3991317387714529] 
\draw    (209.5,110.25) -- (249,118.75) ;

% Text Node
\draw (61,52.5) node [anchor=north west][inner sep=0.75pt]   [align=left] {$\displaystyle x_{1}$};
% Text Node
\draw (66,108.5) node [anchor=north west][inner sep=0.75pt]   [align=left] {$\displaystyle x_{2}$};
% Text Node
\draw (250,54) node [anchor=north west][inner sep=0.75pt]   [align=left] {$\displaystyle x_{3}$};
% Text Node
\draw (250,114) node [anchor=north west][inner sep=0.75pt]   [align=left] {$\displaystyle x_{4}$};

\end{tikzpicture}

\end{center}

\caption{An $\mathbf{x}$-graph of genus $2$ with four ends.}\label{fig-xgraph}

\end{figure}
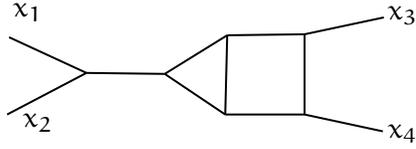

An $\mathbf{x}$-graph can be obtained from a $k$-leaky graph by just forgetting structure, more precisely, the vertex ordering, the orientation and the expansion factors for the bounded edges.

Vice versa, all leaky graphs that lead to the same $\mathbf{x}$-graph after forgetting can be parametrized by the lattice points in the bounded cells of a hyperplane arrangement. This follows from Corollary 2.13 in \cite{CJM2}, we outline the arguments here by following them along an example. We set $k=1$ throughout.

\begin{example}

Consider the  $\mathbf{x}$-graph from Figure \ref{fig-xgraph}. Pick a reference orientation for the bounded edges. Pick two edges $e_1,e_2$ such that $\Gamma\setminus\{e_1,e_2\}$ is a tree. Let the variables $i,j$ stand for the expansion factors of these two edges. The leaky condition then implies expansion factors  for the remaining edges (see Figure \ref{fig-reforientation}). \rcolor{All expansion factors are homogenous linear expressions in the  components of $\mathbf{x}$, $i$ and $j$.}

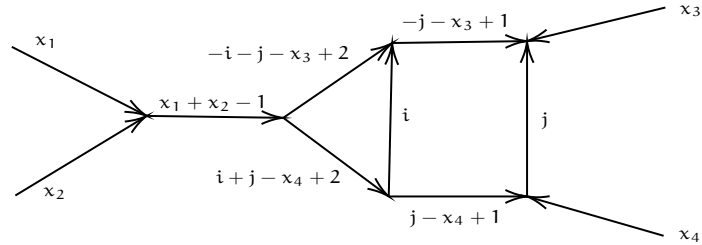
\begin{figure}
\begin{center}

\tikzset{every picture/.style={line width=0.75pt}} %set default line width to 0.75pt        

\begin{tikzpicture}[x=0.75pt,y=0.75pt,yscale=-1,xscale=1]

%Straight Lines [id:da34377374973703767] 
\draw    (60.87,91.15) -- (127.1,125.09) ;
\draw [shift={(128.88,126)}, rotate = 207.13] [color={rgb, 255:red, 0; green, 0; blue, 0 }  ][line width=0.75]    (10.93,-3.29) .. controls (6.95,-1.4) and (3.31,-0.3) .. (0,0) .. controls (3.31,0.3) and (6.95,1.4) .. (10.93,3.29)   ;
%Straight Lines [id:da21059052015767976] 
\draw    (62.62,166.18) -- (127.17,127.04) ;
\draw [shift={(128.88,126)}, rotate = 508.77] [color={rgb, 255:red, 0; green, 0; blue, 0 }  ][line width=0.75]    (10.93,-3.29) .. controls (6.95,-1.4) and (3.31,-0.3) .. (0,0) .. controls (3.31,0.3) and (6.95,1.4) .. (10.93,3.29)   ;
%Straight Lines [id:da16081529621099] 
\draw    (128.88,126) -- (195.76,126.94) ;
\draw [shift={(197.76,126.97)}, rotate = 180.81] [color={rgb, 255:red, 0; green, 0; blue, 0 }  ][line width=0.75]    (10.93,-3.29) .. controls (6.95,-1.4) and (3.31,-0.3) .. (0,0) .. controls (3.31,0.3) and (6.95,1.4) .. (10.93,3.29)   ;
%Straight Lines [id:da039566895467258556] 
\draw    (197.76,126.97) -- (251.04,90.34) ;
\draw [shift={(252.69,89.21)}, rotate = 505.5] [color={rgb, 255:red, 0; green, 0; blue, 0 }  ][line width=0.75]    (10.93,-3.29) .. controls (6.95,-1.4) and (3.31,-0.3) .. (0,0) .. controls (3.31,0.3) and (6.95,1.4) .. (10.93,3.29)   ;
%Straight Lines [id:da5194255195490357] 
\draw    (197.76,126.97) -- (249.34,165.47) ;
\draw [shift={(250.94,166.66)}, rotate = 216.74] [color={rgb, 255:red, 0; green, 0; blue, 0 }  ][line width=0.75]    (10.93,-3.29) .. controls (6.95,-1.4) and (3.31,-0.3) .. (0,0) .. controls (3.31,0.3) and (6.95,1.4) .. (10.93,3.29)   ;
%Straight Lines [id:da3354695366027892] 
\draw    (250.94,166.66) -- (252.64,91.21) ;
\draw [shift={(252.69,89.21)}, rotate = 451.29] [color={rgb, 255:red, 0; green, 0; blue, 0 }  ][line width=0.75]    (10.93,-3.29) .. controls (6.95,-1.4) and (3.31,-0.3) .. (0,0) .. controls (3.31,0.3) and (6.95,1.4) .. (10.93,3.29)   ;
%Straight Lines [id:da2521243381034911] 
\draw    (252.69,89.21) -- (318.7,88.27) ;
\draw [shift={(320.7,88.24)}, rotate = 539.1800000000001] [color={rgb, 255:red, 0; green, 0; blue, 0 }  ][line width=0.75]    (10.93,-3.29) .. controls (6.95,-1.4) and (3.31,-0.3) .. (0,0) .. controls (3.31,0.3) and (6.95,1.4) .. (10.93,3.29)   ;
%Straight Lines [id:da4848016126445681] 
\draw    (250.94,166.66) -- (318.7,166.66) ;
\draw [shift={(320.7,166.66)}, rotate = 180] [color={rgb, 255:red, 0; green, 0; blue, 0 }  ][line width=0.75]    (10.93,-3.29) .. controls (6.95,-1.4) and (3.31,-0.3) .. (0,0) .. controls (3.31,0.3) and (6.95,1.4) .. (10.93,3.29)   ;
%Straight Lines [id:da814603322540856] 
\draw    (320.7,166.66) -- (320.7,90.24) ;
\draw [shift={(320.7,88.24)}, rotate = 450] [color={rgb, 255:red, 0; green, 0; blue, 0 }  ][line width=0.75]    (10.93,-3.29) .. controls (6.95,-1.4) and (3.31,-0.3) .. (0,0) .. controls (3.31,0.3) and (6.95,1.4) .. (10.93,3.29)   ;
%Straight Lines [id:da5574961449618587] 
\draw    (390.45,71.3) -- (322.64,87.77) ;
\draw [shift={(320.7,88.24)}, rotate = 346.35] [color={rgb, 255:red, 0; green, 0; blue, 0 }  ][line width=0.75]    (10.93,-3.29) .. controls (6.95,-1.4) and (3.31,-0.3) .. (0,0) .. controls (3.31,0.3) and (6.95,1.4) .. (10.93,3.29)   ;
%Straight Lines [id:da3991317387714529] 
\draw    (389.57,186.51) -- (322.62,167.22) ;
\draw [shift={(320.7,166.66)}, rotate = 376.07] [color={rgb, 255:red, 0; green, 0; blue, 0 }  ][line width=0.75]    (10.93,-3.29) .. controls (6.95,-1.4) and (3.31,-0.3) .. (0,0) .. controls (3.31,0.3) and (6.95,1.4) .. (10.93,3.29)   ;

% Text Node
\draw (70.96,84.62) node [anchor=north west][inner sep=0.75pt]  [font=\tiny] [align=left] {$\displaystyle x_{1}$};
% Text Node
\draw (75.18,160.55) node [anchor=north west][inner sep=0.75pt]  [font=\tiny] [align=left] {$\displaystyle x_{2}$};
% Text Node
\draw (396.04,68.52) node [anchor=north west][inner sep=0.75pt]  [font=\tiny] [align=left] {$\displaystyle x_{3}$};
% Text Node
\draw (395.54,182.2) node [anchor=north west][inner sep=0.75pt]  [font=\tiny] [align=left] {$\displaystyle x_{4}$};
% Text Node
\draw (255.91,120.61) node [anchor=north west][inner sep=0.75pt]  [font=\tiny] [align=left] {$\displaystyle i$};
% Text Node
\draw (326.39,122.68) node [anchor=north west][inner sep=0.75pt]  [font=\tiny] [align=left] {$\displaystyle j$};
% Text Node
\draw (133.54,114.09) node [anchor=north west][inner sep=0.75pt]  [font=\tiny] [align=left] {$\displaystyle x_{1} +x_{2} -1$};
% Text Node
\draw (162.02,152) node [anchor=north west][inner sep=0.75pt]  [font=\tiny] [align=left] {$\displaystyle i+j-x_{4} +2$};
% Text Node
\draw (158.04,89.35) node [anchor=north west][inner sep=0.75pt]  [font=\tiny] [align=left] {$\displaystyle -i-j-x_{3} +2$};
% Text Node
\draw (256.06,73.15) node [anchor=north west][inner sep=0.75pt]  [font=\tiny] [align=left] {$\displaystyle -j-x_{3} +1$};
% Text Node
\draw (259.84,171.85) node [anchor=north west][inner sep=0.75pt]  [font=\tiny] [align=left] {$\displaystyle j-x_{4} +1$};

\end{tikzpicture}

\end{center}
\caption{A reference orientation for the $\mathbf{x}$-graph of Figure \ref{fig-xgraph}, and its corresponding expansion factors.}
\label{fig-reforientation}
\end{figure}

In order to obtain a leaky graph, the expansion factors must be positive. We can reach this by reversing the orientation of negative edges. The hyperplane arrangement whose (non trivial) bounded chambers correspond to possible types of leaky graphs is thus given by setting the \rcolor{linear expressions for the} expansion factors \rcolor{to} zero. We depict this hyperplane arrangement in Figure \ref{fig-hyperplane}.

\begin{figure}
\begin{center}

\tikzset{every picture/.style={line width=0.75pt}} %set default line width to 0.75pt        

\begin{tikzpicture}[x=0.75pt,y=0.75pt,yscale=-1,xscale=1]
%uncomment if require: \path (0,451); %set diagram left start at 0, and has height of 451

%Straight Lines [id:da9385234888601794] 
\draw    (162.27,190.8) -- (307.8,190.6) ;
\draw [shift={(309.8,190.6)}, rotate = 539.9200000000001] [color={rgb, 255:red, 0; green, 0; blue, 0 }  ][line width=0.75]    (10.93,-3.29) .. controls (6.95,-1.4) and (3.31,-0.3) .. (0,0) .. controls (3.31,0.3) and (6.95,1.4) .. (10.93,3.29)   ;
%Straight Lines [id:da7398391097092911] 
\draw    (190.2,218.6) -- (189.81,93) ;
\draw [shift={(189.8,91)}, rotate = 449.82] [color={rgb, 255:red, 0; green, 0; blue, 0 }  ][line width=0.75]    (10.93,-3.29) .. controls (6.95,-1.4) and (3.31,-0.3) .. (0,0) .. controls (3.31,0.3) and (6.95,1.4) .. (10.93,3.29)   ;
%Straight Lines [id:da20233298209147976] 
\draw  [dash pattern={on 4.5pt off 4.5pt}]  (104.47,190.6) -- (338.47,191) ;
%Straight Lines [id:da8165198309303803] 
\draw  [dash pattern={on 4.5pt off 4.5pt}]  (108.28,220.31) -- (342.28,220.71) ;
%Straight Lines [id:da7591622654744768] 
\draw  [dash pattern={on 4.5pt off 4.5pt}]  (102.87,149.8) -- (215.5,149.99) -- (336.87,150.2) ;
%Straight Lines [id:da22203168276214746] 
\draw  [dash pattern={on 4.5pt off 4.5pt}]  (189.8,239.8) -- (189.4,61.8) ;
%Straight Lines [id:da0509471454825835] 
\draw  [dash pattern={on 4.5pt off 4.5pt}]  (210.17,249.8) -- (93.51,133.23) ;
%Straight Lines [id:da5157290737930295] 
\draw  [dash pattern={on 4.5pt off 4.5pt}]  (296.33,247.13) -- (143.13,92.07) ;
%Shape: Right Triangle [id:dp05406478703612261] 
\draw  [draw opacity=0][fill={rgb, 255:red, 155; green, 155; blue, 155 }  ,fill opacity=1 ] (190.58,140.69) -- (199.42,149.69) -- (190.58,149.69) -- cycle ;
%Straight Lines [id:da3762028785229905] 
\draw    (100.4,322.6) -- (119.6,314.2) ;
%Straight Lines [id:da9476427303002163] 
\draw    (100.4,322.6) -- (120,333.4) ;
%Straight Lines [id:da6811634716614926] 
\draw    (120,333.4) -- (119.6,314.2) ;
%Straight Lines [id:da28232787824488415] 
\draw    (119.6,314.2) -- (139.6,314.2) ;
%Straight Lines [id:da30175433344521896] 
\draw    (120,333.4) -- (139.6,333) ;
%Straight Lines [id:da23504127472413705] 
\draw    (139.6,314.2) -- (139.6,333) ;
%Straight Lines [id:da4009289589298298] 
\draw    (182.8,322.2) -- (202,313.8) ;
%Straight Lines [id:da18616547013601015] 
\draw    (182.8,322.2) -- (202.4,333) ;
%Straight Lines [id:da6527318092091899] 
\draw    (202.4,333) -- (202,313.8) ;
%Straight Lines [id:da7508501768529577] 
\draw    (202,313.8) -- (222,313.8) ;
%Straight Lines [id:da3119345573049659] 
\draw    (202.4,333) -- (222,332.6) ;
%Straight Lines [id:da13317651242244244] 
\draw    (222,313.8) -- (222,332.6) ;
%Straight Lines [id:da4530942007555995] 
\draw    (269.6,322.6) -- (288.8,314.2) ;
%Straight Lines [id:da8513227827797136] 
\draw    (269.6,322.6) -- (289.2,333.4) ;
%Straight Lines [id:da2776132787980977] 
\draw    (289.2,333.4) -- (288.8,314.2) ;
%Straight Lines [id:da8342580544021218] 
\draw    (288.8,314.2) -- (308.8,314.2) ;
%Straight Lines [id:da6004640373666722] 
\draw    (289.2,333.4) -- (308.8,333) ;
%Straight Lines [id:da09335338663310544] 
\draw    (308.8,314.2) -- (308.8,333) ;
%Straight Lines [id:da1661373880708753] 
\draw    (356,323) -- (375.2,314.6) ;
%Straight Lines [id:da01828089551410328] 
\draw    (356,323) -- (375.6,333.8) ;
%Straight Lines [id:da4528184582443049] 
\draw    (375.6,333.8) -- (375.2,314.6) ;
%Straight Lines [id:da20629892470894784] 
\draw    (375.2,314.6) -- (395.2,314.6) ;
%Straight Lines [id:da23283398503301] 
\draw    (375.6,333.8) -- (395.2,333.4) ;
%Straight Lines [id:da7108665495644371] 
\draw    (395.2,314.6) -- (395.2,333.4) ;
%Shape: Free Drawing [id:dp3432459198767053] 
\draw  [color={rgb, 255:red, 0; green, 0; blue, 0 }  ][line width=0.75] [line join = round][line cap = round] (117.6,322.4) .. controls (118.4,323.47) and (118.68,325.82) .. (120,325.6) .. controls (121.53,325.34) and (121.6,322.93) .. (122.4,321.6) ;
%Shape: Free Drawing [id:dp573231785889089] 
\draw  [color={rgb, 255:red, 0; green, 0; blue, 0 }  ][line width=0.75] [line join = round][line cap = round] (136.4,326) .. controls (137.73,324.53) and (138.42,321.6) .. (140.4,321.6) .. controls (142.07,321.6) and (142,324.53) .. (142.8,326) ;
%Shape: Free Drawing [id:dp13625743176556815] 
\draw  [color={rgb, 255:red, 0; green, 0; blue, 0 }  ][line width=0.75] [line join = round][line cap = round] (218.4,325.2) .. controls (219.47,323.47) and (219.56,320) .. (221.6,320) .. controls (223.87,320) and (224.53,323.47) .. (226,325.2) ;
%Shape: Free Drawing [id:dp49530629760893174] 
\draw  [color={rgb, 255:red, 0; green, 0; blue, 0 }  ][line width=0.75] [line join = round][line cap = round] (209.86,329.51) .. controls (211.57,330.47) and (215.09,330.41) .. (215,332.37) .. controls (214.9,334.56) and (211.38,334.85) .. (209.57,336.09) ;
%Shape: Free Drawing [id:dp9738050556267902] 
\draw  [color={rgb, 255:red, 0; green, 0; blue, 0 }  ][line width=0.75] [line join = round][line cap = round] (127.57,330.09) .. controls (129.31,330.67) and (133.51,331.18) .. (132.43,332.66) .. controls (131.37,334.11) and (129.38,334.56) .. (127.86,335.51) ;
%Shape: Free Drawing [id:dp7688163214395757] 
\draw  [color={rgb, 255:red, 0; green, 0; blue, 0 }  ][line width=0.75] [line join = round][line cap = round] (127.86,310.94) .. controls (129.1,312.18) and (131.65,312.91) .. (131.57,314.66) .. controls (131.5,316.29) and (128.9,316.56) .. (127.57,317.51) ;
%Shape: Free Drawing [id:dp66301211867379] 
\draw  [color={rgb, 255:red, 0; green, 0; blue, 0 }  ][line width=0.75] [line join = round][line cap = round] (295.86,310.94) .. controls (298.14,312.18) and (300.75,312.95) .. (302.71,314.66) .. controls (303.36,315.22) and (300.97,314.43) .. (300.14,314.66) .. controls (298.45,315.11) and (296.9,315.99) .. (295.29,316.66) ;
%Shape: Free Drawing [id:dp9913151508655599] 
\draw  [color={rgb, 255:red, 0; green, 0; blue, 0 }  ][line width=0.75] [line join = round][line cap = round] (383.29,310.94) .. controls (385.1,312.09) and (388.41,312.25) .. (388.71,314.37) .. controls (388.96,316.1) and (385.67,316.09) .. (384.14,316.94) ;
%Shape: Free Drawing [id:dp9631690599633395] 
\draw  [color={rgb, 255:red, 0; green, 0; blue, 0 }  ][line width=0.75] [line join = round][line cap = round] (199.86,324.66) .. controls (200.81,323.23) and (201.03,320.69) .. (202.71,320.37) .. controls (204.17,320.09) and (204.81,322.47) .. (205.86,323.51) ;
%Shape: Free Drawing [id:dp06926288438893691] 
\draw  [color={rgb, 255:red, 0; green, 0; blue, 0 }  ][line width=0.75] [line join = round][line cap = round] (371.86,324.66) .. controls (373,323.23) and (373.46,320.3) .. (375.29,320.37) .. controls (377.31,320.45) and (377.95,323.42) .. (379.29,324.94) ;
%Shape: Free Drawing [id:dp8336766258435198] 
\draw  [color={rgb, 255:red, 0; green, 0; blue, 0 }  ][line width=0.75] [line join = round][line cap = round] (280.14,324.37) .. controls (280.9,326.18) and (283.53,328.18) .. (282.43,329.8) .. controls (281.18,331.63) and (278.05,330.37) .. (275.86,330.66) ;
%Shape: Free Drawing [id:dp3169233912103394] 
\draw  [color={rgb, 255:red, 0; green, 0; blue, 0 }  ][line width=0.75] [line join = round][line cap = round] (193.57,323.23) .. controls (194.24,325.23) and (196.63,327.41) .. (195.57,329.23) .. controls (194.55,330.97) and (191.57,329.8) .. (189.57,330.09) ;
%Shape: Free Drawing [id:dp3079094310881685] 
\draw  [color={rgb, 255:red, 0; green, 0; blue, 0 }  ][line width=0.75] [line join = round][line cap = round] (109.29,324.37) .. controls (110.52,325.9) and (114.29,327.47) .. (113,328.94) .. controls (111.31,330.88) and (107.86,328.94) .. (105.29,328.94) ;
%Shape: Free Drawing [id:dp12385565990888359] 
\draw  [color={rgb, 255:red, 0; green, 0; blue, 0 }  ][line width=0.75] [line join = round][line cap = round] (109,316.66) .. controls (110.52,316.85) and (112.64,316.01) .. (113.57,317.23) .. controls (114.3,318.17) and (112.43,319.32) .. (111.86,320.37) ;
%Shape: Free Drawing [id:dp3582002158089017] 
\draw  [color={rgb, 255:red, 0; green, 0; blue, 0 }  ][line width=0.75] [line join = round][line cap = round] (190.43,316.37) .. controls (191.67,316.85) and (193.58,316.6) .. (194.14,317.8) .. controls (194.67,318.92) and (193.19,320.09) .. (192.71,321.23) ;
%Shape: Free Drawing [id:dp5617087446531775] 
\draw  [color={rgb, 255:red, 0; green, 0; blue, 0 }  ][line width=0.75] [line join = round][line cap = round] (276.14,316.09) .. controls (278.14,316.56) and (281.02,315.79) .. (282.14,317.51) .. controls (283.03,318.87) and (280.62,320.37) .. (279.86,321.8) ;
%Shape: Free Drawing [id:dp75859323806921] 
\draw  [color={rgb, 255:red, 0; green, 0; blue, 0 }  ][line width=0.75] [line join = round][line cap = round] (361.57,316.66) .. controls (364.24,316.85) and (367.89,315.15) .. (369.57,317.23) .. controls (370.89,318.85) and (367.48,320.85) .. (366.43,322.66) ;
%Shape: Polygon [id:ds04524188741032775] 
\draw  [draw opacity=0][fill={rgb, 255:red, 155; green, 155; blue, 155 }  ,fill opacity=1 ] (189.57,220.8) -- (189,229.09) -- (181.29,221.09) -- (181.29,221.09) -- (181.29,221.09) -- cycle ;
%Shape: Free Drawing [id:dp9412407996680404] 
\draw  [color={rgb, 255:red, 0; green, 0; blue, 0 }  ][line width=0.75] [line join = round][line cap = round] (392.17,321.8) .. controls (392.94,323.69) and (392.46,327.36) .. (394.5,327.47) .. controls (396.72,327.58) and (397.17,323.91) .. (398.5,322.13) ;
%Shape: Free Drawing [id:dp39022929618586233] 
\draw  [color={rgb, 255:red, 0; green, 0; blue, 0 }  ][line width=0.75] [line join = round][line cap = round] (297,329.8) .. controls (299,331) and (303.16,331.07) .. (303,333.4) .. controls (302.85,335.67) and (298.73,335) .. (296.6,335.8) ;
%Shape: Free Drawing [id:dp6573855519810146] 
\draw  [color={rgb, 255:red, 0; green, 0; blue, 0 }  ][line width=0.75] [line join = round][line cap = round] (383.8,330.6) .. controls (385.93,331.53) and (388.15,332.3) .. (390.2,333.4) .. controls (391.99,334.36) and (385,334.56) .. (385,336.6) ;
%Shape: Free Drawing [id:dp984959851575813] 
\draw  [color={rgb, 255:red, 0; green, 0; blue, 0 }  ][line width=0.75] [line join = round][line cap = round] (366.33,325.13) .. controls (367.11,326.91) and (369.99,329.05) .. (368.67,330.47) .. controls (367.07,332.17) and (364,330.24) .. (361.67,330.13) ;
%Shape: Free Drawing [id:dp6616240393448424] 
\draw  [color={rgb, 255:red, 0; green, 0; blue, 0 }  ][line width=0.75] [line join = round][line cap = round] (210.17,310.3) .. controls (212.17,311.52) and (216.03,311.63) .. (216.17,313.97) .. controls (216.28,315.97) and (212.39,315.3) .. (210.5,315.97) ;
%Shape: Free Drawing [id:dp8710413294253505] 
\draw  [color={rgb, 255:red, 0; green, 0; blue, 0 }  ][line width=0.75] [line join = round][line cap = round] (305.83,321.8) .. controls (307.11,323.26) and (306.36,326.41) .. (308.17,327.13) .. controls (310.06,327.89) and (310.41,323.69) .. (311.17,321.8) ;
%Shape: Free Drawing [id:dp2629896133206505] 
\draw  [color={rgb, 255:red, 0; green, 0; blue, 0 }  ][line width=0.75] [line join = round][line cap = round] (287.17,322.13) .. controls (287.94,323.13) and (288.25,325.31) .. (289.5,325.13) .. controls (290.93,324.93) and (291.06,322.69) .. (291.83,321.47) ;

% Text Node
\draw (347.27,186.33) node [anchor=north west][inner sep=0.75pt]  [font=\tiny] [align=left] {$\displaystyle j=0$};
% Text Node
\draw (348.65,216.09) node [anchor=north west][inner sep=0.75pt]  [font=\tiny] [align=left] {$\displaystyle j-x_{4} +1=0$};
% Text Node
\draw (344.27,146.07) node [anchor=north west][inner sep=0.75pt]  [font=\tiny] [align=left] {$\displaystyle -j-x_{3} +1=0$};
% Text Node
\draw (177.53,243.47) node [anchor=north west][inner sep=0.75pt]  [font=\tiny] [align=left] {$\displaystyle i=0$};
% Text Node
\draw (57.49,123.27) node [anchor=north west][inner sep=0.75pt]  [font=\tiny] [align=left] {$\displaystyle i+j-x_{4} +2=0$};
% Text Node
\draw (81.55,78.85) node [anchor=north west][inner sep=0.75pt]  [font=\tiny] [align=left] {$\displaystyle -i-j-x_{3} +2=0$};
% Text Node
\draw (158.27,164.67) node [anchor=north west][inner sep=0.75pt]   [align=left] {I};
% Text Node
\draw (197,166.4) node [anchor=north west][inner sep=0.75pt]   [align=left] {II};
% Text Node
\draw (288,286.8) node [anchor=north west][inner sep=0.75pt]   [align=left] {III};
% Text Node
\draw (374.4,287.2) node [anchor=north west][inner sep=0.75pt]   [align=left] {IV};
% Text Node
\draw (120.4,286.8) node [anchor=north west][inner sep=0.75pt]   [align=left] {I};
% Text Node
\draw (202.4,287.2) node [anchor=north west][inner sep=0.75pt]   [align=left] {II};
% Text Node
\draw (172,196.2) node [anchor=north west][inner sep=0.75pt]   [align=left] {III};
% Text Node
\draw (209.07,199.6) node [anchor=north west][inner sep=0.75pt]   [align=left] {IV};

\end{tikzpicture}

\end{center}

\caption{The hyperplane arrangement for the $\mathbf{x}$-graph of Figure \ref{fig-xgraph}.}
\label{fig-hyperplane}

\end{figure}
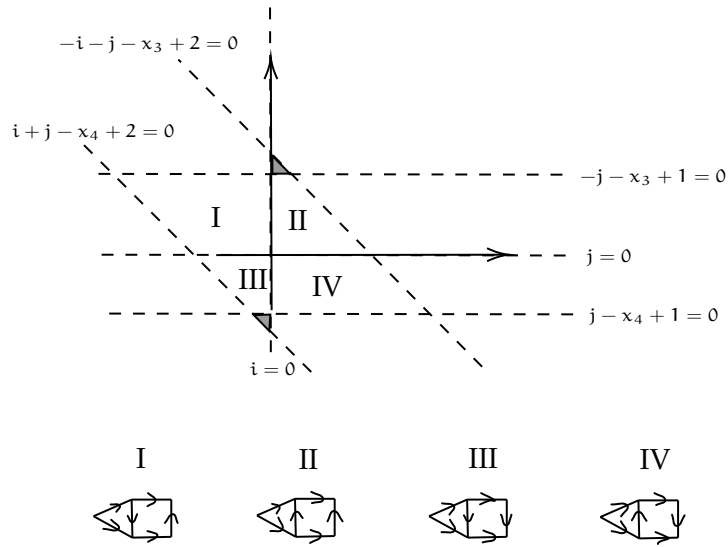
\end{example}

Notice that three hyperplanes that correspond to three edges adjacent to the same vertex enclose a small triangle times a linear space (where the linear space is trivial if $g=2$ for dimension reasons, as in Figure \ref{fig-hyperplane}). This is a consequence of the leaky condition. We call such bounded chambers trivial. They do not contain lattice points in the interior.

By Lemma 2.12 in \cite{CJM2}, the non trivial bounded chambers correspond to orientations without directed cycles. 

An interior lattice point in  a non trivial bounded chamber corrsesponds to a feasible orientation and the assignment of expansion factors for all edges. What is still missing to obtain a leaky graph is a choice of vertex ordering that is compatible with the edge orientations. Notice that the number of such vertex orderings is the same for each lattice point in a chamber $A$. We denote it by $m(A)$. This is true because the orientation of an edge is determined by the side of the corresponding hyperplane on which the lattice point of the leaky graph lies. Within a chamber, all edges are oriented in the same way, and thus the number of vertex orderings which is compatible with this orientation is constant.

We can deduce the following statement:
\begin{lemma}
Given an $\mathbf{x}$-graph $\Gamma$, the leaky graphs that yield $\Gamma$ after forgetting are in weighted bijection with the interior lattice points in the non trivial bounded chambers of the hyperplane arrangement given by the expansion factors in a $g$-dimensional vector space, where each lattice point in a chamber $A$ is weighted by $m(A)$.
\end{lemma}

The weight with which a leaky graph contributes to a count of leaky covers equals the size of the automorphism group of $\Gamma$ times the product of the expansion factors of the bounded edges.

For the $\mathbf{x}$-graph (with reference orientation) of Figure \ref{fig-reforientation}, the latter product equals the absolute value of

$$
(x_1+x_2-1)\cdot (-i-j-x_3+2)\cdot (i+j-x_4+2)\cdot i\cdot j\cdot (-j-x_3+1)\cdot (j-x_4+1).
$$

We denote this product by $\varphi_\Gamma$, and note it depends only on the chosen $\mathbf{x}$-graph $\Gamma$ (with reference orientation) and it equals the product of expansion factors (up to sign). More precisely, $\varphi_\Gamma$ equals the product of expansion factors if the reference orientation equals the orientation of the edges as imposed by the cover. Given a lattice point $(i_1,\ldots,i_g)$ in the hyperplane arrangement, which corresponds to a choice of cover, the orientation of an edge $e$ is determined by the question on which side of the hyperplane of $e$ the lattice point is. For some chambers in the hyperplane arrangement, we reverse the orientation of an odd number of edges, for some for an even number of edges. In the first case, $\varphi_\Gamma$ equals the negative of the product of the expansion factors, in the latter, it is equal to the product of the expansion factors.

The expression  $\varphi_\Gamma$ is  polynomial in the entries $x_i$ and in the variables $i_1,\ldots,i_g$, for the $g$ edges breaking the cycles. A choice of lattice point fixes the coordinates $(i_1,\ldots,i_g)$.

For each chamber $A$ in the hyperplane arrangement of the bounded edges, we let $\sign(A)$ be $+1$ if $\varphi_\Gamma$ equals the product of expansion factors, and $-1$ otherwise.

Denote by $S(\Gamma)$ the contribution to the count of leaky covers  $ \leakyn$ of genus $g$ and degree $\mathbf{x}$ given by a fixed $\mathbf{x}$-graph $\Gamma$, i.e.\ $\leakyn = \sum S(\Gamma)$, where the sum goes over all $\mathbf{x}$-graphs $\Gamma$. The following result follows from the discussion above:

\begin{proposition}\label{prop-contributionxgraph}
The contribution $S(\Gamma)$ of the $\mathbf{x}$-graph $\Gamma$ to the count of leaky covers  $ \leakyn$ of genus $g$ and degree $\mathbf{x}$ equals
$$S(\Gamma)=\frac{1}{|\Aut(\Gamma)|}\sum_A \sign(A)\cdot m(A)\cdot \sum_{(i_1,\ldots,i_g)\in A^{\circ}\cap \ZZ^g}\varphi_\Gamma(\mathbf{x},i_1,\ldots,i_g).$$
\end{proposition}

The contribution to the count of leaky covers of an $\mathbf{x}$-graph that we determine in Proposition \ref{prop-contributionxgraph} is independent from the choice of the edges breaking the cycle as this just amounts to choosing a basis of the $g$-dimensional vector space parametrizing edge expansion factors. The beginning of Section 2 of \cite{CJM2} offers a coordinate-free expression of this vector space. For the sake of explicitness, we chose to pick coordinates here.

\begin{proof}[Proof of Theorem \ref{thm-Ngxpiecewisepoly}]
For $g>0$, the piecewise polynomial structure of the map $ \leakyn$ follows from Proposition \ref{prop-contributionxgraph}, since summing a polynomial $\varphi_\Gamma$ of degree $3g-3+n$ (which equals the number of bounded edges of a trivalent graph $\Gamma$ with $n$ ends) over lattice points in a polytope in a $g$-dimensional vector space yields a polynomial of degree $4g-3+n$ by the iterated application of Bernoulli's formulas for the sum of  $k$-th powers. This is similar to Ehrhart theory, where we add lattice points (i.e. discretely integrate the function $1$) over a dilate of a polytope (in our case given by the hyperplanes determined by the edge weights) and obtain a polynomial. Here, we discretely integrate a polynomial function and use Bernoulli's formulas for the sum of  $k$-th powers to obtain a polynomial. The degree grows for each variable by $1$, i.e.\ in total by $g$.

The contribution $S(\Gamma)$ of an $\mathbf{x}$-graph $\Gamma$ is polynomial in $\mathbf{x}$ for any $\mathbf{x}$ such that the topology of the hyperplane arrangement given by the expansion factors does not change. By Lemma 3.8 of \cite{CJM2}, the topology changes for values $\mathbf{x}$ for which there exists a set of edges of $\Gamma$ of expansion factor $0$ that cut $\Gamma$ into two connected components. One of the connected components contains the ends with $i\in I$ and some vertices, at which there is leaking. Thus the value  $\mathbf{x}$ satisfies an equation of the form $\big\{\sum_{i\in I}x_i =w\big\}$ for some $w\in \ZZ$.

If $g=0$, consider the set of all directed trees with $n$ ends. As they are trees, the expansion factors of the bounded edges are determined by the expansion factors of the ends, and they are linear polynomials in the $x_i$. Given a directed tree, it contributes to the count $L_{0}(\mathbf{x})$ in a region bounded by the walls
$$\big\{\sum_{i\in I}x_i =w\big\},$$
if and only if the expansion factors of all edges are positive. Fix a region $B$.
Denote by $m(\Gamma)$ the number of vertex orderings compatible with the edge directions. Let 
$$s(\Gamma,B)=\begin{cases}0 \mbox{ if an expansion factor is negative,}\\ 1 \mbox{ if all expansion factors are positive.} \end{cases}$$
 Thus a directed tree $\Gamma$ contributes with
$$s(\Gamma,B)m(\Gamma)\frac{1}{|\Aut(\Gamma)|}\prod_e \omega(e),$$
where the product goes over the bounded edges $e$ of $\Gamma$ and $\omega(e)$ denotes their expansion factors.
Here, the product $\prod_e \omega(e)$ is a product of linear polynomials in the $x_i$. As there are $n-3$ bounded edges, this product is a polynomial in the $x_i$ of degree $n-3$.
We can see that the contribution of a single directed tree is $0$ or a polynomial in the $x_i$ of degree $n-3$. As $L_{0}(\mathbf{x})$ equals the sum of all contributions of all directed trees with $n$ ends, we obtain a piecewise polynomial. The piecewise structure results from the fact that different graphs yield nonzero contributions for different regions, as the walls are given by equations that set a possible expansion factor of a bounded edge zero. On one side of the wall, the corresponding edge would have to be oriented in a certain way to guarantee a positive expansion factor, on the other side of the wall, its direction would have to be reversed. 

\end{proof}

\begin{example}
Let $g=0$, $n=4$, $x_1,x_2>0$ and $x_3,x_4<0$. We must have $x_1+x_2+x_3+x_4=2$.
There are three ways to arrange the labels to a trivalent tree with four ends. If $x_1$ and $x_2$ are adjacent, there is only one way to orient the bounded edge in such a way that positive expansion factors can be obtained. For the other two possible choices of labels, we have to consider both ways to orient. Thus, there are five oriented trees to consider, as depicted in Figure \ref{fig-fivetrees}.
The expansion factors of the bounded edges which can be positive or negative are given by $x_1+x_3-1$, $x_1+x_4-1$.
Thus, we have to consider only two walls given by $x_1+x_3-1=0$, $x_1+x_4-1=0$. This two walls subdivide the set of definition of the map $\mathbf{x}\mapsto L_{0}(\mathbf{x})$ into four regions. Figure \ref{fig-fivetrees} gives, for each of the four regions, the trees with a nonzero contribution, and the polynomial which equals $L_{0}(\mathbf{x})$.
\begin{figure}
\begin{center}

\tikzset{every picture/.style={line width=0.75pt}} %set default line width to 0.75pt        

\begin{tikzpicture}[x=0.75pt,y=0.75pt,yscale=-1,xscale=1]
%uncomment if require: \path (0,784); %set diagram left start at 0, and has height of 784

%Straight Lines [id:da2621809604430423] 
\draw    (20.25,180.25) -- (66.28,188.88) ;
\draw [shift={(68.25,189.25)}, rotate = 190.62] [color={rgb, 255:red, 0; green, 0; blue, 0 }  ][line width=0.75]    (10.93,-3.29) .. controls (6.95,-1.4) and (3.31,-0.3) .. (0,0) .. controls (3.31,0.3) and (6.95,1.4) .. (10.93,3.29)   ;
%Straight Lines [id:da25982957893129555] 
\draw    (20.75,198.75) -- (66.29,189.64) ;
\draw [shift={(68.25,189.25)}, rotate = 528.69] [color={rgb, 255:red, 0; green, 0; blue, 0 }  ][line width=0.75]    (10.93,-3.29) .. controls (6.95,-1.4) and (3.31,-0.3) .. (0,0) .. controls (3.31,0.3) and (6.95,1.4) .. (10.93,3.29)   ;
%Straight Lines [id:da4362735321140505] 
\draw    (68.25,189.25) -- (148.75,189.74) ;
\draw [shift={(150.75,189.75)}, rotate = 180.35] [color={rgb, 255:red, 0; green, 0; blue, 0 }  ][line width=0.75]    (10.93,-3.29) .. controls (6.95,-1.4) and (3.31,-0.3) .. (0,0) .. controls (3.31,0.3) and (6.95,1.4) .. (10.93,3.29)   ;
%Straight Lines [id:da14291236484484904] 
\draw    (150.75,189.75) -- (187.8,181.68) ;
\draw [shift={(189.75,181.25)}, rotate = 527.7] [color={rgb, 255:red, 0; green, 0; blue, 0 }  ][line width=0.75]    (10.93,-3.29) .. controls (6.95,-1.4) and (3.31,-0.3) .. (0,0) .. controls (3.31,0.3) and (6.95,1.4) .. (10.93,3.29)   ;
%Straight Lines [id:da38894997260353736] 
\draw    (150.75,189.75) -- (186.84,201.15) ;
\draw [shift={(188.75,201.75)}, rotate = 197.53] [color={rgb, 255:red, 0; green, 0; blue, 0 }  ][line width=0.75]    (10.93,-3.29) .. controls (6.95,-1.4) and (3.31,-0.3) .. (0,0) .. controls (3.31,0.3) and (6.95,1.4) .. (10.93,3.29)   ;
%Straight Lines [id:da9823651270561455] 
\draw    (245.75,172.75) -- (288.27,179.44) ;
\draw [shift={(290.25,179.75)}, rotate = 188.94] [color={rgb, 255:red, 0; green, 0; blue, 0 }  ][line width=0.75]    (10.93,-3.29) .. controls (6.95,-1.4) and (3.31,-0.3) .. (0,0) .. controls (3.31,0.3) and (6.95,1.4) .. (10.93,3.29)   ;
%Straight Lines [id:da9917276880914707] 
\draw    (246.25,225.25) -- (304.29,213.64) ;
\draw [shift={(306.25,213.25)}, rotate = 528.69] [color={rgb, 255:red, 0; green, 0; blue, 0 }  ][line width=0.75]    (10.93,-3.29) .. controls (6.95,-1.4) and (3.31,-0.3) .. (0,0) .. controls (3.31,0.3) and (6.95,1.4) .. (10.93,3.29)   ;
%Straight Lines [id:da03843456126229161] 
\draw    (306.25,213.25) -- (344.86,226.6) ;
\draw [shift={(346.75,227.25)}, rotate = 199.07] [color={rgb, 255:red, 0; green, 0; blue, 0 }  ][line width=0.75]    (10.93,-3.29) .. controls (6.95,-1.4) and (3.31,-0.3) .. (0,0) .. controls (3.31,0.3) and (6.95,1.4) .. (10.93,3.29)   ;
%Straight Lines [id:da12384436797215714] 
\draw    (290.25,179.75) -- (340.77,173.01) ;
\draw [shift={(342.75,172.75)}, rotate = 532.4100000000001] [color={rgb, 255:red, 0; green, 0; blue, 0 }  ][line width=0.75]    (10.93,-3.29) .. controls (6.95,-1.4) and (3.31,-0.3) .. (0,0) .. controls (3.31,0.3) and (6.95,1.4) .. (10.93,3.29)   ;
%Straight Lines [id:da6101978738565964] 
\draw    (290.25,179.75) -- (305.39,211.45) ;
\draw [shift={(306.25,213.25)}, rotate = 244.47] [color={rgb, 255:red, 0; green, 0; blue, 0 }  ][line width=0.75]    (10.93,-3.29) .. controls (6.95,-1.4) and (3.31,-0.3) .. (0,0) .. controls (3.31,0.3) and (6.95,1.4) .. (10.93,3.29)   ;
%Straight Lines [id:da7854162834049212] 
\draw    (415.75,174.25) -- (458.27,180.94) ;
\draw [shift={(460.25,181.25)}, rotate = 188.94] [color={rgb, 255:red, 0; green, 0; blue, 0 }  ][line width=0.75]    (10.93,-3.29) .. controls (6.95,-1.4) and (3.31,-0.3) .. (0,0) .. controls (3.31,0.3) and (6.95,1.4) .. (10.93,3.29)   ;
%Straight Lines [id:da5962670433414663] 
\draw    (416.25,226.75) -- (474.29,215.14) ;
\draw [shift={(476.25,214.75)}, rotate = 528.69] [color={rgb, 255:red, 0; green, 0; blue, 0 }  ][line width=0.75]    (10.93,-3.29) .. controls (6.95,-1.4) and (3.31,-0.3) .. (0,0) .. controls (3.31,0.3) and (6.95,1.4) .. (10.93,3.29)   ;
%Straight Lines [id:da11506236137975046] 
\draw    (476.25,214.75) -- (514.86,228.1) ;
\draw [shift={(516.75,228.75)}, rotate = 199.07] [color={rgb, 255:red, 0; green, 0; blue, 0 }  ][line width=0.75]    (10.93,-3.29) .. controls (6.95,-1.4) and (3.31,-0.3) .. (0,0) .. controls (3.31,0.3) and (6.95,1.4) .. (10.93,3.29)   ;
%Straight Lines [id:da019812545100440615] 
\draw    (460.25,181.25) -- (510.77,174.51) ;
\draw [shift={(512.75,174.25)}, rotate = 532.4100000000001] [color={rgb, 255:red, 0; green, 0; blue, 0 }  ][line width=0.75]    (10.93,-3.29) .. controls (6.95,-1.4) and (3.31,-0.3) .. (0,0) .. controls (3.31,0.3) and (6.95,1.4) .. (10.93,3.29)   ;
%Straight Lines [id:da7082232669584333] 
\draw    (460.25,181.25) -- (475.39,212.95) ;
\draw [shift={(476.25,214.75)}, rotate = 244.47] [color={rgb, 255:red, 0; green, 0; blue, 0 }  ][line width=0.75]    (10.93,-3.29) .. controls (6.95,-1.4) and (3.31,-0.3) .. (0,0) .. controls (3.31,0.3) and (6.95,1.4) .. (10.93,3.29)   ;
%Straight Lines [id:da40628292200701777] 
\draw    (233,286.75) -- (292.77,294.98) ;
\draw [shift={(294.75,295.25)}, rotate = 187.84] [color={rgb, 255:red, 0; green, 0; blue, 0 }  ][line width=0.75]    (10.93,-3.29) .. controls (6.95,-1.4) and (3.31,-0.3) .. (0,0) .. controls (3.31,0.3) and (6.95,1.4) .. (10.93,3.29)   ;
%Straight Lines [id:da5800897759269987] 
\draw    (231.5,332.25) -- (272.04,323.66) ;
\draw [shift={(274,323.25)}, rotate = 528.04] [color={rgb, 255:red, 0; green, 0; blue, 0 }  ][line width=0.75]    (10.93,-3.29) .. controls (6.95,-1.4) and (3.31,-0.3) .. (0,0) .. controls (3.31,0.3) and (6.95,1.4) .. (10.93,3.29)   ;
%Straight Lines [id:da8430678820459045] 
\draw    (274,323.25) -- (327.52,330.96) ;
\draw [shift={(329.5,331.25)}, rotate = 188.2] [color={rgb, 255:red, 0; green, 0; blue, 0 }  ][line width=0.75]    (10.93,-3.29) .. controls (6.95,-1.4) and (3.31,-0.3) .. (0,0) .. controls (3.31,0.3) and (6.95,1.4) .. (10.93,3.29)   ;
%Straight Lines [id:da0406429377553309] 
\draw    (294.75,295.25) -- (345.27,288.51) ;
\draw [shift={(347.25,288.25)}, rotate = 532.4100000000001] [color={rgb, 255:red, 0; green, 0; blue, 0 }  ][line width=0.75]    (10.93,-3.29) .. controls (6.95,-1.4) and (3.31,-0.3) .. (0,0) .. controls (3.31,0.3) and (6.95,1.4) .. (10.93,3.29)   ;
%Straight Lines [id:da20953821543097473] 
\draw    (274,323.25) -- (293.56,296.86) ;
\draw [shift={(294.75,295.25)}, rotate = 486.54] [color={rgb, 255:red, 0; green, 0; blue, 0 }  ][line width=0.75]    (10.93,-3.29) .. controls (6.95,-1.4) and (3.31,-0.3) .. (0,0) .. controls (3.31,0.3) and (6.95,1.4) .. (10.93,3.29)   ;
%Straight Lines [id:da8285115093796002] 
\draw    (419.5,289.75) -- (479.27,297.98) ;
\draw [shift={(481.25,298.25)}, rotate = 187.84] [color={rgb, 255:red, 0; green, 0; blue, 0 }  ][line width=0.75]    (10.93,-3.29) .. controls (6.95,-1.4) and (3.31,-0.3) .. (0,0) .. controls (3.31,0.3) and (6.95,1.4) .. (10.93,3.29)   ;
%Straight Lines [id:da8779920143097661] 
\draw    (418,335.25) -- (458.54,326.66) ;
\draw [shift={(460.5,326.25)}, rotate = 528.04] [color={rgb, 255:red, 0; green, 0; blue, 0 }  ][line width=0.75]    (10.93,-3.29) .. controls (6.95,-1.4) and (3.31,-0.3) .. (0,0) .. controls (3.31,0.3) and (6.95,1.4) .. (10.93,3.29)   ;
%Straight Lines [id:da27592585979482953] 
\draw    (460.5,326.25) -- (514.02,333.96) ;
\draw [shift={(516,334.25)}, rotate = 188.2] [color={rgb, 255:red, 0; green, 0; blue, 0 }  ][line width=0.75]    (10.93,-3.29) .. controls (6.95,-1.4) and (3.31,-0.3) .. (0,0) .. controls (3.31,0.3) and (6.95,1.4) .. (10.93,3.29)   ;
%Straight Lines [id:da5772724540149299] 
\draw    (481.25,298.25) -- (531.77,291.51) ;
\draw [shift={(533.75,291.25)}, rotate = 532.4100000000001] [color={rgb, 255:red, 0; green, 0; blue, 0 }  ][line width=0.75]    (10.93,-3.29) .. controls (6.95,-1.4) and (3.31,-0.3) .. (0,0) .. controls (3.31,0.3) and (6.95,1.4) .. (10.93,3.29)   ;
%Straight Lines [id:da4291005231918722] 
\draw    (460.5,326.25) -- (480.06,299.86) ;
\draw [shift={(481.25,298.25)}, rotate = 486.54] [color={rgb, 255:red, 0; green, 0; blue, 0 }  ][line width=0.75]    (10.93,-3.29) .. controls (6.95,-1.4) and (3.31,-0.3) .. (0,0) .. controls (3.31,0.3) and (6.95,1.4) .. (10.93,3.29)   ;
%Straight Lines [id:da15124764344149488] 
\draw    (250.67,389.33) -- (250,712.33) ;
%Straight Lines [id:da9004444218245604] 
\draw    (91.33,550) -- (410.67,550.67) ;

% Text Node
\draw (19.5,162.25) node [anchor=north west][inner sep=0.75pt]   [align=left] {$\displaystyle x_{1}$};
% Text Node
\draw (17.5,196.25) node [anchor=north west][inner sep=0.75pt]   [align=left] {$\displaystyle x_{2}$};
% Text Node
\draw (165.5,198.25) node [anchor=north west][inner sep=0.75pt]   [align=left] {$\displaystyle -x_{4}$};
% Text Node
\draw (168,161.25) node [anchor=north west][inner sep=0.75pt]   [align=left] {$\displaystyle -x_{3}$};
% Text Node
\draw (72,168.25) node [anchor=north west][inner sep=0.75pt]   [align=left] {$\displaystyle x_{1} +x_{2} -1$};
% Text Node
\draw (235.5,152.75) node [anchor=north west][inner sep=0.75pt]   [align=left] {$\displaystyle x_{1}$};
% Text Node
\draw (242.5,201.75) node [anchor=north west][inner sep=0.75pt]   [align=left] {$\displaystyle x_{2}$};
% Text Node
\draw (340,207.75) node [anchor=north west][inner sep=0.75pt]   [align=left] {$\displaystyle -x_{4}$};
% Text Node
\draw (339,155.75) node [anchor=north west][inner sep=0.75pt]   [align=left] {$\displaystyle -x_{3}$};
% Text Node
\draw (301.5,184.75) node [anchor=north west][inner sep=0.75pt]   [align=left] {$\displaystyle x_{1} +x_{3} -1$};
% Text Node
\draw (405.5,154.25) node [anchor=north west][inner sep=0.75pt]   [align=left] {$\displaystyle x_{1}$};
% Text Node
\draw (412.5,203.25) node [anchor=north west][inner sep=0.75pt]   [align=left] {$\displaystyle x_{2}$};
% Text Node
\draw (510,209.25) node [anchor=north west][inner sep=0.75pt]   [align=left] {$\displaystyle -x_{3}$};
% Text Node
\draw (509,157.25) node [anchor=north west][inner sep=0.75pt]   [align=left] {$\displaystyle -x_{4}$};
% Text Node
\draw (471.5,186.25) node [anchor=north west][inner sep=0.75pt]   [align=left] {$\displaystyle x_{1} +x_{4} -1$};
% Text Node
\draw (240,268.25) node [anchor=north west][inner sep=0.75pt]   [align=left] {$\displaystyle x_{1}$};
% Text Node
\draw (224,310.25) node [anchor=north west][inner sep=0.75pt]   [align=left] {$\displaystyle x_{2}$};
% Text Node
\draw (316.06,330.17) node [anchor=north west][inner sep=0.75pt]  [rotate=-0.35] [align=left] {$\displaystyle -x_{4}$};
% Text Node
\draw (343.5,271.25) node [anchor=north west][inner sep=0.75pt]   [align=left] {$\displaystyle -x_{3}$};
% Text Node
\draw (299,299.75) node [anchor=north west][inner sep=0.75pt]   [align=left] {$\displaystyle -x_{1} -x_{3} +1$};
% Text Node
\draw (426.5,271.25) node [anchor=north west][inner sep=0.75pt]   [align=left] {$\displaystyle x_{1}$};
% Text Node
\draw (410.5,313.25) node [anchor=north west][inner sep=0.75pt]   [align=left] {$\displaystyle x_{2}$};
% Text Node
\draw (502.56,333.17) node [anchor=north west][inner sep=0.75pt]  [rotate=-0.35] [align=left] {$\displaystyle -x_{3}$};
% Text Node
\draw (530,274.25) node [anchor=north west][inner sep=0.75pt]   [align=left] {$\displaystyle -x_{4}$};
% Text Node
\draw (485.5,302.75) node [anchor=north west][inner sep=0.75pt]   [align=left] {$\displaystyle -x_{1} -x_{4} +1$};
% Text Node
\draw (84,232) node [anchor=north west][inner sep=0.75pt]   [align=left] {I};
% Text Node
\draw (300,240.67) node [anchor=north west][inner sep=0.75pt]   [align=left] {II};
% Text Node
\draw (481.33,240) node [anchor=north west][inner sep=0.75pt]   [align=left] {III};
% Text Node
\draw (290,352.67) node [anchor=north west][inner sep=0.75pt]   [align=left] {IV};
% Text Node
\draw (472.67,352) node [anchor=north west][inner sep=0.75pt]   [align=left] {V};
% Text Node
\draw (80.83,380.08) node [anchor=north west][inner sep=0.75pt]   [align=left] {$\displaystyle x_{1} +x_{3} -1 >0$};
% Text Node
\draw (82.83,404.25) node [anchor=north west][inner sep=0.75pt]   [align=left] {$\displaystyle x_{1} +x_{4} -1 >0$};
% Text Node
\draw (282.83,404.25) node [anchor=north west][inner sep=0.75pt]   [align=left] {$\displaystyle x_{1} +x_{4} -1< 0$};
% Text Node
\draw (72.83,581.58) node [anchor=north west][inner sep=0.75pt]   [align=left] {$\displaystyle x_{1} +x_{4} -1 >0$};
% Text Node
\draw (280.83,379.42) node [anchor=north west][inner sep=0.75pt]   [align=left] {$\displaystyle x_{1} +x_{3} -1 >0$};
% Text Node
\draw (73.5,558.75) node [anchor=north west][inner sep=0.75pt]   [align=left] {$\displaystyle x_{1} +x_{3} -1< 0$};
% Text Node
\draw (280.83,560.75) node [anchor=north west][inner sep=0.75pt]   [align=left] {$\displaystyle x_{1} +x_{3} -1< 0$};
% Text Node
\draw (280.17,584.92) node [anchor=north west][inner sep=0.75pt]   [align=left] {$\displaystyle x_{1} +x_{4} -1< 0$};
% Text Node
\draw (108,448.33) node [anchor=north west][inner sep=0.75pt]   [align=left] {I,II,III};
% Text Node
\draw (286.67,446.33) node [anchor=north west][inner sep=0.75pt]   [align=left] {I,II,V};
% Text Node
\draw (114.67,625.67) node [anchor=north west][inner sep=0.75pt]   [align=left] {I,IV,III};
% Text Node
\draw (288,624.33) node [anchor=north west][inner sep=0.75pt]   [align=left] {I,IV,IV};
% Text Node
\draw (66.67,492.67) node [anchor=north west][inner sep=0.75pt]   [align=left] {$\displaystyle 3x_{1} +x_{2} +x_{3} +x_{4} -3$};
% Text Node
\draw (267.33,491.33) node [anchor=north west][inner sep=0.75pt]   [align=left] {$\displaystyle x_{1} +x_{2} +x_{3} -x_{4} -1$};
% Text Node
\draw (74.67,657.33) node [anchor=north west][inner sep=0.75pt]   [align=left] {$\displaystyle x_{1} +x_{2} -x_{3} +x_{4} -1$};
% Text Node
\draw (271.33,656) node [anchor=north west][inner sep=0.75pt]   [align=left] {$\displaystyle -x_{1} +x_{2} -x_{3} -x_{4} +1$};

\end{tikzpicture}

\end{center}
\caption{The five directed trees we need to consider for $g=0$, $n=4$, $x_1,x_2>0$, $x_3,x_4<0$. Below, the four regions of polynomiality with the trees that yield a nonzero contribution, and the polynomial which equals $L_{0}(\mathbf{x})$ in this region.}\label{fig-fivetrees}
\end{figure}
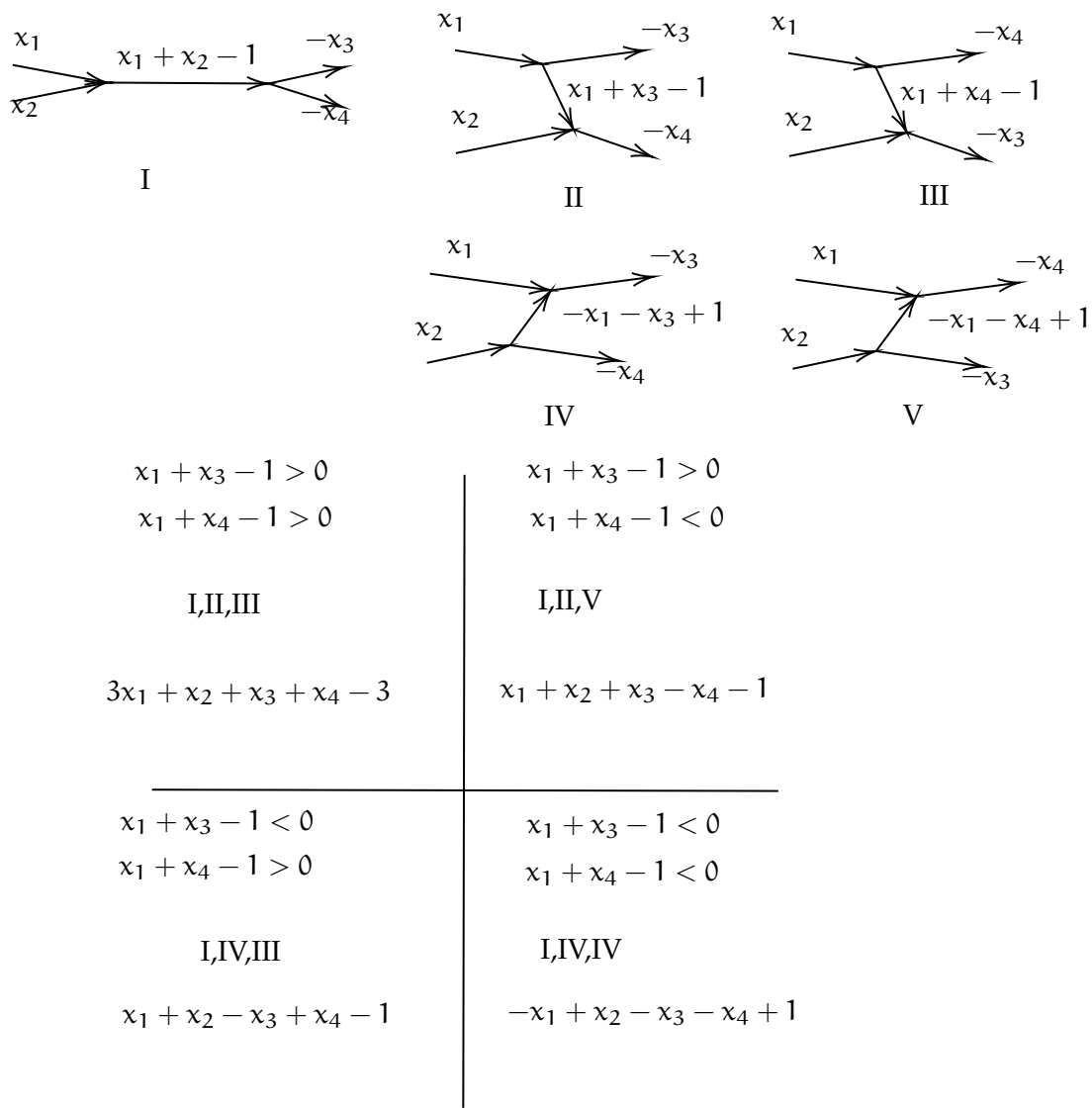

\end{example}

\begin{example}

Let $g=1$ and $\mathbf{x}$ consist of one positive and one negative entry. Denote the positive entry by $x$, then the negative entry must be $-(x-2)$, as their sum equals $n-2+2g=2-2+2\cdot 1$. We must have $x>0$ and $-(x-2)<0$, hence $x>2$. The set of definition of the map $\mathbf{x}\mapsto \leakyn$ thus equals $\{x \in \ZZ|\; x>2\}$.

\begin{figure}
\begin{center}

\tikzset{every picture/.style={line width=0.75pt}} %set default line width to 0.75pt        

\begin{tikzpicture}[x=0.75pt,y=0.75pt,yscale=-1,xscale=1]
%uncomment if require: \path (0,451); %set diagram left start at 0, and has height of 451

%Straight Lines [id:da784339877455535] 
\draw    (29.6,280) -- (80.4,280.4) ;
%Curve Lines [id:da6046417434992952] 
\draw    (80.4,280.4) .. controls (96.4,264.8) and (108.4,266) .. (120,280.4) ;
%Curve Lines [id:da4926924651995884] 
\draw    (80.4,280.4) .. controls (94.4,294.8) and (103.2,296.8) .. (120,280.4) ;
%Straight Lines [id:da8258769570087142] 
\draw    (120,280.4) -- (170.8,280.8) ;
%Shape: Free Drawing [id:dp8169472547665438] 
\draw  [color={rgb, 255:red, 0; green, 0; blue, 0 }  ][line width=0.75] [line join = round][line cap = round] (51.2,274) .. controls (52.8,275.33) and (54.28,276.83) .. (56,278) .. controls (57.23,278.84) and (60.87,278.79) .. (60,280) .. controls (57.94,282.86) and (53.87,283.47) .. (50.8,285.2) ;
%Shape: Free Drawing [id:dp24107582762779667] 
\draw  [color={rgb, 255:red, 0; green, 0; blue, 0 }  ][line width=0.75] [line join = round][line cap = round] (105.6,264.8) .. controls (106.27,267.07) and (108.66,269.49) .. (107.6,271.6) .. controls (106.5,273.79) and (103.07,273.47) .. (100.8,274.4) ;
%Shape: Free Drawing [id:dp6903252384605195] 
\draw  [color={rgb, 255:red, 0; green, 0; blue, 0 }  ][line width=0.75] [line join = round][line cap = round] (108.4,284) .. controls (110.93,284.27) and (114.36,282.85) .. (116,284.8) .. controls (117.48,286.56) and (115.2,289.33) .. (114.8,291.6) ;
%Shape: Free Drawing [id:dp30877804655663155] 
\draw  [color={rgb, 255:red, 0; green, 0; blue, 0 }  ][line width=0.75] [line join = round][line cap = round] (134.4,276.8) .. controls (136.78,277.92) and (143.6,279.73) .. (141.2,280.8) .. controls (138.85,281.85) and (134,281.02) .. (134,283.6) ;

% Text Node
\draw (34,262.8) node [anchor=north west][inner sep=0.75pt]   [align=left] {$\displaystyle x$};
% Text Node
\draw (140.8,262.4) node [anchor=north west][inner sep=0.75pt]   [align=left] {$\displaystyle x-2$};
% Text Node
\draw (95.6,251.6) node [anchor=north west][inner sep=0.75pt]   [align=left] {$\displaystyle i$};
% Text Node
\draw (76.8,291.6) node [anchor=north west][inner sep=0.75pt]   [align=left] {$\displaystyle x-1-i$};

\end{tikzpicture}

\end{center}

\caption{The $\mathbf{x}$-graph of genus $1$ with one in- and one out-end, with reference orientation.}\label{fig-egg}

\end{figure}
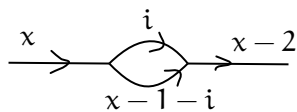

Figure \ref{fig-egg} shows the only $\mathbf{x}$-graph contributing to the count.
The hyperplane arrangement given by the expansion factors of the bounded edges is depicted in Figure \ref{fig-hyperplaneegg}.

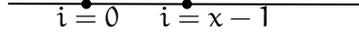
\begin{figure}
\begin{center}

\tikzset{every picture/.style={line width=0.75pt}} %set default line width to 0.75pt        

\begin{tikzpicture}[x=0.75pt,y=0.75pt,yscale=-1,xscale=1]
%uncomment if require: \path (0,451); %set diagram left start at 0, and has height of 451

%Straight Lines [id:da20497365963369985] 
\draw    (30.4,290.4) -- (210.8,290.8) ;
%Shape: Circle [id:dp09178240549661498] 
\draw  [fill={rgb, 255:red, 0; green, 0; blue, 0 }  ,fill opacity=1 ] (67.85,290.55) .. controls (67.85,289.45) and (68.75,288.55) .. (69.85,288.55) .. controls (70.95,288.55) and (71.85,289.45) .. (71.85,290.55) .. controls (71.85,291.65) and (70.95,292.55) .. (69.85,292.55) .. controls (68.75,292.55) and (67.85,291.65) .. (67.85,290.55) -- cycle ;
%Shape: Circle [id:dp7006889568339706] 
\draw  [fill={rgb, 255:red, 0; green, 0; blue, 0 }  ,fill opacity=1 ] (118.6,290.6) .. controls (118.6,289.5) and (119.5,288.6) .. (120.6,288.6) .. controls (121.7,288.6) and (122.6,289.5) .. (122.6,290.6) .. controls (122.6,291.7) and (121.7,292.6) .. (120.6,292.6) .. controls (119.5,292.6) and (118.6,291.7) .. (118.6,290.6) -- cycle ;

% Text Node
\draw (53.25,290.13) node [anchor=north west][inner sep=0.75pt]   [align=left] {$\displaystyle i=0$};
% Text Node
\draw (105.25,290.13) node [anchor=north west][inner sep=0.75pt]   [align=left] {$\displaystyle i=x-1$};

\end{tikzpicture}

\end{center}
\caption{The hyperplane arrangement of expansion factors for the $\mathbf{x}$-graph in Figure \ref{fig-egg}.} \label{fig-hyperplaneegg}

\end{figure}

There is one bounded chamber $A$. The sign with which the lattice points in there contribute is one, as is the number of vertex orderings compatible with the edge orientations.

Thus $$\leakyn= \sum_{i=1}^{x-2} i\cdot (x-1-i) = \frac{1}{6}(x-1)(x-2)x.$$

As the hyperplane arrangement never changes its topology within the set of definition, the map is polynomial in this example.
\end{example}

\subsection{Counts of leaky covers as matrix elements on a Fock space}\label{sec:fs}
The bosonic Heisenberg algebra $\mathcal{H}$ is the Lie algebra with basis $a_n, n\in \mathbb{Z}$ satisfying commutator relations
\begin{equation} \label{commut}
[a_n, a_m]=(n\cdot \delta_{n,-m})a_0 \end{equation}
where $\delta_{n,-m}$ is the Kronecker symbol.% Note that $a_0$ is central, and is often denoted $c$ in other sources. 

The bosonic Fock space $F$ is a particular representation of $\mathcal{H}$.  It is generated by a single ``vacuum vector'' $v_\emptyset$.  The positive generators annihilate $v_\emptyset$: $a_n\cdot v_\emptyset=0$ for $n>0$; $a_0$ acts as the identity, and the negative operators act freely.  That is, $F$ has a basis $b_\mu$ indexed by partitions, where 
\begin{equation}\label{expand}
b_\mu=a_{-\mu_1}\dots a_{-\mu_m}\cdot v_{\emptyset}
\end{equation}

%In more familiar terms, the bosonic Fock space is isomorphic to a polynomial ring in infinitely many variables, $p_n, n>0$. The vacuum vector $v_\emptyset$ is the element $1\in \mathbb{C}[p_1,p_2,\dots]$, and the operators $a_n$ act as follows:
%$$
%a_n=\left\{\begin{array}{ll}  p_{-n}\cdot & n<0 \\ 
%n\frac{\partial}{\partial p_n} & n>0 \end{array}\right.
%$$
%The vector $b_\mu$ is mapped to the monomial $p_{\mu_1}\cdots p_{\mu_m}$ under this isomorphism.

 We define an inner product on $F$ by declaring $\langle v_\emptyset | v_\emptyset \rangle=1$ and $a_n$ to be the adjoint of $a_{-n}$. Thus we have 
\begin{equation} \label{Pair} \langle b_{\mu}| b_{\nu} \rangle = \mathfrak{z}(\mu) \delta_{\mu,\nu}.\end{equation}
where
$$\mathfrak{z}(\mu)=|\Aut(\mu)|\cdot\prod\mu_i$$
is the size of the centralizer of an element of cycle type $\mu$ in $S_{|\mu|}$.

%Note that being a representation of a Lie algebra is equivalent to being equivalent to a representation of the universal enveloping algebra of that Lie algebra.  All this means in our case is that products of elements of $\mathcal{H}$ act on $F$, not just commutators of these elements.

Following standard conventions, we write 
$\langle \alpha|A|\beta\rangle$ for $\langle \alpha|A\beta\rangle$ for $\alpha, \beta \in F$ and an operator $A$ that is a product of elements $\mathcal{H}$. Such expressions are referred to as \emph{matrix elements}. We also write $\langle A \rangle$ for $\langle v_\emptyset |A|v_\emptyset \rangle$, such a value is called a \emph{vacuum expectation}.

Our goal now is to write counts of leaky covers as  bosonic matrix elements.  This is largely a matter of rephrasing results from \cite{OP06, BG14}.  

%The normal ordering product sorts a product of operators $a_{k_i}$ so that the $k_i$ are decreasing.  
%
%\begin{definition}
%Let $k_i, 1\leq i \leq n$ be any integers, and let $\ell_i, 1\leq i\leq n$ be the same multiset of integers sorted in decreasing order, $\ell_1\geq \ell_2\geq \cdots \geq \ell_n$.  The \emph{normal ordered product} of the $a_{k_i}$, denoted $:\prod_{i=1}^n a_k:$, is defined by
%
%$$:\prod_{i=1}^n a_{k_i} :\;\; = \prod_{i=1}^n a_{\ell_i}$$
%\end{definition}

\begin{definition} The \emph{$k$-leaky cut-join} operator is defined by:
\begin{equation}
\label{caj}M_k %=\sum_{i+j+k=0}\frac{1}{6} :a_i a_j a_k:\;\; 
= \frac{1}{2} \sum_{\substack{0< i\leq j }} a_{-j}a_{-i}a_{i+j-k}+a_{-i-j-k}a_{i}a_{j}  - \frac{1}{24}a_{-k}
\end{equation}
\end{definition}

\begin{theorem}\label{thm-Fock}
The count of leaky covers can be expressed as a matrix element:
$$ \leakyn=\frac{1}{\mathfrak{z}(\mathbf{x}^+)}\frac{1}{\mathfrak{z}(\mathbf{x}^-)}\langle b_{\mathbf{x}^+} |M_k^{n-2+2g} |b_{-\mathbf{x}^-}\rangle.$$
\end{theorem}

The proof of Theorem \ref{thm-Fock} is  analgous to the proof of~\cite[Theorem 5.3.4]{CJMR}, and most of the technicalities involved are standard in the study of operator algebras; we provide an outline of a proof, noting how to adapt the results in loc. cit.

The matrix element $\langle b_{\mathbf{x}^+} |M_k^{n-2+2g} |b_{-\mathbf{x}^-}\rangle$  can be computed as a weighted sum over a set of  Feynman graphs $\Gamma$, which are naturally in bijection with the monomials $m_\Gamma$ in a formal expansion of   $M_k^{n-2+2g}$ that contribute to the matrix element, i.e. such that    $\langle  b_{\mathbf{x}^+} |m_\Gamma |b_{-\mathbf{x}^-}\rangle \not=0$. 

The relevant Feynman graphs arise as follows. Every monomial  in $M_k$ is associated to a {\it Feynman fragment}: a vertex with germs of edges  labeled and oriented according to the operators appearing in the monomial;  the two vectors $b_{\mathbf{x}^+}$ and $b_{-\mathbf{x}^-}$ are associated to labelled ends, see Figure \ref{fig-Feynman} for an illustration with an explicit monomial and its associated fragment.

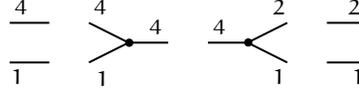
\begin{figure}
\begin{center}

\tikzset{every picture/.style={line width=0.75pt}} %set default line width to 0.75pt        

\begin{tikzpicture}[x=0.75pt,y=0.75pt,yscale=-1,xscale=1]
%uncomment if require: \path (0,784); %set diagram left start at 0, and has height of 784

%Shape: Circle [id:dp06001031495372244] 
\draw  [fill={rgb, 255:red, 0; green, 0; blue, 0 }  ,fill opacity=1 ] (208.75,440.04) .. controls (208.75,439.17) and (209.46,438.46) .. (210.33,438.46) .. controls (211.21,438.46) and (211.92,439.17) .. (211.92,440.04) .. controls (211.92,440.92) and (211.21,441.63) .. (210.33,441.63) .. controls (209.46,441.63) and (208.75,440.92) .. (208.75,440.04) -- cycle ;
%Shape: Circle [id:dp3856850866046483] 
\draw  [fill={rgb, 255:red, 0; green, 0; blue, 0 }  ,fill opacity=1 ] (268.75,439.92) .. controls (268.75,439.04) and (269.46,438.33) .. (270.33,438.33) .. controls (271.21,438.33) and (271.92,439.04) .. (271.92,439.92) .. controls (271.92,440.79) and (271.21,441.5) .. (270.33,441.5) .. controls (269.46,441.5) and (268.75,440.79) .. (268.75,439.92) -- cycle ;
%Straight Lines [id:da45296044104424504] 
\draw    (210,440) -- (230,440) ;
%Straight Lines [id:da890849052841217] 
\draw    (250,440) -- (270,440) ;
%Straight Lines [id:da8726264786084266] 
\draw    (190,430) -- (210,440) ;
%Straight Lines [id:da1377829845843328] 
\draw    (190,450) -- (210,440) ;
%Straight Lines [id:da45547977240923887] 
\draw    (270,440) -- (290,430) ;
%Straight Lines [id:da16774276157770818] 
\draw    (270,440) -- (290,450) ;
%Straight Lines [id:da013533457365263568] 
\draw    (150,430) -- (170,430) ;
%Straight Lines [id:da7396477856040038] 
\draw    (150,450) -- (170,450) ;
%Straight Lines [id:da16452251242467963] 
\draw    (310,430) -- (330,430) ;
%Straight Lines [id:da9921459368266342] 
\draw    (310,450) -- (330,450) ;

% Text Node
\draw (149,452) node [anchor=north west][inner sep=0.75pt]  [font=\scriptsize] [align=left] {$\displaystyle 1$};
% Text Node
\draw (151,417) node [anchor=north west][inner sep=0.75pt]  [font=\scriptsize] [align=left] {$\displaystyle 4$};
% Text Node
\draw (191,417) node [anchor=north west][inner sep=0.75pt]  [font=\scriptsize] [align=left] {$\displaystyle 4$};
% Text Node
\draw (192,453) node [anchor=north west][inner sep=0.75pt]  [font=\scriptsize] [align=left] {$\displaystyle 1$};
% Text Node
\draw (219,427) node [anchor=north west][inner sep=0.75pt]  [font=\scriptsize] [align=left] {$\displaystyle 4$};
% Text Node
\draw (251,427) node [anchor=north west][inner sep=0.75pt]  [font=\scriptsize] [align=left] {$\displaystyle 4$};
% Text Node
\draw (281,417) node [anchor=north west][inner sep=0.75pt]  [font=\scriptsize] [align=left] {$\displaystyle 2$};
% Text Node
\draw (281,452) node [anchor=north west][inner sep=0.75pt]  [font=\scriptsize] [align=left] {$\displaystyle 1$};
% Text Node
\draw (319,417) node [anchor=north west][inner sep=0.75pt]  [font=\scriptsize] [align=left] {$\displaystyle 2$};
% Text Node
\draw (321,452) node [anchor=north west][inner sep=0.75pt]  [font=\scriptsize] [align=left] {$\displaystyle 1$};

\end{tikzpicture}

\end{center}
\caption{The augmented  ordered list of Feynman fragments for the monomial $(a_{-4}a_{-1}a_{4})(a_{-4} a_{2}a_{1})$  contributing to the vacuum expectation $\langle b_{(4,1)} |M_1^{2} |b_{(2,1)}\rangle $ which equals the count of $1$-leaky tropical covers $L_0((4,1,-2,-1), 1)$. The way the fragments are drawn also suggests the unique pairing of germs of edges that  gives rise to the Feynman graph.
}\label{fig-Feynman}
\end{figure}

Any monomial in the formal expansion of $M_k^{n-2+2g}$ then  naturally corresponds to an ordered list of $n-2+2g$ Feynman fragments, which we augment by adding the labelled ends on either side. A monomial contributes to the matrix element  if and only if the two following conditions are verified:
\begin{enumerate}
\item each operator with a positive index $a_{j}$ is paired up with a term $a_{-j}$  that appears on its right, possibly as part of the expansion of  $b_{-\mathbf{x}^-}$ as in \eqref{expand}: the commutator relations \eqref{commut} can be used as a way to move each positive $a_j$ to the right, and if $a_j$ can get past all other operators and thus act on the vacuum vector, we obtain zero.
\item each operator with a negative index $a_{j}$ which is unpaired after the above step can pair with an $a_{-j}$  part  of the expansion of $b_{\mathbf{x}^+}$: otherwise the inner product \eqref{Pair}  vanishes. 
\end{enumerate}
 Each nonzero contribution thus corresponds to a way to connect pairs of edge germs and obtain a global Feynamm graph from the (augmented) fragments, see Example \ref{ex-Feynman}.

The process just described can be viewed as a variant of Wick's theorem \cite{Wic50}, see also~\cite[Proposition 5.2]{BG14}. In the situation of Theorem \ref{thm-Fock}, the Feynman graphs in question are the leaky tropical covers defined in \ref{def-leakygraph}. The proof of Theorem~\ref{thm-Fock} then follows from Theorem \ref{thm: leaky-formula}.
This is analogous to how Theorem 5.3.4 in \cite{CJMR} is deduced from an expression in terms of a sum over graphs (see Definition 3.1.1 and Theorem 3.1.2 in \cite{CJMR}) using Wick's theorem, see Proposition 5.4.3 in \cite{CJMR}.

\begin{example}\label{ex-Feynman}
Consider the count of $1$-leaky covers of genus $0$ with left ends of weight $4$ and $1$ and right ends of weight $2$ and $1$, $L_0((4,1,-2,-1), 1)$. By Theorem \ref{thm-Fock}, it equals the vacuum expectation
$\langle b_{(4,1)} |M_1^{2} |b_{(2,1)}\rangle = \langle a_{-4}a_{-1}v_\emptyset |M_1^{2} |a_{-2}a_{-1}v_\emptyset\rangle = \langle v_\emptyset |a_1a_4\cdot M_k^{2} |a_{-2}a_{-1}v_\emptyset\rangle$.
The summand $a_{-4}a_{-1}a_{4}$ appears in $M_k$, as well as $a_{-4}a_2a_1$. Figure \ref{fig-Feynman} depicts the edge germs we draw for the monomial in $a_i$ corresponding to this choice of summands. As we can connect the edge germs to form a leaky tropical cover, the monomial yields a nonzero contribution to the vacuum expectation. Since all nonzero contributions arise in this way, we obtain the equality stated in Theorem \ref{thm-Fock}.

\end{example}

\bibliographystyle{siam} 
\bibliography{LeakyCovers} 

\begin{thebibliography}{10}

\bibitem{ACFW}
{\sc D.~Abramovich, C.~Cadman, B.~Fantechi, and J.~Wise}, {\em Expanded
  degenerations and pairs}, Comm. Algebra, 41 (2013), pp.~2346--2386.

\bibitem{ACP}
{\sc D.~Abramovich, L.~Caporaso, and S.~Payne}, {\em The tropicalization of the
  moduli space of curves}, Ann. Sci. {\'E}c. Norm. Sup{\'e}r., 48 (2015),
  pp.~765--809.

\bibitem{ACGS15}
{\sc D.~Abramovich, Q.~Chen, M.~Gross, and B.~Siebert}, {\em Decomposition of
  degenerate {G}romov-{W}itten invariants}, Compos. Math., 156 (2020),
  pp.~2020--2075.

\bibitem{AK00}
{\sc D.~Abramovich and K.~Karu}, {\em Weak semistable reduction in
  characteristic 0}, Invent. Math., 139 (2000), pp.~241--273.

\bibitem{AW}
{\sc D.~Abramovich and J.~Wise}, {\em {Birational invariance in logarithmic
  Gromov--Witten theory}}, Comp. Math., 154 (2018), pp.~595--620.

\bibitem{AB14}
{\sc F.~Ardila and E.~Brugall\'e}, {\em The double {G}romov-{W}itten invariants
  of {H}irzebruch surfaces are piecewise polynomial}, Int. Math. Res. Not.
  IMRN,  (2017), pp.~614--641.

\bibitem{BHPSS}
{\sc Y.~Bae, D.~Holmes, R.~Pandharipande, J.~Schmitt, and R.~Schwarz}, {\em
  Pixton's formula and {A}bel-{J}acobi theory on the {P}icard stack}, Acta
  Math., 230 (2023), pp.~205--319.

\bibitem{BS22}
{\sc Y.~Bae, J.~Schmitt, and J.~Skowera}, {\em Chow rings of stacks of
  prestable curves {I}}, Forum Math. Sigma, 10 (2022), p.~e28.

\bibitem{BJ}
{\sc M.~Baker and D.~Jensen}, {\em Degeneration of Linear Series from the
  Tropical Point of View and Applications}, Springer International Publishing,
  Cham, 2016, pp.~365--433.

\bibitem{BNR22}
{\sc L.~Battistella, N.~Nabijou, and D.~Ranganathan}, {\em {Gromov-Witten
  theory via roots and logarithms}}, arXiv:2203.17224,  (2022).

\bibitem{BBM}
{\sc B.~Bertrand, E.~Brugall{\'e}, and G.~Mikhalkin}, {\em Tropical open
  {H}urwitz numbers}, Rend. Semin. Mat. Univ. Padova, 125 (2011), pp.~157--171.

\bibitem{BG14}
{\sc F.~Block and L.~G{\"o}ttsche}, {\em Fock spaces and refined {S}everi
  degrees}, Int. Math. Res. Not.,  (2015), p.~rvn355.

\bibitem{BDKLM}
{\sc G.~Borot, N.~Do, M.~Karev, D.~Lewa{\'n}ski, and E.~Moskovsky}, {\em
  {Double Hurwitz numbers: polynomiality, topological recursion and
  intersection theory}}, arXiv:2002.00900,  (2020).

\bibitem{BSSZ}
{\sc A.~Buryak, S.~Shadrin, L.~Spitz, and D.~Zvonkine}, {\em Integrals of
  $\psi$-classes over double ramification cycles}, Amer. J. Math., 137 (2015),
  pp.~699--737.

\bibitem{Cav16}
{\sc R.~Cavalieri}, {\em Hurwitz theory and the double ramification cycle},
  Japanese Journal of Mathematics, 11 (2016), pp.~305--331.

\bibitem{CCUW}
{\sc R.~Cavalieri, M.~Chan, M.~Ulirsch, and J.~Wise}, {\em A moduli stack of
  tropical curves}, {Forum Math. Sigma}, 8 (2020), pp.~1--93.

\bibitem{CJM1}
{\sc R.~Cavalieri, P.~Johnson, and H.~Markwig}, {\em Tropical {H}urwitz
  numbers}, J. Algebraic Combin., 32 (2010), pp.~241--265.

\bibitem{CJM2}
\leavevmode\vrule height 2pt depth -1.6pt width 23pt, {\em Wall crossings for
  double {H}urwitz numbers}, Adv. Math., 228 (2011), pp.~1894--1937.

\bibitem{CJMR}
{\sc R.~Cavalieri, P.~Johnson, H.~Markwig, and D.~Ranganathan}, {\em A
  graphical interface for the {G}romov-{W}itten theory of curves}, in Algebraic
  geometry: {S}alt {L}ake {C}ity 2015, vol.~97 of Proc. Sympos. Pure Math.,
  Amer. Math. Soc., Providence, RI, 2018, pp.~139--167.

\bibitem{cm:dhn}
{\sc R.~Cavalieri and S.~Marcus}, {\em A geometric perspective on the
  polynomiality of double {H}urwitz numbers}, Canad. Math. Bull.,  (2014).

\bibitem{CMR14a}
{\sc R.~Cavalieri, H.~Markwig, and D.~Ranganathan}, {\em {Tropicalizing the
  space of admissible covers}}, {Math. Ann.}, 364 (2016), pp.~1275--1313.

\bibitem{CFPU}
{\sc M.-W. Cheung, L.~Fantini, J.~Park, and M.~Ulirsch}, {\em Faithful
  realizability of tropical curves}, Int. Math. Res. Not.,  (2015), p.~rnv269.

\bibitem{CSS21}
{\sc M.~Costantini, A.~Sauvaget, and J.~Schmitt}, {\em Integrals of
  $\psi$-classes on twisted double ramification cycles and spaces of
  differentials}, arXiv:2112.04238,  (2021).

\bibitem{DL}
{\sc N.~Do and D.~Lewa{\'n}ski}, {\em On the {Goulden-Jackson-Vakil} conjecture
  for double {H}urwitz numbers}, arXiv:2003.08043,  (2020).

\bibitem{FP03}
{\sc C.~Faber and R.~Pandharipande}, {\em Hodge integrals, partition matrices,
  and the $\lambda_g$ conjecture}, Ann. of Math.,  (2003), pp.~97--124.

\bibitem{FP18}
{\sc G.~Farkas and R.~Pandharipande}, {\em The moduli space of twisted
  canonical divisors}, J. Inst. Math. Jussieu, 17 (2018), pp.~615--672.

\bibitem{GKM07}
{\sc A.~Gathmann, M.~Kerber, and H.~Markwig}, {\em Tropical fans and the moduli
  space of rational tropical curves}, Compos. Math., 145 (2009), pp.~173--195.

\bibitem{GJV}
{\sc I.~Goulden, D.~Jackson, and R.~Vakil}, {\em Towards the geometry of double
  {H}urwitz numbers}, Adv. Math., 198 (2005), pp.~43--92.

\bibitem{GV05}
{\sc T.~Graber and R.~Vakil}, {\em Relative virtual localization and vanishing
  of tautological classes on moduli spaces of curves}, Duke Math. J., 130
  (2005), pp.~1--37.

\bibitem{HMY10}
{\sc C.~Haase, G.~Musiker, and J.~Yu}, {\em Linear systems on tropical curves},
  in 22nd {I}nternational {C}onference on {F}ormal {P}ower {S}eries and
  {A}lgebraic {C}ombinatorics ({FPSAC} 2010), Discrete Math. Theor. Comput.
  Sci. Proc., AN, Assoc. Discrete Math. Theor. Comput. Sci., Nancy, 2010,
  pp.~295--306.

\bibitem{Herr}
{\sc L.~Herr}, {\em The log product formula}, Algebra Number Theory, 17 (2023),
  pp.~1281--1323.

\bibitem{Hol17}
{\sc D.~Holmes}, {\em Extending the double ramification cycle by resolving the
  {A}bel-{J}acobi map}, J. Inst. Math. Jussieu, 20 (2021), pp.~331--359.

\bibitem{HMPPS}
{\sc D.~Holmes, S.~Molcho, R.~Pandharipande, A.~Pixton, and J.~Schmitt}, {\em
  Logarithmic double ramification cycles}, arXiv:2207.06778,  (2022).

\bibitem{HPS19}
{\sc D.~Holmes, A.~Pixton, and J.~Schmitt}, {\em Multiplicativity of the double
  ramification cycle}, Doc. Math., 24 (2019), pp.~545--562.

\bibitem{HS21}
{\sc D.~Holmes and R.~Schwarz}, {\em Logarithmic intersections of double
  ramification cycles}, Algebr. Geom., 9 (2022), pp.~574--605.

\bibitem{JPPZ}
{\sc F.~Janda, R.~Pandharipande, A.~Pixton, and D.~Zvonkine}, {\em Double
  ramification cycles on the moduli spaces of curves}, Publ. Math. IH{\'E}S,
  125 (2017), pp.~221--266.

\bibitem{JPPZ2}
\leavevmode\vrule height 2pt depth -1.6pt width 23pt, {\em Double ramification
  cycles with target varieties}, J. Topol., 13 (2020), pp.~1725--1766.

\bibitem{Li02}
{\sc J.~Li}, {\em {A degeneration formula of GW-invariants}}, J. Diff. Geom.,
  60 (2002), pp.~199--293.

\bibitem{MW17}
{\sc S.~Marcus and J.~Wise}, {\em Logarithmic compactification of the
  {A}bel-{J}acobi section}, Proc. Lond. Math. Soc. (3), 121 (2020),
  pp.~1207--1250.

\bibitem{MR20}
{\sc D.~Maulik and D.~Ranganathan}, {\em {Logarithmic Donaldson-Thomas
  theory}}, arXiv:2006.06603,  (2020).

\bibitem{Mi03}
{\sc G.~Mikhalkin}, {\em Enumerative tropical geometry in {${\mathbb{R}^2}$}},
  J. Amer. Math. Soc, 18 (2005), pp.~313--377.

\bibitem{Mol16}
{\sc S.~Molcho}, {\em Universal stacky semistable reduction}, Israel J. Math.,
  242 (2021), pp.~55--82.

\bibitem{Mol22}
\leavevmode\vrule height 2pt depth -1.6pt width 23pt, {\em Smooth
  compactifications of the {A}bel-{J}acobi section}, Forum Math. Sigma, 11
  (2023), pp.~Paper No. e88, 35.

\bibitem{MPS21}
{\sc S.~Molcho, R.~Pandharipande, and J.~Schmitt}, {\em The {H}odge bundle, the
  universal 0-section, and the log {C}how ring of the moduli space of curves},
  Compos. Math., 159 (2023), pp.~306--354.

\bibitem{MR21}
{\sc S.~Molcho and D.~Ranganathan}, {\em {A case study of intersections on
  blowups of the moduli of curves}}, arXiv:2106.15194,  (2021).

\bibitem{NS06}
{\sc T.~Nishinou and B.~Siebert}, {\em Toric degenerations of toric varieties
  and tropical curves}, Duke Math. J., 135 (2006), pp.~1--51.

\bibitem{Oes19}
{\sc J.~Oesinghaus}, {\em {Quasisymmetric functions and the Chow ring of the
  stack of expanded pairs}}, Res. Math. Sci., 6 (2019), p.~5.

\bibitem{OP06}
{\sc A.~Okounkov and R.~Pandharipande}, {\em Gromov-{W}itten theory, {H}urwitz
  theory, and completed cycles}, Ann. of Math., 163 (2006), pp.~517--560.

\bibitem{R19b}
{\sc D.~Ranganathan}, {\em A note on cycles of curves in a product of pairs},
  arXiv:1910.00239,  (2019).

\bibitem{R19}
\leavevmode\vrule height 2pt depth -1.6pt width 23pt, {\em Logarithmic
  {G}romov-{W}itten theory with expansions}, Algebr. Geom., 9 (2022),
  pp.~714--761.

\bibitem{RUK22}
{\sc D.~Ranganathan and A.~U. Kumaran}, {\em Logarithmic gromov-witten theory
  and double ramification cycles}, arXiv:2212.11171,  (2022).

\bibitem{RSW17A}
{\sc D.~Ranganathan, K.~Santos-Parker, and J.~Wise}, {\em Moduli of stable maps
  in genus one and logarithmic geometry, {I}}, Geom. Topol., 23 (2019),
  pp.~3315--3366.

\bibitem{Shad:pc}
{\sc S.~Shadrin}, {\em Personal communication - {L}eysin conference: recent
  advances on moduli spaces of curves.}, 2022.

\bibitem{STR22}
{\sc P.~Shimpi, D.~Townsend, and D.~Ranganathan}, {\em {Quasisymmetric
  functions, Chow rings, and moduli of tropical expansions}}, In preparation.

\bibitem{UZ19}
{\sc M.~Ulirsch and D.~Zakharov}, {\em Tropical double ramification loci},
  arXiv:1910.01499,  (2019).

\bibitem{Vi89}
{\sc A.~Vistoli}, {\em Intersection theory on algebraic stacks and on their
  moduli spaces}, Invent. Math., 97 (1989), pp.~613--670.

\bibitem{Wic50}
{\sc G.~C. Wick}, {\em The evaluation of the collision matrix}, Physical Rev.
  (2), 80 (1950), pp.~268--272.

\end{thebibliography}

\end{document}